\documentclass[onefignum,onetabnum]{siamart251216}

\usepackage{amsopn}

\usepackage{amssymb}

\usepackage{lipsum}
\usepackage{amsfonts}
\usepackage{graphicx}
\usepackage{epstopdf}
\usepackage{algorithmic}
\ifpdf
  \DeclareGraphicsExtensions{.eps,.pdf,.png,.jpg}
\else
  \DeclareGraphicsExtensions{.eps}
\fi

\newsiamremark{remark}{Remark}
\newsiamremark{hypothesis}{Hypothesis}
\newsiamremark{conjecture}{Conjecture}
\crefname{hypothesis}{Hypothesis}{Hypotheses}
\newsiamthm{claim}{Claim}

\headers{Peel Neighborhoods}{S. Huntsman}

\title{Peel Neighborhoods
}

\author{Steve Huntsman\thanks{Cynnovative
  (\email{steve.huntsman@cynnovative.com}).}
}

\makeatletter
\newcommand*{\addFileDependency}[1]{
  \typeout{(#1)}
  \@addtofilelist{#1}
  \IfFileExists{#1}{}{\typeout{No file #1.}}
}
\makeatother

\ifpdf
\hypersetup{
  pdftitle={Peel Neighborhoods},
  pdfauthor={S. Huntsman}
}
\fi

\begin{document}

\maketitle

\begin{abstract}
We introduce the canonical, parameter-free, and efficiently computable notion of \emph{peel neighborhoods} in a finite metric space of strict negative type. Using a soft threshold to upper bound their radius or cardinality allows peel neighborhoods to be computed at scale, enabling useful microscopic descriptions of geometry and topology. As an example of their utility, peel neighborhoods enable efficient and performant estimates of local dimension and detections of singularities in samples from stratified manifolds.
\end{abstract}



\section{Introduction}\label{sec:Introduction}

The typical ways of encoding ``hard'' locality (versus ``soft'' locality exemplified by, e.g. kernels or topological persistence) in point clouds are with balls of fixed radius or with equal numbers of neighbors around points \cite{jones2024manifold}. In this paper, we introduce an attractive alternative called \emph{peel neighborhoods}. The key underlying difference between peel neighborhoods and other constructions is that peel neighborhoods pay deference to a robust and dimension-agnostic notion of convexity. Without thresholding of a sort detailed below, a nondegenerate peel neighborhood (i.e., one which does not equal the entire space) is defined by a circumscribed sphere in an embedding space \cite{asao2025geometric,devriendt2025geometry}. Degenerate peel neighborhoods arise because of a convexity deficit that indicates a boundary. 

Importantly, while exact convex hull algorithms exhibit exponential scaling in complexity as a function of dimension \cite{balestriero2022deephull} and/or precision \cite{zhuang2025polynomial}, the calculation of peel neighborhoods is generally much more efficient, and a very simple heuristic gives a provably good approximation \cite{huntsman2025peeling}. By Theorem \ref{thm:peeling} below, the complexity of the heuristic and exact calculations is independent of dimension, except insofar as dimension might inform the calculation of distances; although exact peeling has only weakly polynomial complexity, in practice precision is not an issue either. Since the complexity of pairwise distance calculations typically scales linearly with data dimension and amortized approximate nearest neighbor search scales polylogarithmically with data cardinality \cite{malkov2018efficient}, the amortized computation of each thresholded peel neighborhood typically scales linearly with data dimension and polylogarithmically with data cardinality. We have observed this in practice up to scales of hundreds of thousands of points and thousands of dimensions.

This scalability is unusual for constructions associated with the theory of \emph{magnitude} \cite{leinster2013magnitude,leinster2021entropy}. This theory synthesizes enriched category theory and information theory in a general notion of ``effective size'' that computationally boils down to dense linear algebra in the specific context of metric geometry. This typically presents a severe scaling bottleneck that peel neighborhoods are able to avoid. Meanwhile, the simplicity of peel neighborhoods complements their efficiency. The \emph{peel} of a finite metric space of negative type has an elementary definition, and a simple algorithm approximates it efficiently and very well at the small scales that peel neighborhoods use \cite{huntsman2025peeling}. An augmentation with anytime convergence to the exact result is not much more troublesome to deploy, losing only a theoretical runtime guarantee but terminating quickly in practice. 

The efficiency, simplicity, theoretical depth, and performance (both exact and heuristic) of peel neighborhoods suggest that they can be useful for a multitude of ``microscopic'' applications in geometry and topology. Here, we focus on estimating local dimension and identifying singularities in samples from stratified manifolds.

The main contributions of this paper are the definition of peel neighborhoods \emph{per se} in \S \ref{sec:peelNeighborhoods} along with experiments in \S \ref{sec:experiments} that characterize peel neighborhoods and that demonstrate their utility. Of note is \S \ref{sec:gradients}, in which we introduce a parameter-free and scalable singularity detection score for stratified manifolds with performance that is competitive with or better than the state of the art, while being faster to compute.

The paper is organized as follows. \S \ref{sec:Preliminaries} covers preliminaries on weightings and diversity that draw on the theory of magnitude. \S \ref{sec:peelingTheorem} reviews the peeling theorem and Algorithm \ref{alg:ApproximatePeel} from \cite{huntsman2025peeling}, gives examples of peels, and also gives some useful results about peels. \S \ref{sec:peelNeighborhoods} introduces peel neighborhoods and gives some basic examples as well as detailing a practical thresholding technique that enhances their utility and accelerates their computation (this can be largely dispensed with, but at the cost of intricacy). \S \ref{sec:experiments} discusses a series of experiments involving peel neighborhoods. \S \ref{sec:comparison} compares them to common alternatives. \S \ref{sec:annuli} indicates how peel neighborhoods accurately reflect topology and can therefore serve as an alternative to topological persistence. \S \ref{sec:scaling} discusses how peel neighborhoods can approximate peels themselves. \S \ref{sec:localDimension} illustrates how peel neighborhoods enable good local dimension estimates in sampled stratified manifolds. \S \ref{sec:gradients} shows how estimated gradient norms of these dimension estimates can identify singularities in sampled stratified manifolds. \S \ref{sec:VGT} compares the method of \S \ref{sec:localDimension} to the volume growth transform of \cite{robinson2024structure}. The paper concludes in \S \ref{sec:conclusions}. 

Appendices extend Proposition \ref{prop:inclusionBound2}; detail another experiment along the lines of \S \ref{sec:comparison}; provide additional statistical detail for the experiment of \S \ref{sec:annuli}; detail experiments for $m = 2$ along the lines of the $m = 10$ experiments in \S \ref{sec:gradients}; detail sampling from a hyperbolic surface for \S \ref{sec:VGT}; and finally use MNIST as an example of how peel neighborhoods can be computed at scale.

\section{Weightings and diversity}\label{sec:Preliminaries}

A matrix $d$ of distances for a finite metric space is \emph{negative type} if $x^T d x \le 0$ for $1^T x = 0$ and $x^T x = 1$. If the inequality is strict, then $d$ is \emph{strict negative type}. Examples of strict negative type metrics on finite spaces include finite subsets of Euclidean space endowed with the $L^p$ distance for $1 \le p \le 2$; finite subsets of spheres without antipodal points endowed with the geodesic (cosine) distance, hyperbolic space, and ultrametrics on finite spaces \cite{hjorth1998finite}. The $L^p$ product of finite strict negative type metrics is also strict negative type iff $p > 1$ \cite{huntsman2025peeling}. In practice, for data embedded in some Euclidean space, the vast majority of nondegenerate dissimilarities in applications are likely to be \emph{locally} strict negative type in the sense that the underlying geometry is still locally close enough to Euclidean. This situation is good enough for our purposes, as in the case of the Bolza surface as used in \S \ref{sec:VGT} and described in Appendix \S \ref{sec:bolza}. 

By Schoenberg's theorem \cite{schoenberg1937certain}, if $d$ is strict negative type, then $$Z_{jk} = \exp(-td_{jk})$$ is positive definite for $t \in (0,\infty)$. The matrix $Z$ just above is a special case of a \emph{similarity matrix}: i.e., a square nonnegative matrix with positive diagonal. For a generic similarity matrix, a \emph{weighting} $w$ is a solution to $$Zw = 1,$$ where $1$ indicates a vector of all ones. The sum of the weighting vector is well-defined, even if the weighting is not unique: this sum is called the \emph{magnitude} of $Z$ in the literature. If $d$ is strict negative type, and $t > 0$ is sufficiently small, the corresponding unique weighting is concentrated on (a subset of) the ``boundary'' of the underlying space \cite{willerton2009heuristic,bunch2020practical,huntsman2025peeling}. This behavior is related to the notion of Bessel capacities \cite{meckes2015magnitude}.

Meanwhile, for a probability distribution $p$ and generic similarity matrix $Z$, the \emph{diversity of order $q$} is \begin{equation}
\label{eq:diversity}
D_q^{Z}(p) := \exp \left ( \frac{1}{1-q} \log \sum_{j: p_j > 0} p_j (Zp)_j^{q-1} \right )
\end{equation} 
for $1 < q < \infty$, and via limits for $q = 1,\infty$ \cite{leinster2016maximizing,leinster2021entropy}. Note that if $Z = I$, then the logarithm of diversity is the R\'enyi entropy; if furthermore $q = 1$, then the logarithm of diversity is the Shannon entropy. 

A remarkable result of Leinster and Meckes \cite{leinster2016maximizing} on the $q$-independence of diversity-maximizing distributions has the following special case that connects weightings with a nice tradeoff between concreteness and generality:
\begin{theorem}
\label{thm:maxDiversity}
Let $Z_{jk} = \exp(-t d_{jk})$ for $t > 0$ and $d$ strict negative type. If the weighting $w$ that solves $Zw = 1$ is positive, then $w$ is proportional to the diversity-maximizing distribution for all $q$.
\end{theorem}
This reduces diversity maximization to a standard linear algebra problem while also rendering the parameter $q$ irrelevant. The sequel addresses the limit $t \downarrow 0$.

\section{The peeling theorem}\label{sec:peelingTheorem}


The diversity of order $1$ of a probability distribution $p$ on $[n]$ endowed with a similarity matrix $Z$ is 
\begin{equation}
\label{eq:diversity1}
D_1^{Z}(p) := \prod_{j:p_j > 0} (Zp)_j^{-p_j}.
\end{equation}
The corresponding generalized entropy is $\log D_1^{Z}(p)$. When $d$ is strict negative type, it is possible to efficiently optimize these quantities for $Z = \exp[-td]$ in the limit $t \downarrow 0$. In this limit, the first-order approximation $Z = \exp[-td] \approx 11^T - td$ yields
\begin{equation}
\label{eq:ssEntropy1Approx}
\log D_1^{Z}(p) \approx t p^T d p.
\end{equation}
The quantity $p^T d p$ appearing on the right side of \eqref{eq:ssEntropy1Approx} is called the \emph{quadratic entropy} of $d$. If $d$ is strict negative type, then the quadratic entropy is convex, and it can be efficiently optimized. We recall a particularly convenient heuristic that is the first phase of a somewhat more involved algorithm for the exact optimum.

\begin{theorem}[peeling theorem; informal \cite{huntsman2025peeling}; see \S G of \cite{huntsman2025peeling} for details\footnote{A previous version of this paper on arXiv carried over an incorrect version of this theorem from \cite{huntsman2025peeling}, which has been corrected and extensively elaborated in that paper's \S G. Along with using the correct statement here, we ran every example using \emph{both} the heuristic approximation of Algorithm \ref{alg:ApproximatePeel} and exact peel neighborhoods, reporting discrepancies throughout. Such discrepancies turn out to be minor and functionally insignificant in the present paper when they occur at all.}]
\label{thm:peeling}
Let $\Delta_{N-1} := \{ p \in [0,1]^N : 1^T p = 1 \}$. For $d$ strict negative type and metric, Algorithm \ref{alg:ApproximatePeel} returns an approximation to
\begin{align}
\label{eq:argMaxQE}
p_*(d) := & \ \arg \max_{p \in \Delta_{N-1}} p^T d p \nonumber \\
\overset{\forall q}{=} & \ \arg \max_{p \in \Delta_{N-1}} \lim_{t \downarrow 0} D_q^Z(p)
\end{align}
with strong computable \emph{a priori} error bounds (and in practice, very good if not exact results) in time $O(N^{\omega+1})$, where $\omega \le 3$ is the matrix multiplication complexity exponent. Augmenting Algorithm \ref{alg:ApproximatePeel} with a subsequent efficient (and in practice, very brief) ascent phase produces the exact result in finitely many steps while also providing an anytime convergence guarantee.
\end{theorem}

\begin{algorithm}
  \caption{\textsc{ApproximatePeel}$(d)$}
  \label{alg:ApproximatePeel}
\begin{algorithmic}[1]
  \REQUIRE Strict negative type metric $d$ on $[N] \equiv \{1, \dots, N\}$
  \STATE $p \leftarrow \frac{d^{-1} 1}{1^T d^{-1} 1}$
  \WHILE{$\exists i : p_i < 0$}
    \STATE $\mathcal{J} \leftarrow \{ j : p_j > 0 \}$ \hfill \emph{// Restriction of support}
    \STATE $p \leftarrow 0_{[N]}$
    \STATE $p_\mathcal{J} \leftarrow \frac{d|_\mathcal{J}^{-1} 1}{1^T d|_\mathcal{J}^{-1} 1}$ 
  \ENDWHILE
  \ENSURE $p \approx p_*(d)$
\end{algorithmic}
\end{algorithm}

\begin{figure}[htbp]
  \centering
  \includegraphics[trim = 20mm 15mm 15mm 10mm, clip, width=.9\textwidth,keepaspectratio]{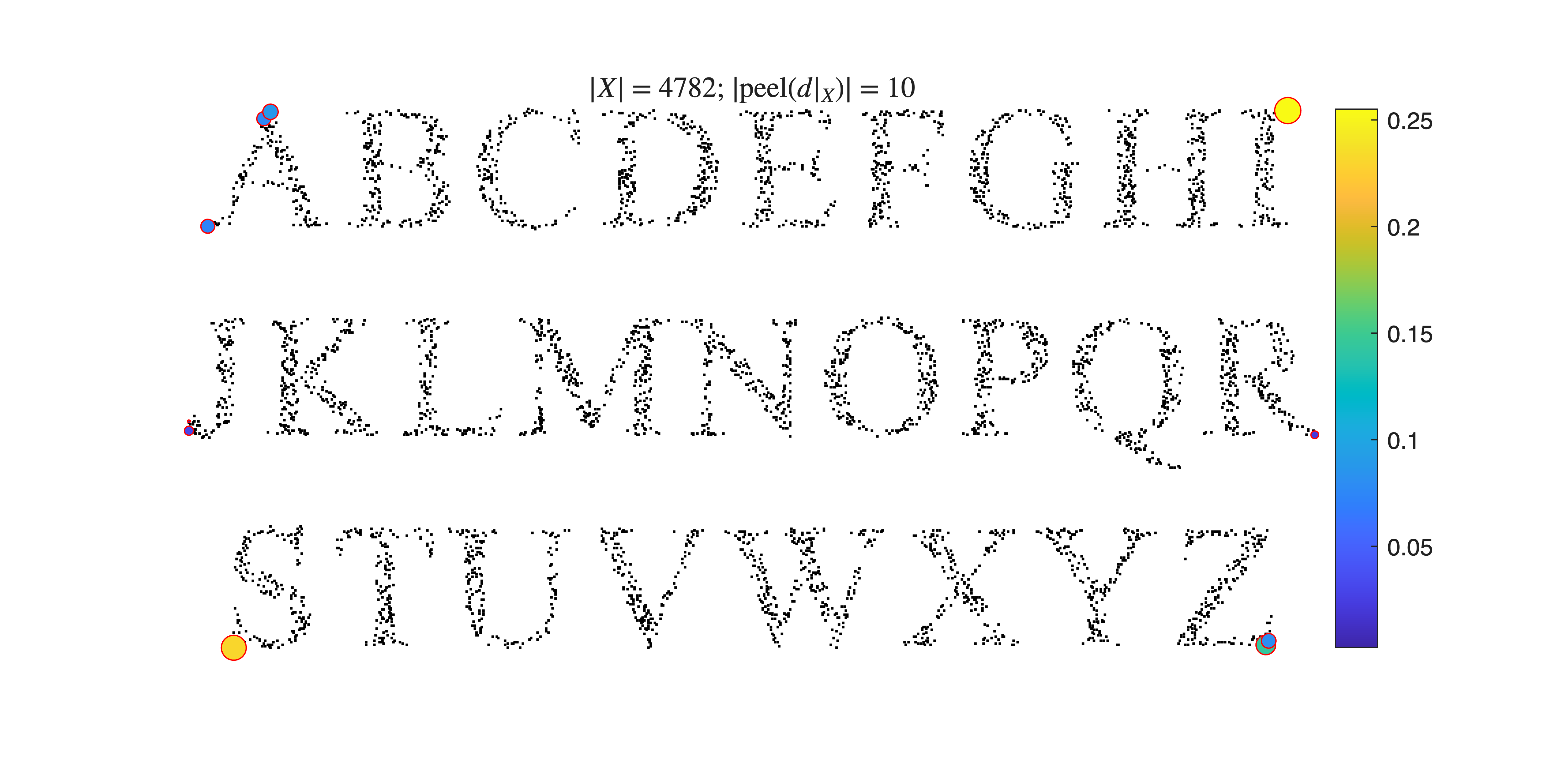}
\caption{The peel distribution, shown with overlaid colored and variably-sized disks {\color{red}circled in red}, of a sample of 4782 IID uniform points from a complicated region in $\mathbb{R}^2$, which are shown in black. The support of the peel distribution is just 10 points located at extremities. The peel and its heuristic approximation (from Algorithm \ref{alg:ApproximatePeel}) are exactly equal here.
} 
  \label{fig:letters_peel}
\end{figure}

Whenever $d$ is a strict negative type metric, we call $p_*(d)$ (or, as context warrants, its support) the \emph{peel} of $d$; we call the output of Algorithm \ref{alg:ApproximatePeel} the \emph{heuristic approximation} to the peel. Figures \ref{fig:letters_peel} and \ref{fig:MNIST_peel_layers} give a sense of why the peeling nomenclature is appropriate. We distinguish the peel distribution and its support using the notation $\textnormal{peel}(d) := \textnormal{supp}(p_*(d))$. 
Note that the peel of a one-point space is just the point mass. 

Figure \ref{fig:MNIST_peel_layers} shows the results of repeatedly peeling each digit class in the MNIST data set using $L^2$ distance. The $\ell$-fold peel support of a set minus previous peel supports, written $\textnormal{peel}^{(\ell)}$, tends from a set of outliers to a generalized medoid \cite{huntsman2025peeling}.

\begin{figure}[htbp]
  \centering
  \includegraphics[trim = 0mm 0mm 0mm 0mm, clip, width=.45\textwidth,keepaspectratio]{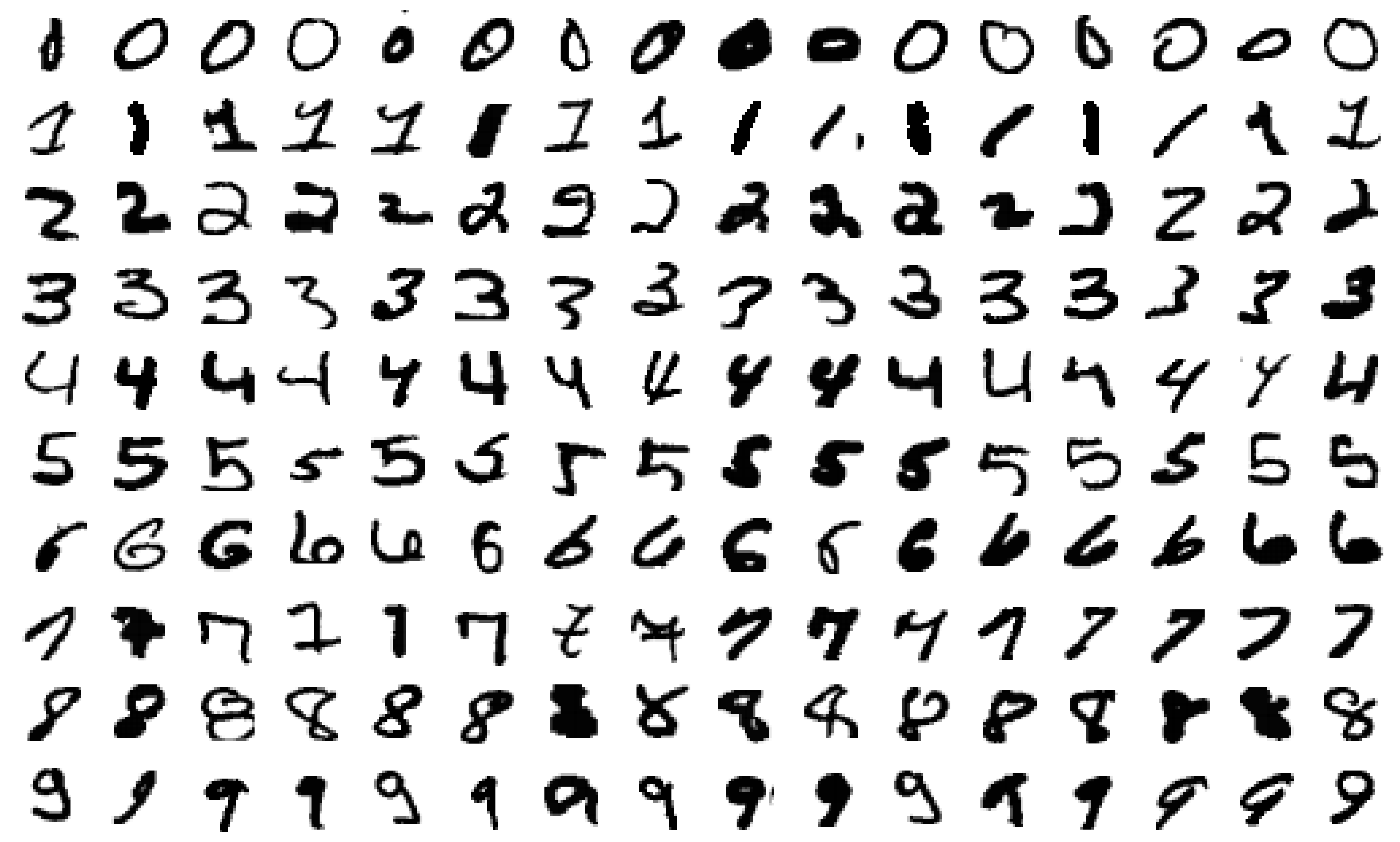}
  \includegraphics[trim = 0mm 0mm 0mm 0mm, clip, width=.45\textwidth,keepaspectratio]{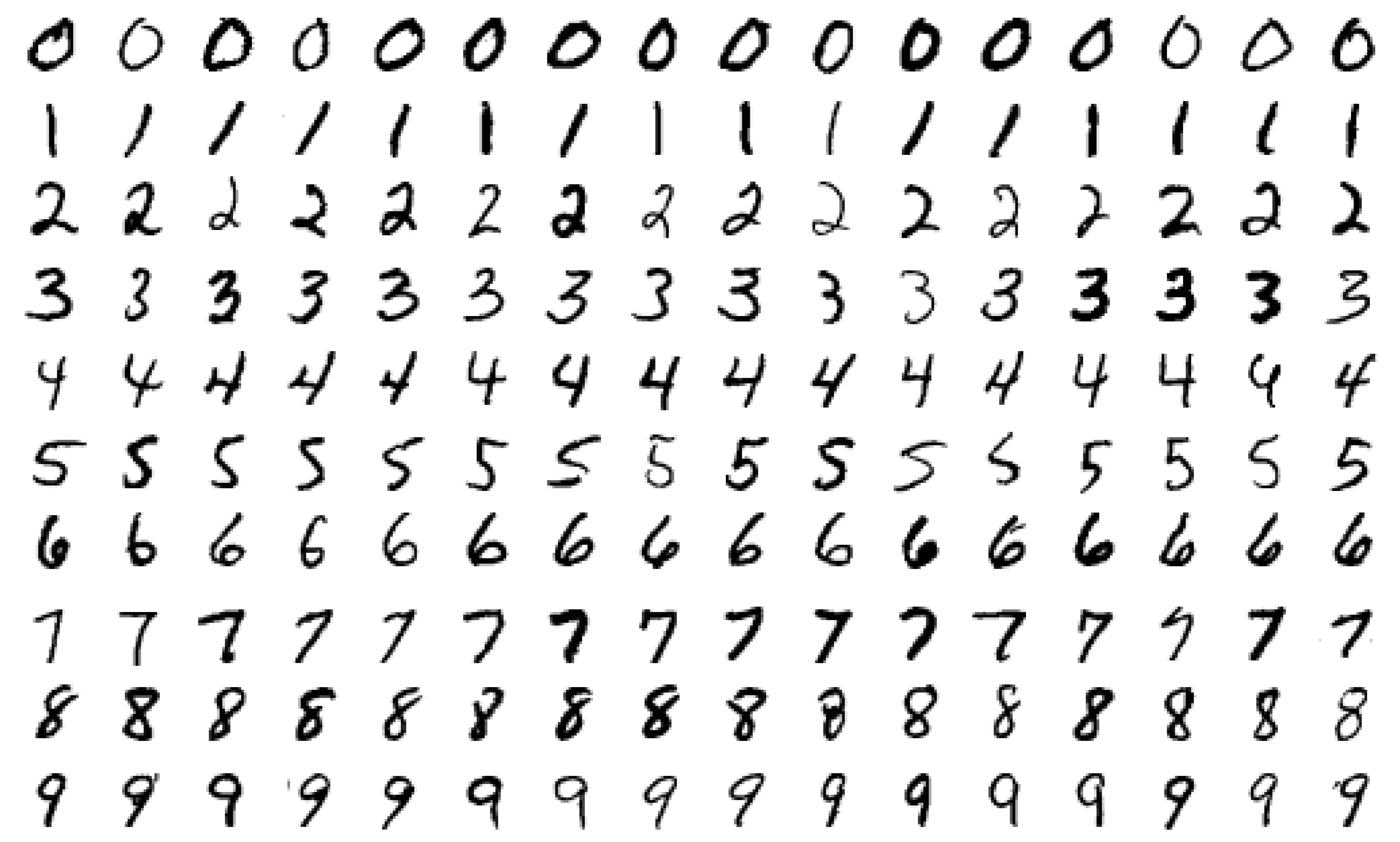}
\caption{Left: representative elements of $\textnormal{peel}^{(1)}(d|_{\textnormal{digit}})$ for each digit label in MNIST. 
Right: representative elements of $\textnormal{peel}^{(31)}(d|_{\textnormal{digit}})$. The digits $0 \dots 9$ respectively have 42, 59, 37, 41, 41, 37, 42, 48, 40, and 43 total peel layers. Across digits, $\textnormal{peel}^{(1)}(d|_{\textnormal{digit}})$ and its heuristic approximation by Algorithm \ref{alg:ApproximatePeel} respectively support at least 96.8\% and 99.47\% of their corresponding distributions on their intersection.
} 
  \label{fig:MNIST_peel_layers}
\end{figure}

\subsection{Some useful results}\label{sec:someUsefulResults}

\begin{lemma}\label{lem:1eigenvector}
    If $d$ is a strict negative type metric on $[N]$ for $N \ge 2$ and $1$ is an eigenvector of $d$, then $d^{-1}$ is also strict negative type.
    \footnote{Such matrices are called \emph{regular Euclidean distance matrices} in \S 4.1.1 of \cite{alfakih2018euclidean}.}
\end{lemma}

\begin{proof}
    Let the orthonormal eigenvectors of $d$ be $N^{-1/2} \cdot 1,v_1,\dots,v_{N-1}$. The eigenvalue corresponding to $N^{-1/2} \cdot 1$ is $\lambda_0 := N^{-1} 1^T d 1$: let the other eigenvalues be $\lambda_1,\dots,\lambda_{N-1}$. By Lemma 3.6 of \cite{hjorth1998finite}, $\lambda_j < 0$ for $j \in [N-1]$. Now let $x \ne 0$ be such that $1^T x = 0$: then we can write $x = \sum_{j=1}^{N-1} \alpha_j v_j$, and $x^T d^{-1} x = \sum_{j=1}^{N-1} \alpha_j^2 \lambda_j^{-1} < 0$, so $d^{-1}$ is strict negative type.
\end{proof}

\begin{lemma}\label{lem:scale0magnitude}
    If $d$ is a strict negative type metric on $[N]$ for $N \ge 2$, then $1^T d^{-1} 1 > 0$. (In other words, the magnitude of $d$ is positive at scale 0).
\end{lemma}

\begin{proof}
    If 1 is an eigenvector of $d$ and $d^{-1}$, then we are done by Lemma \ref{lem:1eigenvector}. Otherwise, let $\mu := N^{-1} 1^T d^{-1} 1$ and $u := d^{-1}1-\mu 1$. Without loss of generality, $1$ is not an eigenvector of $d^{-1}$, so $u \ne 0$ and $1^T u = 0$. Since $d$ is strict negative type, $u^T d u < 0$. Expanding this gives $1^T d^{-1} 1 - 2\mu 1^T 1 + \mu^2 1^T d 1 < 0.$ Since $\mu 1^T 1 = 1^T d^{-1} 1$, the preceding inequality simplifies to $1^T d^{-1} 1 > \mu^{-2} 1^T d 1.$ It only remains to show that $\mu \ne 0$. To show this, suppose that $\mu = 0$ and let $v := d^{-1}1$, so that $v^T d v = 1^T d^{-1} 1 = \mu = 0$. At the same time, $v^T d v = 1^T v = 0$. Since $v \ne 0$ and $d$ is strict negative type, we must have that $v^T d v < 0$, which is a contradiction.
\end{proof}

\begin{proposition}\label{prop:peel2Points}
    If $d$ is a strict negative type metric on $[N]$ for $N \ge 2$, then $\textnormal{peel}(d)$ contains at least two points. 
\end{proposition}

\begin{proof}
    Suppose that $\textnormal{peel}(d) = \{j\}$ for $j \le N$. Then $p_*(d)$ is the point mass $e_{(j)}$, and $p_*^T d p_* = 0$. Let $j \ne k \le N$ and $q := (1-\varepsilon) e_{(j)} + \varepsilon e_{(k)}$. Now $q^T d q = 2\varepsilon (1-\varepsilon) d_{jk} > 0$ for $\varepsilon > 0$ sufficiently small. Then $q^T d q > p_*^T d p_*$, a contradiction.
\end{proof}

\begin{proposition}\label{prop:peelOfPeel}
    If $d$ is strict negative type on $[N]$ and $\textnormal{peel}(d) = \mathcal{J} \subseteq [N]$, then $\textnormal{peel}(d|_\mathcal{J}) = \mathcal{J}$.
\end{proposition}

\begin{proof}
    The restriction of $p_*(d)$ to $\mathcal{J}$ is $(d|_\mathcal{J}^{-1} 1)/(1^T d|_\mathcal{J}^{-1} 1)$. Because this is positive, Algorithm \ref{alg:ApproximatePeel} immediately terminates on $d|_\mathcal{J}$.
\end{proof}

The following proposition characterizing peel distributions is a restatement of Proposition 5.20 of \cite{devriendt2022graph} phrased in the language of this paper and of \cite{huntsman2025peeling}.
\begin{proposition}\label{prop:dp}
    If $d$ is strict negative type on $[N]$ then $(dp_*(d))_j \ge (dp_*(d))_k$ for all $j \in \textnormal{peel}(d)$ and $k \in [N]$. In particular, $dp_*(d)$ is constant on $\textnormal{peel}(d)$.
\end{proposition}


\begin{theorem}\label{thm:peelConvex}
    If $d$ is the distance matrix of $X = \{x_1,\dots,x_N\} \subset \mathbb{R}^m$, then $\textnormal{peel}(d)$ is a subset of the vertices of the convex hull of $X$.
\end{theorem}

\begin{proof}
    Let $V(X)$ be the set of vertices of the convex hull of $X$. Suppose for contradiction that there exists $x_j \in \textnormal{peel}(d) - V(X)$. Since $x_j \notin V(X)$, we can write it as a nontrivial convex combination of points in $V(X)$, say $x_j = \sum_{\ell = 1}^L \alpha_\ell x_{j_\ell}$ with $x_{j_\ell} \in V(X)$, $\alpha_\ell > 0$, and $j_\ell \ne j$ for all $\ell \in [L]$, and with $\sum_{\ell = 1}^L \alpha_\ell = 1$: finally, nontriviality here means that $L > 1$.


    By Proposition \ref{prop:dp} and writing $p_*(d) \equiv p_*$ here, $c := (dp_*)_j$ is a fixed constant since $x_j \in \textnormal{peel}(d)$. We therefore have that
    \begin{align}
    c & = \sum_{k \in [N]} d_{jk} p_{*k} \nonumber \\
    & = \sum_k \|x_j - x_k \| p_{*k} \nonumber \\
    & = \sum_k \left \|\sum_\ell \alpha_\ell x_{j_\ell} - x_k \right \| p_{*k} \nonumber \\
    & = \sum_k \left \|\sum_\ell \alpha_\ell (x_{j_\ell} - x_k) \right \| p_{*k} \nonumber \\
    & = \left \| \sum_\ell \alpha_\ell (x_{j_\ell} - x_j) \right \| p_{*j} + \sum_{k \ne j} \left \|\sum_\ell \alpha_\ell (x_{j_\ell} - x_k) \right \| p_{*k}. \label{eq:peelConvexConstant}
    \end{align}
    Now since 
    $$0 = \left \| \sum_\ell \alpha_\ell (x_{j_\ell} - x_j) \right \| p_{*j} < \sum_\ell \alpha_\ell \left \| x_{j_\ell} - x_j \right \| p_{*j},$$
    we have by \eqref{eq:peelConvexConstant} the strict inequality
    \begin{equation}\label{eq:peelConvexStrict}
        c < \sum_k \sum_\ell \alpha_\ell\left \| x_{j_\ell} - x_k \right \| p_{*k}.
    \end{equation}
    Meanwhile,
    \begin{align}
    \sum_k \sum_\ell \alpha_\ell \left \| x_{j_\ell} - x_k \right \| p_{*k} & = \sum_\ell \alpha_\ell \sum_k \left \| x_{j_\ell} - x_k \right \| p_{*k} \nonumber \\
    & = \sum_\ell \alpha_\ell \sum_k d_{{j_\ell}k} p_{*k} \nonumber \\
    & = \sum_\ell \alpha_\ell (dp_*)_{j_\ell} \nonumber \\
    & \le \sum_\ell \alpha_\ell c \nonumber \\
    & = c, \label{eq:peelConvexContradiction}
    \end{align}
    where the inequality is by Proposition \ref{prop:dp}. Combining \eqref{eq:peelConvexStrict} and \eqref{eq:peelConvexContradiction} gives $c<c$, a contradiction. It must therefore be the case that $\textnormal{peel}(d) - V(X) = \varnothing$.
\end{proof}

For example, the peel in Figure \ref{fig:letters_peel} consists of 10 of the 20 convex hull vertices. 

The following results show that peel distributions are continuous in $d$, and that peel supports are generically stable. 

\begin{theorem}\label{thm:peelContinuity}
    Let $d$ and $d'$ be strict negative type metrics on $[N]$, and write $$\gamma(d) := -\max_{\|x\| = 1; \ 1^T x = 0} x^T d x > 0.$$ 
    Then 
    \begin{equation}\label{eq:continuity}
        \left \| p_*(d') - p_*(d) \right \|_2 \le \frac{2}{\max \{ \gamma(d), \gamma(d') \}} \left \| d' - d \right \|_\textnormal{op}
    \end{equation}
    where $\| \cdot \|_\textnormal{op}$ indicates the operator norm.
\end{theorem}

\begin{proof}
    For concision, temporarily write $p = p_*(d)$ and $p' = p_*(d')$. Also, without loss of generality, assume that $d \ne d'$. Note that 
    \begin{equation}\label{eq:continuityEq1}
        p'^T d p' = p^T d p + 2(p'-p)^T d p + (p'-p)^T d (p'-p). 
    \end{equation}

    We claim that for all $q \in \Delta_{N-1}$, $(q-p)^T dp \le 0$. To see this, write 
    \begin{align}
        h(p,q;t) := & \ (p+t[q-p])^T d (p+t[q-p]). \nonumber \\
        = & \ p^T d p + 2t(q-p)^T dp + t^2(q-p)^T d (q-p). \nonumber
    \end{align} 
    Since $p$ maximizes $p^T d p$ over $\Delta_{N-1}$ and $p+t[q-p] \in \Delta_{N-1}$ for $t \in [0,1]$, we have that $h(p,q;0) \ge h(p,q;t)$ for $t \in [0,1]$. In particular, the derivative of $h$ is nonpositive at $t = 0^+$, which establishes the claim. 

    Applying this to \eqref{eq:continuityEq1} and subsequently noting that $1^T(p'-p) = 0$ yields the successive inequalities
    \begin{equation}\label{eq:continuityEq2}
        p'^T d p' \le \ p^T d p + (p'-p)^T d (p'-p) \le \ p^T d p - \gamma(d) \cdot \left \| p' - p \right \|_2^2.
    \end{equation}
    Since $p'^T d' p' \ge p^T d' p'$, we obtain
    \begin{equation}\label{eq:continuityEq3}
        \gamma(d) \cdot \left \| p' - p \right \|_2^2 \le \ p^T d p - p'^T d p'.
    \end{equation}

    Writing $E = d'-d$, we have that 
    \begin{equation}\label{eq:continuityEq4}
        p'^T d p' =  \ p'^T d' p' - p'^T E p' \ge \ p^T d' p - p'^T E p' = \ p^T d p + p^T E p - p'^T E p'.
    \end{equation}
    Using \eqref{eq:continuityEq4} in \eqref{eq:continuityEq3} now yields
    \begin{align}\label{eq:continuityEq5}
        \gamma(d) \cdot \left \| p' - p \right \|_2^2 \le & \ p^T d p - \left ( p^T d p + p^T E p - p'^T E p' \right ) \nonumber \\
        = & \ p'^T E p' - p^T E p \nonumber \\
        = & \ (p'-p)^T E p' + p^T E (p'-p).
    \end{align}

    By Cauchy-Schwarz, $|(p'-p)^T E p'| \le \| p'-p \|_2 \cdot \| Ep' \|_2$ and $|p^T E (p'-p)| \le \| Ep \|_2 \cdot \| p'-p \|_2$, so 
    \begin{equation}\label{eq:continuityEq6}
        \gamma(d) \cdot \left \| p' - p \right \|_2^2 \le \| p'-p \|_2 \cdot \| Ep' \|_2 + \| Ep \|_2 \cdot \| p'-p \|_2.
    \end{equation}
    Dividing both sides of \eqref{eq:continuityEq6} by $\| p'-p \|_2$ (which we can assume is positive) and applying $\| E x \|_2 \le \| E \|_\textnormal{op} \cdot \| x \|_2$ yields
    \begin{equation}\label{eq:continuityEq7}
        \gamma(d) \cdot \left \| p' - p \right \|_2 \le \| E \|_\textnormal{op} \cdot \left ( \| p' \|_2 + \| p \|_2 \right ).
    \end{equation}
    Finally, $\| p' \|_2 + \| p \|_2 \le 2$, so
    \begin{equation}\label{eq:continuityEq8}
        \gamma(d) \cdot \left \| p' - p \right \|_2 \le 2\| d' - d \|_\textnormal{op}.
    \end{equation}
    The theorem now follows by symmetry.
\end{proof}

\begin{corollary}\label{cor:peelStability}
    Let $d$ be strict negative type on $[N]$. If for all $j \in \textnormal{peel}(d)$ 
    $$(d p_*(d))_j > \max_{k \not \in \textnormal{peel}(d)} (d p_*(d))_k,$$ 
    then there exists $\varepsilon > 0$ such that $\|d' - d\|_\textnormal{op} < \varepsilon \Rightarrow \textnormal{peel}(d') = \textnormal{peel}(d)$.
\end{corollary}

\begin{proof}
    First, note that by Proposition \ref{prop:dp}, the left hand side of the inequality in the hypothesis does not depend on $j$ apart from the requirement $j \in \textnormal{peel}(d)$. By Theorem \ref{thm:peelContinuity}, we have for sufficiently small $\varepsilon$ that $(d' p_*(d'))_j > \max_{k \not \in \textnormal{peel}(d)} (d' p_*(d'))_k$. The result now follows from Proposition \ref{prop:dp}.
\end{proof}

Suppose that $\mathcal{I} \subseteq \mathcal{J} := \textnormal{peel}(d)$. By Theorem \ref{thm:peeling} and Proposition \ref{prop:peelOfPeel}, if $d|_\mathcal{I}^{-1} 1 > 0$, then $\textnormal{peel}(d|_\mathcal{I}) = \mathcal{I}$. With this in mind, consider the following lemma.

\begin{lemma}\label{lem:inclusionBound}
Suppose that $d$ is strict negative type on $[N]$ and $\mathcal{I} = [J-1] \subset \mathcal{J} = [J] = \textnormal{peel}(d)$, writing
$$d|_\mathcal{J} = \begin{pmatrix} d|_\mathcal{I} & \delta \\ \delta^T & 0 \end{pmatrix}.$$ Then there is a positive constant $C \equiv C(d|_\mathcal{I}, \delta)$ such that
$$d|_\mathcal{I}^{-1} 1 > C d|_\mathcal{I}^{-1} \delta.$$ In particular, if $d|_\mathcal{I}^{-1} \delta > 0$, then $d|_\mathcal{I}^{-1} 1 > 0$ and $\textnormal{peel}(d|_\mathcal{I}) = \mathcal{I}$.
\end{lemma}

\begin{proof}
For clarity in the presence of Schur complements, in this proof we temporarily write $d_\mathcal{I} \equiv d|_\mathcal{I}$. The Schur complement is a scalar:
\begin{equation}
\label{eq:SchurComplement}
d_\mathcal{J} / d_\mathcal{I} = -\delta^T d_\mathcal{I}^{-1} \delta.
\end{equation}
Now writing $\Delta := d_\mathcal{I}^{-1} \delta$, we have that
\begin{align}
\label{eq:blockInverse}
d_\mathcal{J}^{-1} = & \ \begin{pmatrix} d_\mathcal{I}^{-1} + \Delta (d_\mathcal{J} / d_\mathcal{I})^{-1} \Delta^T & -\Delta (d_\mathcal{J} / d_\mathcal{I})^{-1} \\ -(d_\mathcal{J} / d_\mathcal{I})^{-1} \Delta^T & (d_\mathcal{J} / d_\mathcal{I})^{-1} \end{pmatrix} \nonumber \\
= & \ \begin{pmatrix} d_\mathcal{I}^{-1} & 0 \\ 0 & 0 \end{pmatrix} + (d_\mathcal{J} / d_\mathcal{I})^{-1} \cdot \begin{pmatrix} \Delta \\ -1 \end{pmatrix} \begin{pmatrix} \Delta^T & -1 \end{pmatrix}.
\end{align}
Applying \eqref{eq:blockInverse} to $1_\mathcal{J}$ and using \eqref{eq:SchurComplement} yields
\begin{align}
\label{eq:dJJ11}
d_\mathcal{J}^{-1} 1 = & \ \begin{pmatrix} d_\mathcal{I}^{-1} 1 \\ 0 \end{pmatrix} + \frac{\Delta^T 1 - 1}{d_\mathcal{J} / d_\mathcal{I}} \cdot \begin{pmatrix} \Delta \\ -1 \end{pmatrix} \nonumber \\
= & \ \begin{pmatrix} d_\mathcal{I}^{-1} 1 \\ 0 \end{pmatrix} + \frac{1 - \delta^T d_\mathcal{I}^{-1} 1}{\delta^T d_\mathcal{I}^{-1} \delta} \cdot \begin{pmatrix} \Delta \\ -1 \end{pmatrix}.
\end{align}

By hypothesis, $d_\mathcal{J}^{-1} 1 > 0$. Writing 
$$C := -\frac{1 - \delta^T d_\mathcal{I}^{-1} 1}{\delta^T d_\mathcal{I}^{-1} \delta},$$
we therefore have by \eqref{eq:dJJ11} that $C > 0$ and 
\begin{equation}
d_\mathcal{I}^{-1} 1 > C \Delta.
\end{equation}
\end{proof}


\begin{proposition}\label{prop:inclusionBound2}
Under the hypotheses of Lemma \ref{lem:inclusionBound}, using its notation and writing also $w^{(0)} := d|_\mathcal{J}^{-1}1$, we have that
$$\textnormal{peel}(d|_\mathcal{I}) = \mathcal{I} \iff w^{(0)}|_\mathcal{I} + w^{(0)}_J \Delta > 0.$$
\end{proposition}

\begin{proof}
    To begin, $w^{(0)}_J = C$, so if we write $w^{(1)} := d_\mathcal{I}^{-1} 1$, then \eqref{eq:dJJ11} yields $w^{(1)} = w^{(0)}|_\mathcal{I} + w^{(0)}_J \Delta$. Applying Theorem \ref{thm:peeling} now yields the result.
\end{proof}

\section{Peel neighborhoods}\label{sec:peelNeighborhoods}

The notion of enclosure that peels provide immediately suggests a notion of locality in finite metric spaces of strict negative type, based on when a basepoint is enclosed by the peel of a ball around it. By growing a ball and checking if the peel still contains the basepoint, we get a computationally tractable ``hard'' neighborhood.

As usual, suppose that $X$ is finite and endowed with a metric of strict negative type. For $x \in X$ define
\begin{equation}\label{eq:peelNeighborhoodRadius}
\rho(x) := \inf \{r > 0 : x \not \in \textnormal{peel}(d|_{X \cap B_r(x)}) \}.
\end{equation}
The \emph{peel neighborhood} $\nu(x)$ is roughly the smallest ball around $x$ whose peel does not contain $x$:
\begin{equation}\label{eq:peelNeighborhood}
\nu(x) := X \cap B_{\rho(x)}(x).
\end{equation}
Figures \ref{fig:rho_ball_2} and \ref{fig:rho_ball_10} show examples of peel neighborhoods, and Figure \ref{fig:rho_ball_histograms} shows their cardinalities. 

Suppose that instead of \eqref{eq:peelNeighborhoodRadius} we considered 
$$\tilde \rho(x) := \inf \{ r > 0: \textnormal{peel}(d|_{X \cap B_r(x)}) = \textnormal{peel}(d|_{[X-\{x\}] \cap B_r(x)}) \}.$$ This is harder to work with and less precise than \eqref{eq:peelNeighborhoodRadius}, i.e., $\rho(x) \le \tilde \rho(x)$.

\begin{figure}[htbp]
  \centering
  \includegraphics[trim = 10mm 10mm 10mm 10mm, clip, width=.32\textwidth,keepaspectratio]{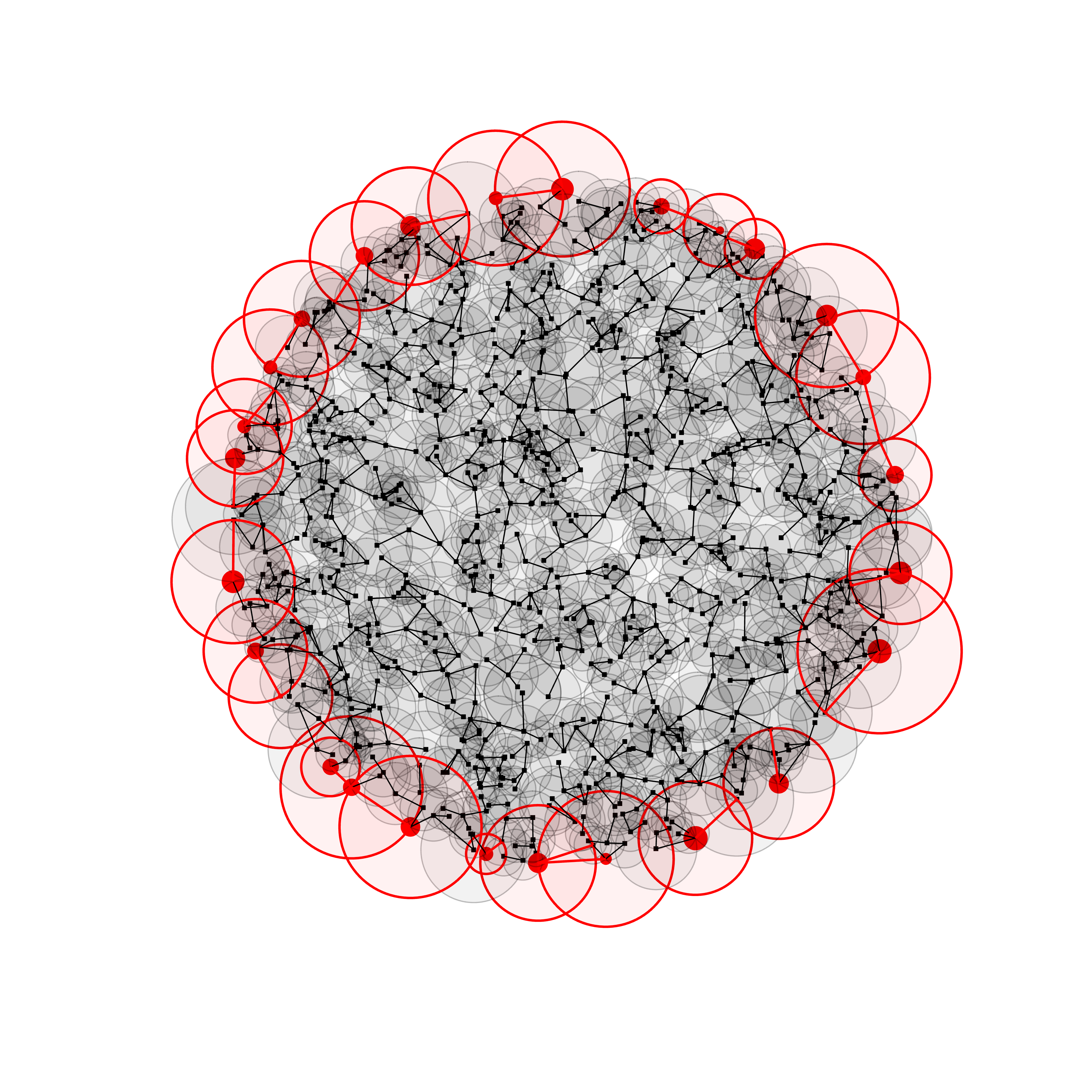}
  \includegraphics[trim = 0mm 0mm 10mm 0mm, clip, width=.60\textwidth,keepaspectratio]{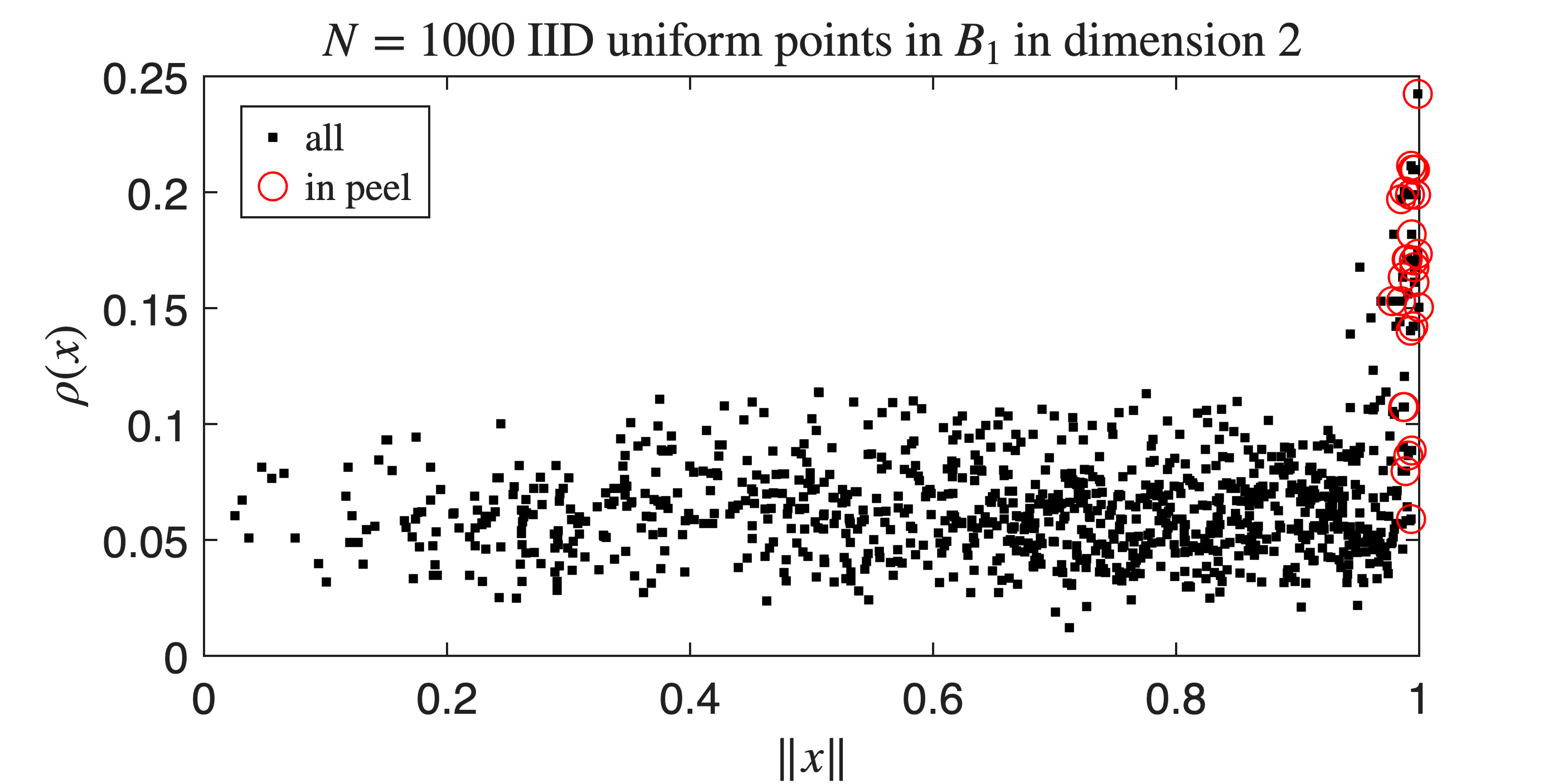}
\caption{Left: unthresholded (for thresholding, see below) peel neighborhoods of 1000 IID uniform points in $B_1 \subset \mathbb{R}^2$, indicated by shaded disks and line segments from basepoints to the point at distance $\rho(x)$. {\color{red}Neighborhoods and segments for points in the overall peel are shown in pink and red}, while others are shown as gray and black. The overall peel and its heuristic approximation are exactly equal here, as are the exact and heuristic approximations of all peel neighborhoods.
Right: $\rho(x)$ is larger for points in the overall peel, which comprise a notion of boundary. Not shown: imposing a practical default radial threshold discussed below affects just 58 of 1000 neighborhoods, near the upper envelope of the plotted points.} 
  \label{fig:rho_ball_2}
\end{figure}

\begin{figure}[htbp]
  \centering
  \includegraphics[trim = 10mm 10mm 10mm 10mm, clip, width=.32\textwidth,keepaspectratio]{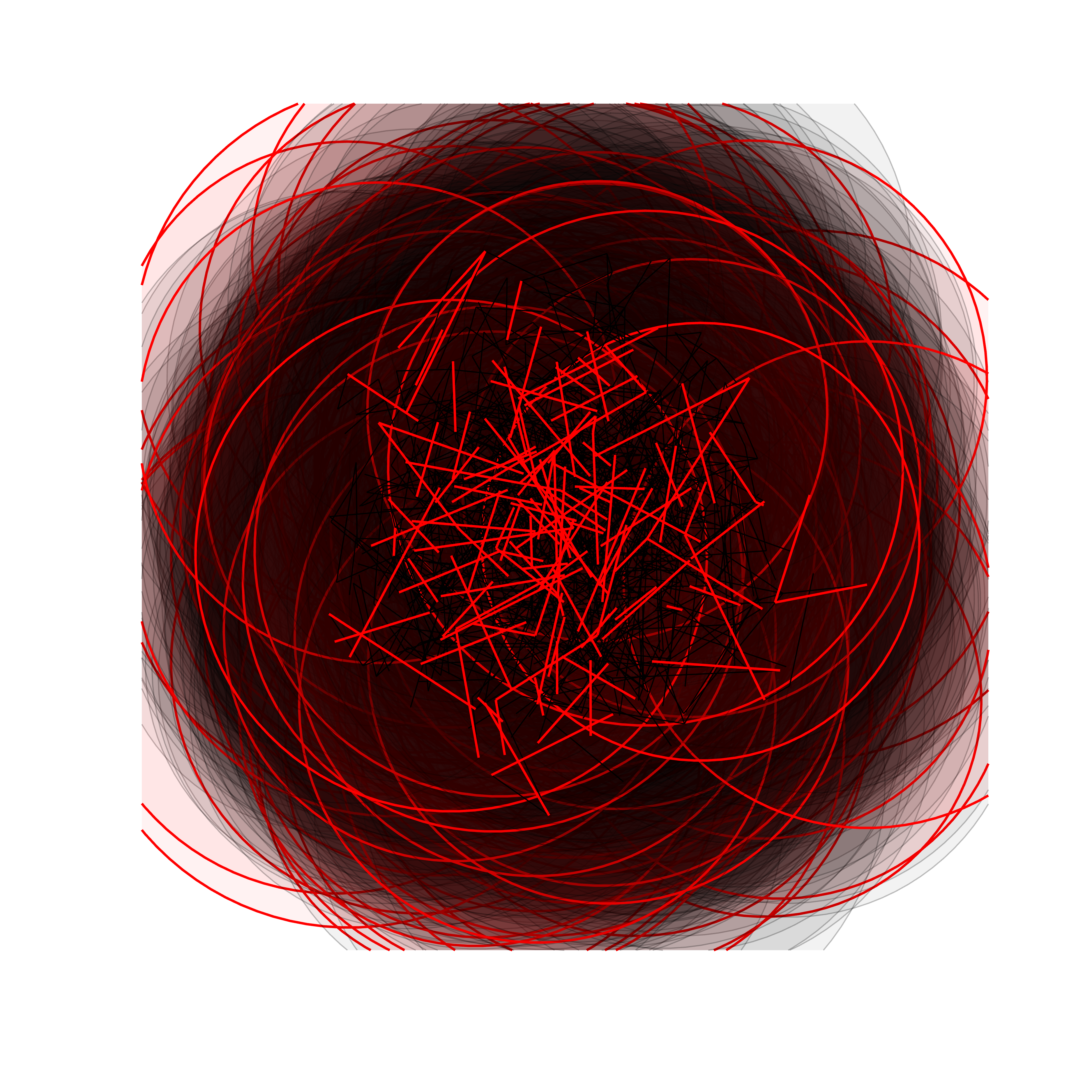}
  \includegraphics[trim = 0mm 0mm 10mm 0mm, clip, width=.60\textwidth,keepaspectratio]{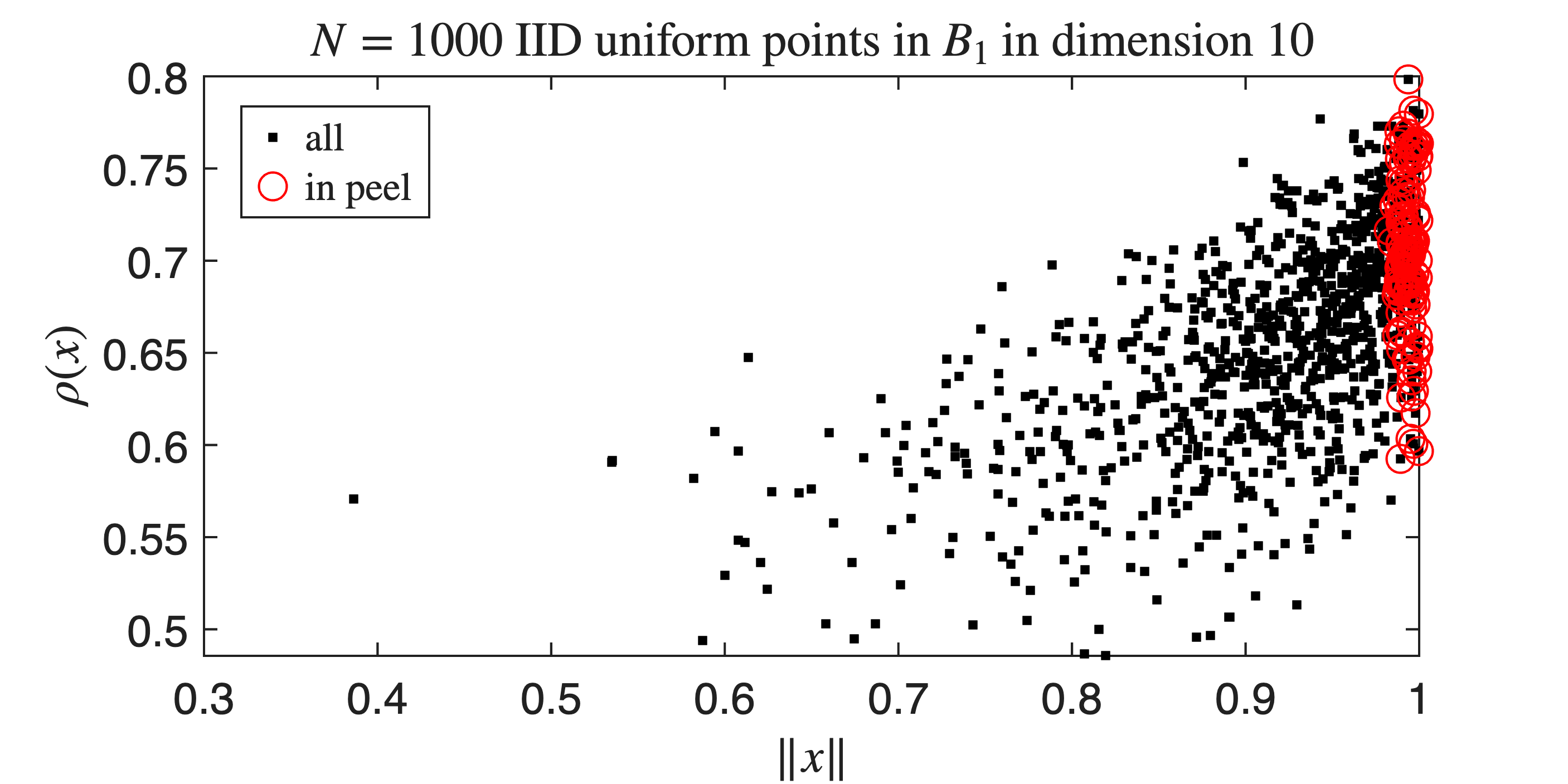}
\caption{As in Figure \ref{fig:rho_ball_2}, but for $B_1 \subset \mathbb{R}^{10}$. 
The overall peel and its heuristic approximation respectively support 99.997\% and 100\% of their corresponding distributions on their intersection; only two peel neighborhoods differ, with heuristic radii $\approx 0.9949$ and $\approx 0.9961$ times the exact values, and one fewer point in each of the two neighborhoods.
Not shown: despite the large radii of the peel neighborhoods, none contains more than 11 points, and 752 of 1000 peel neighborhoods contain six or fewer points. Imposing a practical default radial threshold discussed below affects just 74 of 1000 neighborhoods, near the upper right corner of the plot.} 
  \label{fig:rho_ball_10}
\end{figure}

\begin{figure}[htbp]
  \centering
  \includegraphics[trim = 15mm 50mm 15mm 0mm, clip, width=.9\textwidth,keepaspectratio]{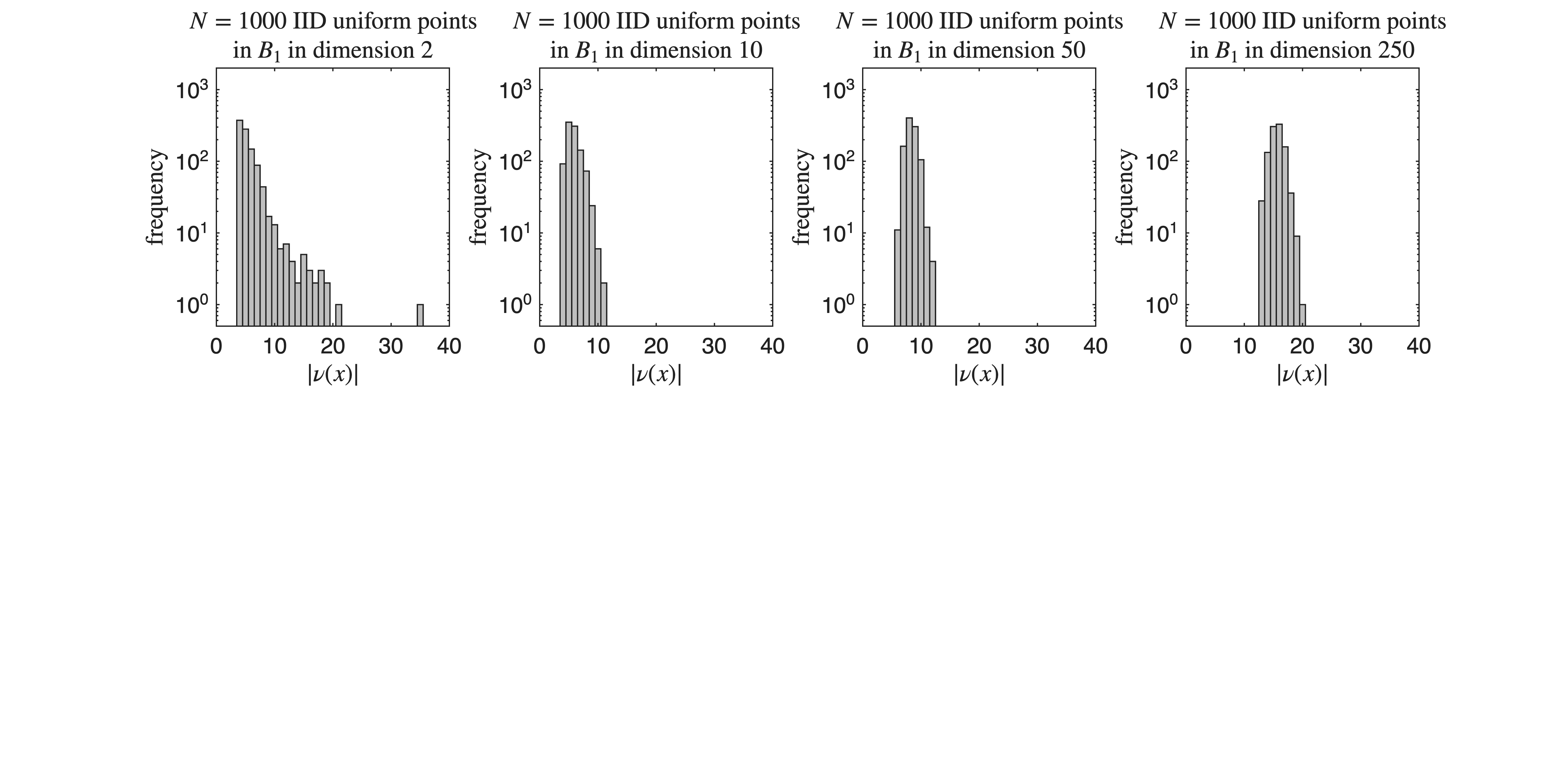}
\caption{Histograms of the cardinalities of unthresholded peel neighborhoods for 1000 IID uniform points in $B_1 \subset \mathbb{R}^m$ for $m \in \{2,10,50,250\}$. 
The exact neighborhoods and their heuristic approximations are identical except for the case $m = 10$, where two of the 1000 heuristic approximations are inexact, with radii $\approx 0.9949$ and $\approx 0.9961$ times the exact values, and one fewer point in each of these two peel neighborhood approximations (cf. Figure \ref{fig:rho_ball_10}, which used the same pseudorandom number generator seed to generate identical data).
} 
  \label{fig:rho_ball_histograms}
\end{figure}

Note that it may be the case that $\rho(x) = \infty$, so that $\nu(x) = X$. It is tempting to speculate that $\rho(x) = \infty$ and $\nu(x) = X$ iff $x \in \textnormal{peel}(X)$, but the example in Figures \ref{fig:letters_peel} and \ref{fig:letters_peel_neighborhood_graph_nearly_full} shows that this is not the case. The example in Figure \ref{fig:letters_peel_neighborhood_graph_nearly_full} also illustrates that it is usually practically expedient to consider a soft threshold on $\rho(x)$ or a maximum number of points in a peel neighborhood. In other words, in practice we grow a ball and check its peel, stopping as soon as the radius or number of points exceeds a threshold, as in Figure \ref{fig:letters_peel_neighborhood_graph}. We usually do not bother with any explicit notation for thresholded peel neighborhoods (Figure \ref{fig:letters_peel_neighborhood_graph} is an exception).

\begin{figure}[htbp]
  \centering
  \includegraphics[trim = 48mm 39mm 40mm 34mm, clip, width=.9\textwidth,keepaspectratio]{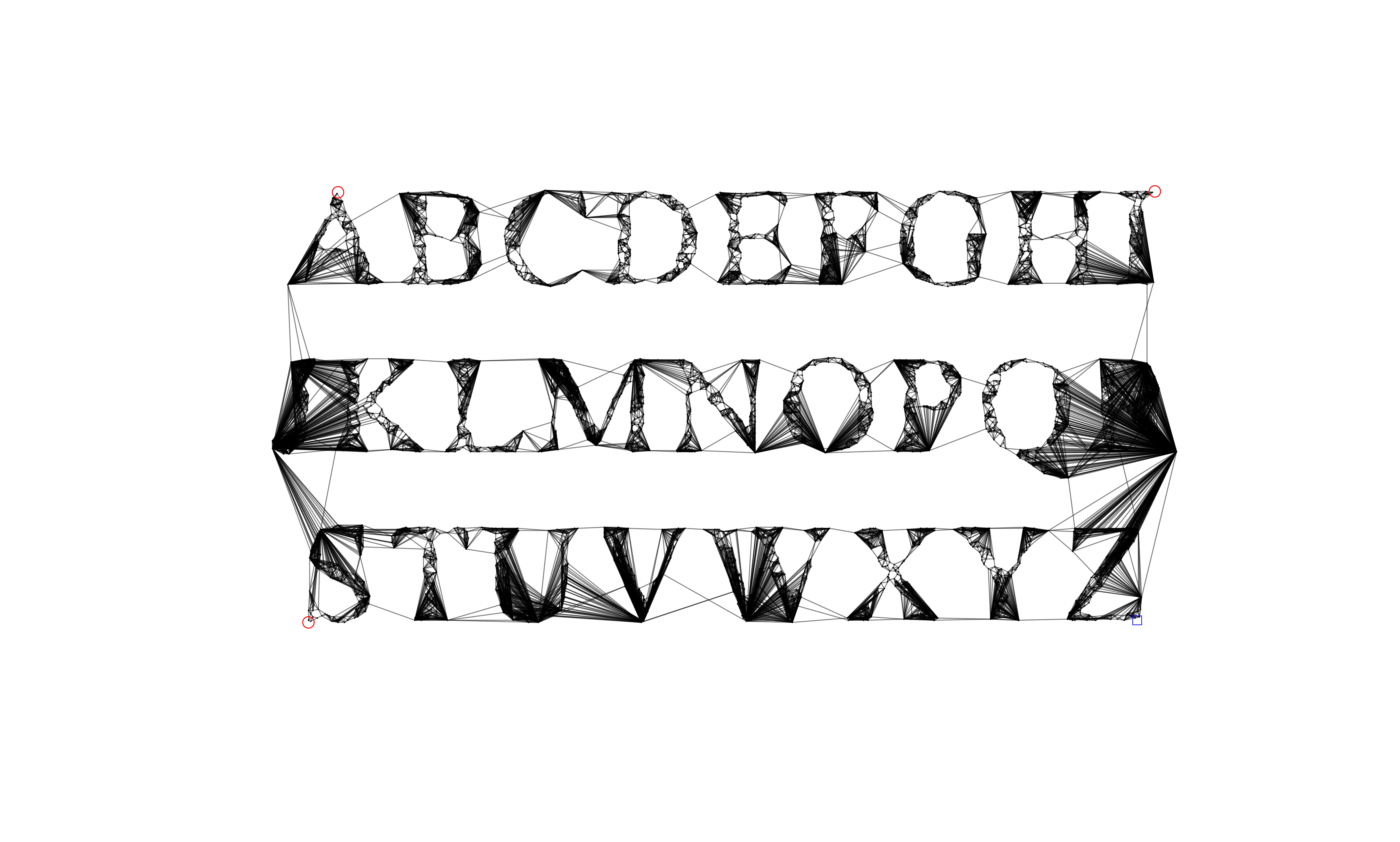}
\caption{A graph on $N = 4782$ vertices embedded in $\mathbb{R}^2$ with edge set contained in $E_\nu := \{(j,k): j \in [N], k \in \nu(j) \} \cup \{(j,k): k \in [N], j \in \nu(k) \}$. {\color{red}Vertices at extreme points circled in red are those for which $\nu(x) = X$, and the corresponding edges are excluded}; {\color{blue}the vertex at the lower right corner inside a blue square has very large but finite $\rho(x)$: the corresponding peel neighborhood is also excluded here, but reaches to the letters F, N, and V}. 
The exact neighborhoods and their heuristic approximations are identical.
} 
  \label{fig:letters_peel_neighborhood_graph_nearly_full}
\end{figure}

\begin{figure}[htbp]
  \centering
  \includegraphics[trim = 48mm 39mm 40mm 34mm, clip, width=.9\textwidth,keepaspectratio]{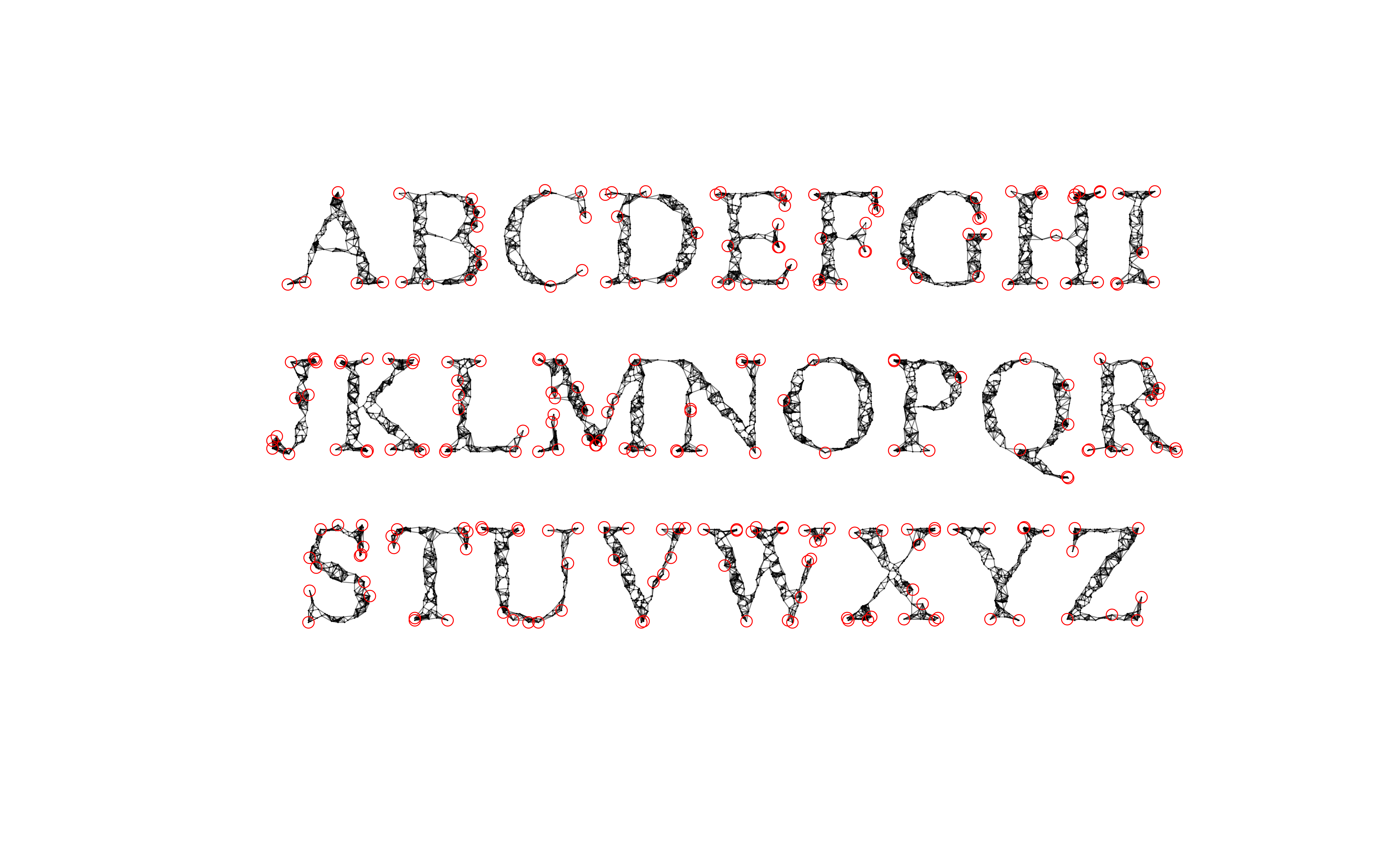}
\caption{A graph on $N = 4782$ vertices embedded in $\mathbb{R}^2$ with the edge set $E_{\rho<R} :=\{(j,k): j \in [N], k \in \nu_{<R}(j) \} \cup \{(j,k): k \in [N], j \in \nu_{<R}(k) \}$, where $\nu_{<R}(x)$ indicates a thresholded peel neighborhood with radius $\min\{R,\rho(x)\}$ with the threshold $R$ the default described in the main text. {\color{red}Vertices circled in red are those for which the default threshold in the main text is reached: these capture a notion of local extreme points.} Note that this graph exhibits an excellent clustering.
There is no difference here between the exact thresholded peel neighborhoods and the thresholded heuristic.
} 
  \label{fig:letters_peel_neighborhood_graph}
\end{figure}

In practice, we efficiently compute approximate nearest neighbors using a hierarchical navigable small world graph \cite{malkov2018efficient}. We typically set a radius threshold to $(1+1/m)$ times the median $k$th nearest neighbor distance, where $k = \lceil \log_2 |X| \rceil$. This threshold typically strikes a reasonable balance between connectivity and efficiency in low to moderate data dimension. The model and Poisson process experiments in \S \ref{sec:aUsefulModel} and our other experiments in \S \ref{sec:experiments} suggest that this is a conservative and effective bound in practice that avoids having a few points take much longer to handle than all of the others combined. 
All of the computational demonstrations in this paper use this default threshold unless explicitly stated otherwise. 
\footnote{Any soft threshold of practical use can itself be largely canonicalized while saving computational effort at the expense of some added code complexity. The idea is to predict and compute peel neighborhoods that are small first, keeping track of neighborhoods that have not been fully formed while expanding the set of small neighborhoods that are fully formed; and to only determine the large ``boundary-like'' approximate peel neighborhoods at the end of this process, capping their size at or slightly above the sizes of their already-computed neighbors. We do not take this route here, but we anticipate taking it in practice going forward.}

Thresholding peel neighborhoods avoids computing many peels of $B_r(x)$ for many values of $r$ when $x$ is an outlier. This allows substantial variability in the radii of peel neighborhoods without substantial risk of ``oversmoothing'' and keeps the computation of thresholded peel neighborhoods efficient. In particular, the number of loops in Algorithm \ref{alg:ApproximatePeel} becomes explicitly bounded or at least practically small, and the computation of peel neighborhoods scales linearly with dimension (because of atomic distance calculations), and polylogarithmically with cardinality in practice (because of approximate nearest neighbor search \cite{malkov2018efficient}). In practice, we have computed peel neighborhoods with this observed scaling (not shown) for data sets involving hundreds of thousands of points and thousands of dimensions, with the calculations requiring less than 20 minutes on average for a platform with two high-end GPUs. Note that once we have computed $\rho(x)$ for all $x \in X$, it is fairly straightforward and much less computationally burdensome to use that information for tighter (re)thresholding.

By Theorem \ref{thm:peeling}, if the cardinality of a peel neighborhood is $n$, then it can be computed in time $O(n^{\omega+1})$. More generally, if the cardinality of all (approximate) peel neighborhoods is bounded by $n$, then they can all be approximated in time $O(mn^{\omega+1} N \cdot \textnormal{polylog}(N))$. The factor of $m$ arises from the complexity of distance calculations, and there are $N$ peel neighborhoods to compute, with polylogarithmic time required for the approximate nearest neighbor search. In practice, we often use the radial threshold described above instead of a cardinality threshold to highlight isolated or boundary-like points efficiently. In any event, choosing an appropriate threshold in practice effectively yields linear scaling in both $m$ and $N$. Any such threshold, including our default choice, can be efficiently determined and justified by random sampling paired with computational timeouts.

\subsection{A useful model and a Poisson process experiment}\label{sec:aUsefulModel}

A reasonable cartoon of many high-dimensional data is that all distances to nearby points are approximately the same, that all distances between pairs of nearby points are approximately the same, and that all angles subtended by pairs of nearby points are approximately right angles. \cite{hall2005geometric} A paradigmatic example is that for $m \gg 1$, the probability mass of the standard Gaussian $\mathcal{N}(0,I_m)$ on $\mathbb{R}^m$ is concentrated in a narrow shell of radius $\sqrt{m}$; the probability mass of distance between two points drawn from $\mathcal{N}(0,I_m)$ is approximately $\mathcal{N}(\sqrt{2(m-1)},1)$ and hence concentrated near $\sqrt{2(m-1)}$; and the inner product of two points will be $\pm O(1/\sqrt{m})$, so the angle they subtend will be $\pi/2 \pm O(1/\sqrt{m})$. \cite{blum2020foundations}

In other words, under a Gaussian \emph{Ansatz} and suppressing irrelevant factors of $\sqrt{m}$, a reasonable model distance matrix on $N$ points for $m \gg 1$ is
\begin{equation}\label{eq:modelDistanceMatrix}
    D_N := \begin{pmatrix}
        0 & 1_{N-1}^T \\
        1_{N-1} & \sqrt{2} \cdot (11^T-I_{N-1})
    \end{pmatrix}
\end{equation}
which has inverse
\begin{equation}\label{eq:modelDistanceMatrixInverse}
    D_N^{-1} = \frac{\sqrt{2}}{2(N-1)} \begin{pmatrix}
        -2(N-2) & \sqrt{2} \cdot 1_{N-1}^T \\
        \sqrt{2} \cdot 1_{N-1} & 11^T-(N-1)I_{N-1}
    \end{pmatrix}.
\end{equation}
Because \eqref{eq:modelDistanceMatrix} is the distance matrix of the origin and points of a regular simplex centered at the origin in sufficiently high dimension, it is strict negative type. We have that
\begin{equation}\label{eq:modelDistanceMatrixInverseONE}
    D_N^{-1} 1_N := \frac{\sqrt{2}}{2} \begin{pmatrix}
        -2\frac{N-2}{N-1} + \sqrt{2} \\
        \frac{\sqrt{2}}{N-1} 1
    \end{pmatrix}.
\end{equation}

Now the origin is not in $\textnormal{peel}(D_N)$ iff $\sqrt{2} (N-1) \le 2(N-2)$, or equivalently iff 
\begin{equation}\label{eq:modelDistanceMatrixThresholdN}
    N \ge \frac{4 - \sqrt{2}}{2-\sqrt{2}} \approx 4.414.
\end{equation}
This suggests that for $m \gg 1$, the peel neighborhood of a basepoint surrounded by standard Gaussian-distributed points will have around five points, counting the basepoint itself. This is a remarkable if soft invariance with respect to dimension that helps explain why peel neighborhood cardinalities are so low in Figure \ref{fig:rho_ball_histograms}. We illustrate this phenomenon in Figure \ref{fig:origin_gaussian}.

\begin{figure}[htbp]
  \centering
  \includegraphics[trim = 15mm 52.5mm 15mm 0mm, clip, width=.9\textwidth,keepaspectratio]{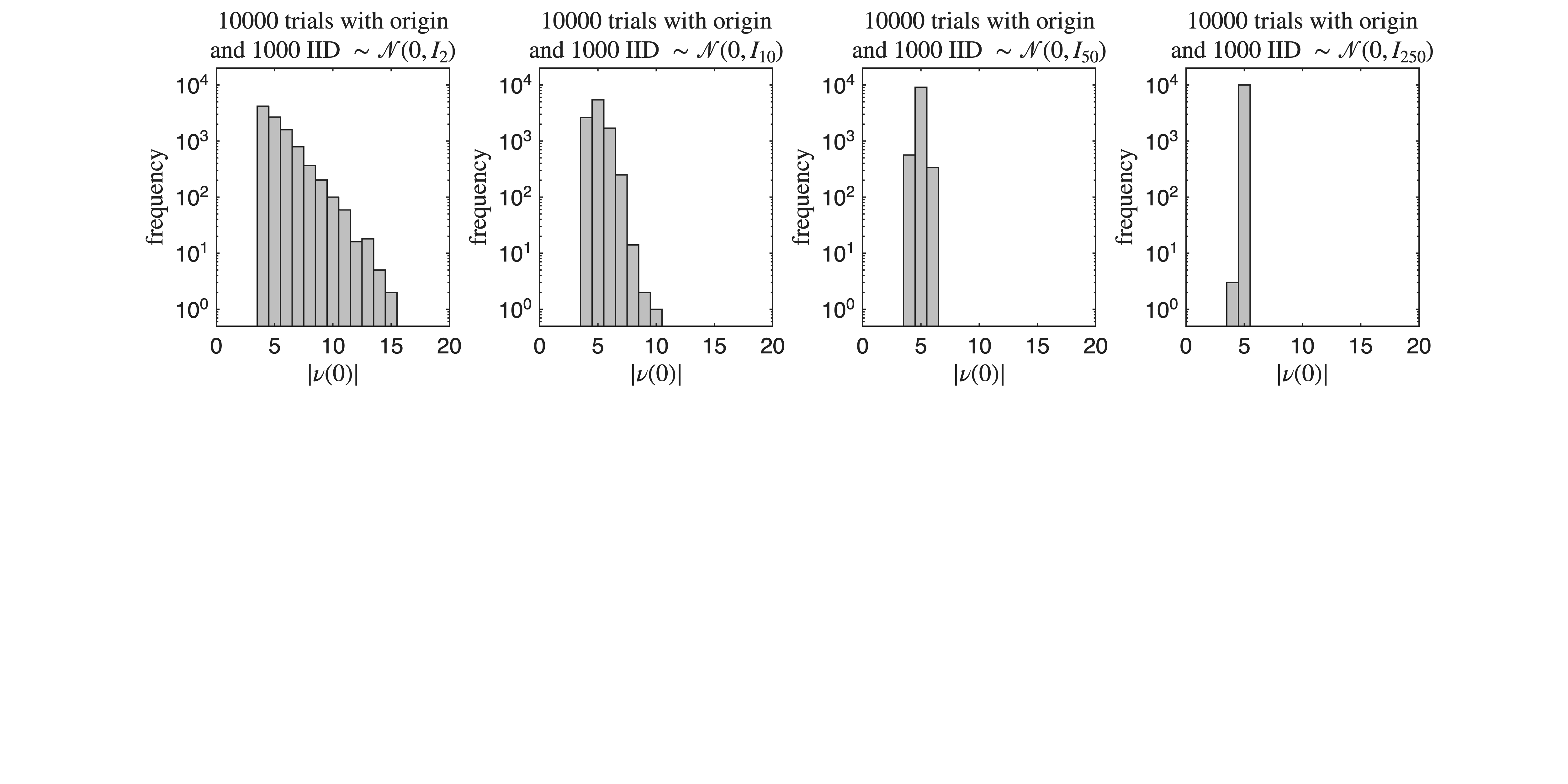}
\caption{Histograms of the cardinalities of unthresholded peel neighborhoods of the origin amongst 1000 IID $\mathcal{N}(0,I_m)$ points, for $m \in \{2, 10, 50, 250\}$. 
The exact neighborhoods and their heuristic approximations are always identical here.}
  \label{fig:origin_gaussian}
\end{figure}

We also performed a similar experiment with unthresholded peel neighborhoods of uniformly distributed points (i.e., a Poisson process) on flat tori $(\mathbb{R}/\mathbb{Z})^m$. However, unless the number $N$ of points grows exponentially with $m$, sparsity controls the cardinalities of peel neighborhoods in this setting. Figure \ref{fig:poisson} still shows clear evidence of convergence as $m$ grows, albeit illustrated over a smaller range and using more points to prevent sparsity from controlling. Reducing $N$ from $10^5$ in Figure \ref{fig:poisson} to $10^4$ results in nearly identical histograms, and even for $N = 10^3$ the quantitative effects are relatively small, on the order of 10 percent smaller maximal bin populations and longer tails for $N$ smaller (not shown). Meanwhile, since peel neighborhoods tend to not have many more than five to ten elements, a sample of more than a thousand ($\approx 2^{10}$) or so points from a constant-intensity Poisson point process on a flat torus can generally be expected to have peel neighborhoods rarely if ever saturate a soft threshold: we have observed this for dimensions $2 \le m \le 10$.

Figure \ref{fig:poisson_threshold} illustrates that unthresholded peel neighborhood radii are well below the default radial threshold across the range $10^3 \le N \le 10^5$ and $2 \le m \le 10$. In short, the default threshold is only triggered for certain boundary-like points (or very sparse data in high dimension, which is in many ways a subcase). Recall that Figure \ref{fig:letters_peel_neighborhood_graph} illustrated this same point. Viewed from this perspective, the default threshold offers analytical as well as computational benefits, with little to no impact on points where the notion of a peel neighborhood is most directly relevant: i.e., interior-like points.

\begin{figure}[htbp]
  \centering
  \includegraphics[trim = 0mm 0mm 0mm 0mm, clip, width=\textwidth,keepaspectratio]{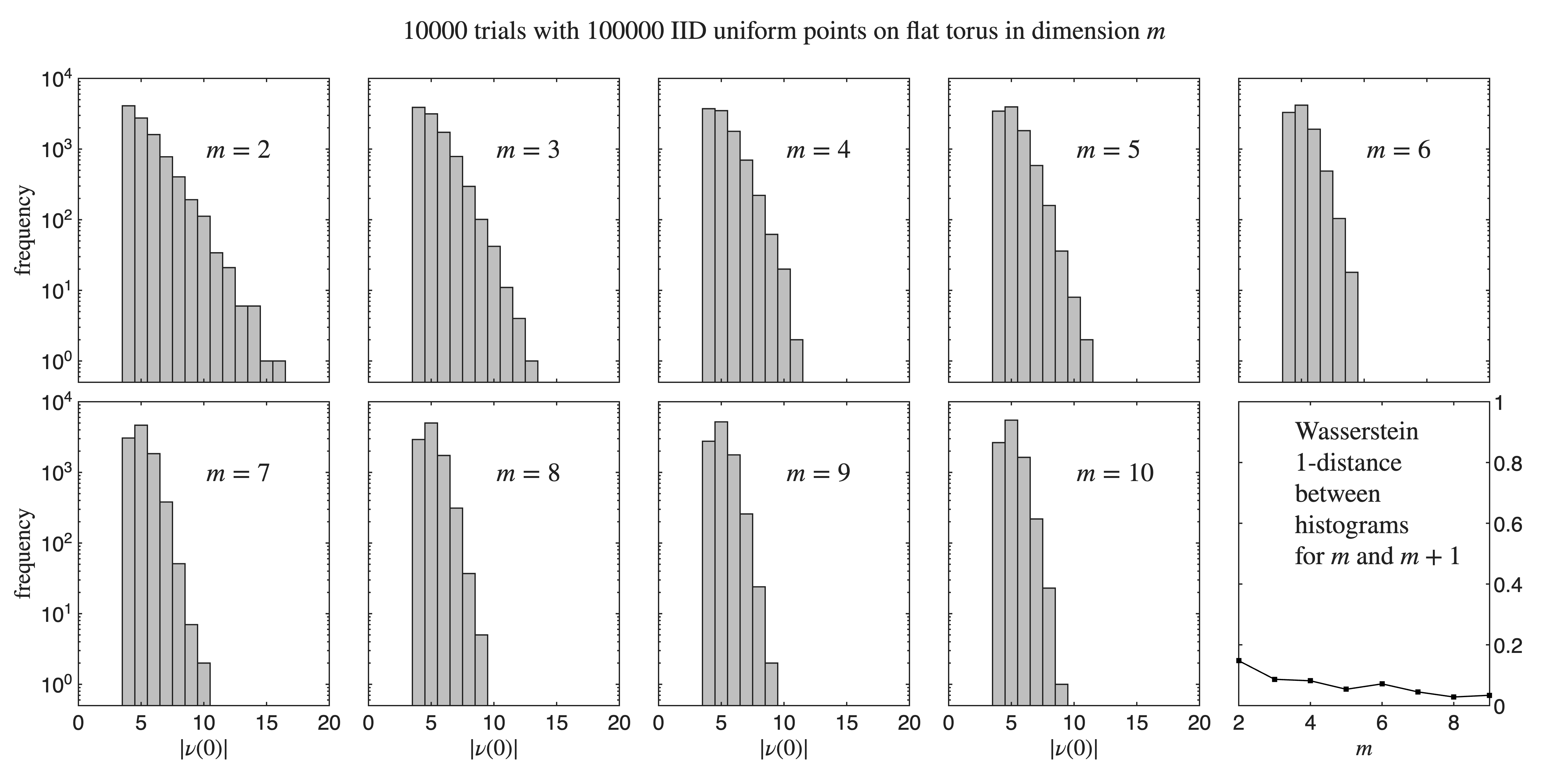}
\caption{First nine panels: histograms of the cardinalities of unthresholded peel neighborhoods for 100000 IID uniform points on $(\mathbb{R}/\mathbb{Z})^m $ for $m \in \{2, \dots, 10\}$. Lower right panel: the 1-Wasserstein distance between successive histograms. 
The cardinalities of exact peel neighborhoods and their heuristic approximations are very nearly equal and visually indistinguishable. There are just six disagreements out of $90000$ peel neighborhood cardinalities, with heuristic radii between 0.98 and 1 times the exact radii except in a single case, where the heuristic radius is $\approx 0.84$ times the radius of the exact radius.
(NB. An earlier version of this paper on arXiv had a bug in the underlying distance matrix computations that yielded different but still well-behaved histograms.)
}
  \label{fig:poisson}
\end{figure}

    

\begin{figure}[htbp]
  \centering
  \includegraphics[trim = 5mm 0mm 5mm 0mm, clip, width=.48\textwidth,keepaspectratio]{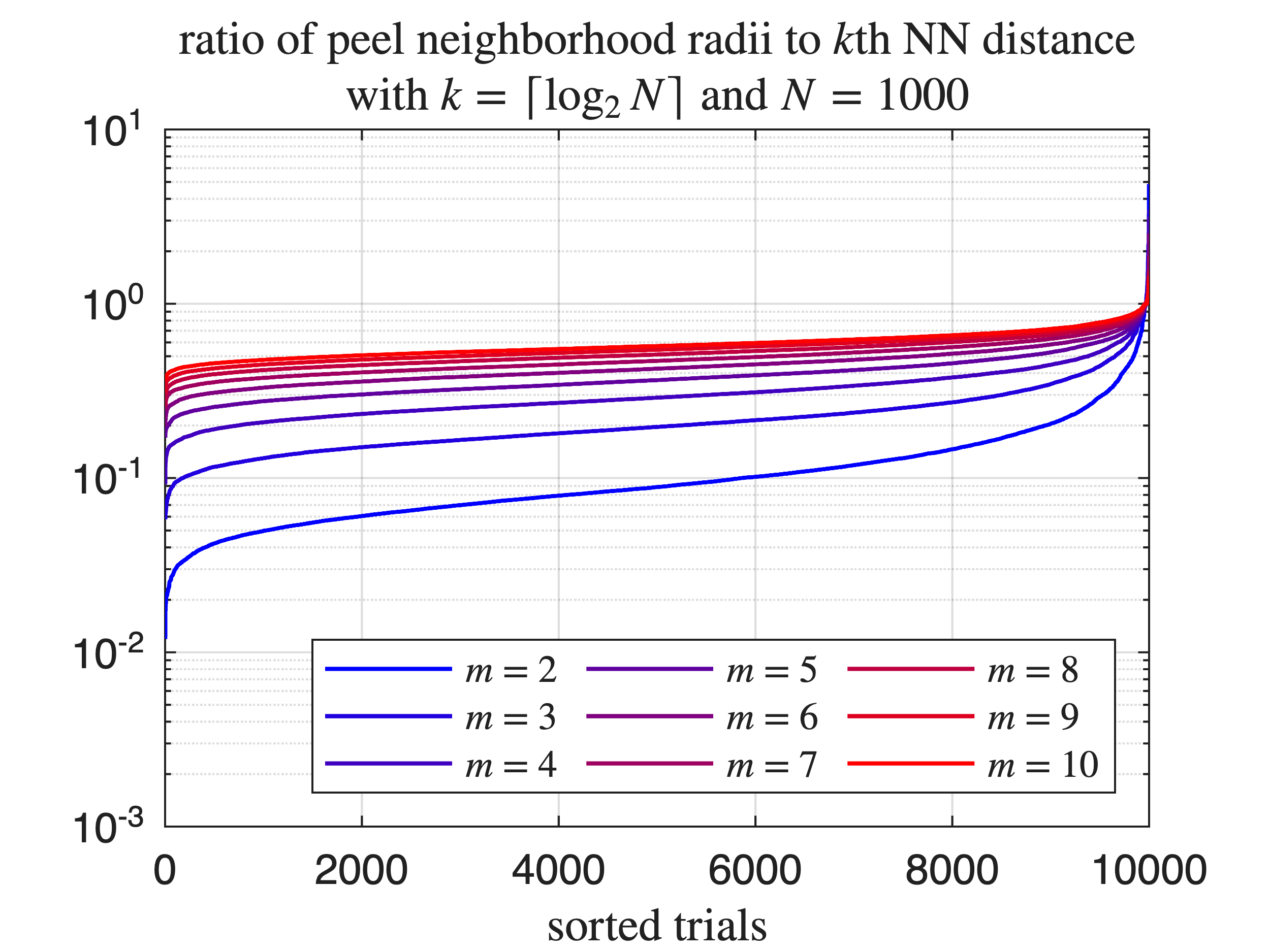}
  \includegraphics[trim = 5mm 0mm 5mm 0mm, clip, width=.48\textwidth,keepaspectratio]{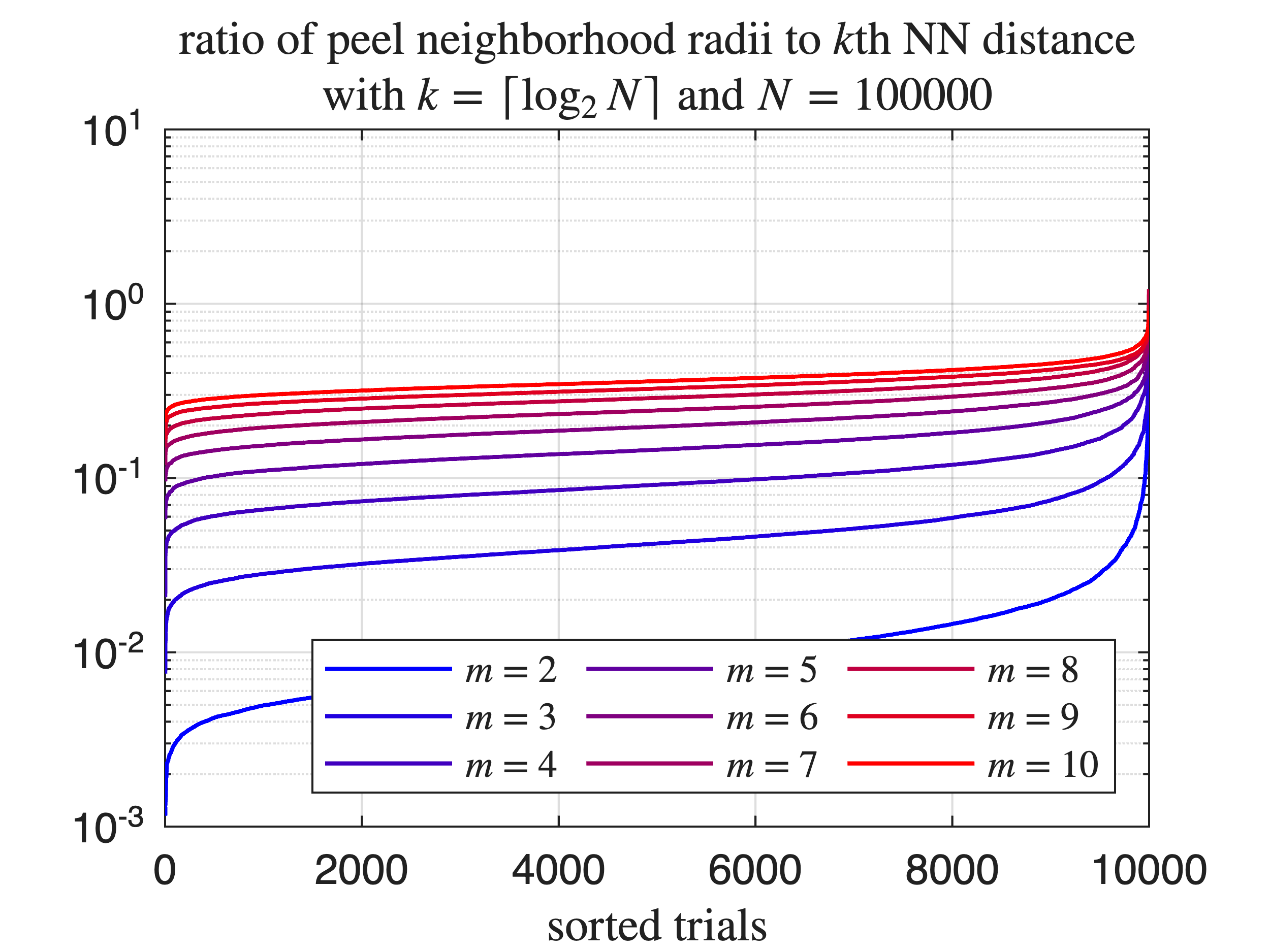}
\caption{Left: sorted ratios of unthresholded peel neighborhood radii to the $k$th nearest neighbor distance with $k = \lceil \log_2 N \rceil$ over $10^4$ trials with $N = 10^3$ IID uniform points on $(\mathbb{R}/\mathbb{Z})^m$ for $m \in \{2,\dots,10\}$. Right: as in the left panel, but for $N = 10^5$. Note that $k \gg 5$ in any practical situation.} 
  \label{fig:poisson_threshold}
\end{figure}

\section{Experiments}\label{sec:experiments}

\subsection{Comparison with fixed-radius and $k$-nearest neighborhoods}\label{sec:comparison}

Overloading a word, we will give quantitative evidence that peel neighborhoods encode an ``efficient'' notion of locality. For context, the \emph{efficiency} of a distance-weighted graph $G$ on $[N]$ is $$\textnormal{eff}(G) := \frac{1}{N(N-1)} \sum_{j,k \in [N]; j\ne k} d_{jk}^{-1},$$ where $d_{jk}$ is the shortest-path distance \cite{latora2001efficient}. More sparsity-aware variants of this are the efficiencies per edge and per length, given by $\textnormal{eff}(G)/\sum_{j,k} A_{jk}$ with $A$ respectively equal to the unweighted and distance-weighted adjacency matrix of $G$ \cite{huntsman2020fast}.

With this context, consider the ``radial'' graphs\footnote{These are usually called \emph{(random) geometric graphs} in the literature.} obtained by connecting vertices within a given radius and the ``$k$NN'' graphs obtained by connecting the $k$ nearest neighbors. Obviously, as the radius and/or $k$ increase, the efficiency will increase as well, but the efficiency per edge or per length may not. Meanwhile, the number of connected components is nonincreasing and ultimately decreases to one.

\begin{figure}[htbp]
  \centering
  \includegraphics[trim = 0mm 15mm 0mm 15mm, clip, width=\textwidth,keepaspectratio]{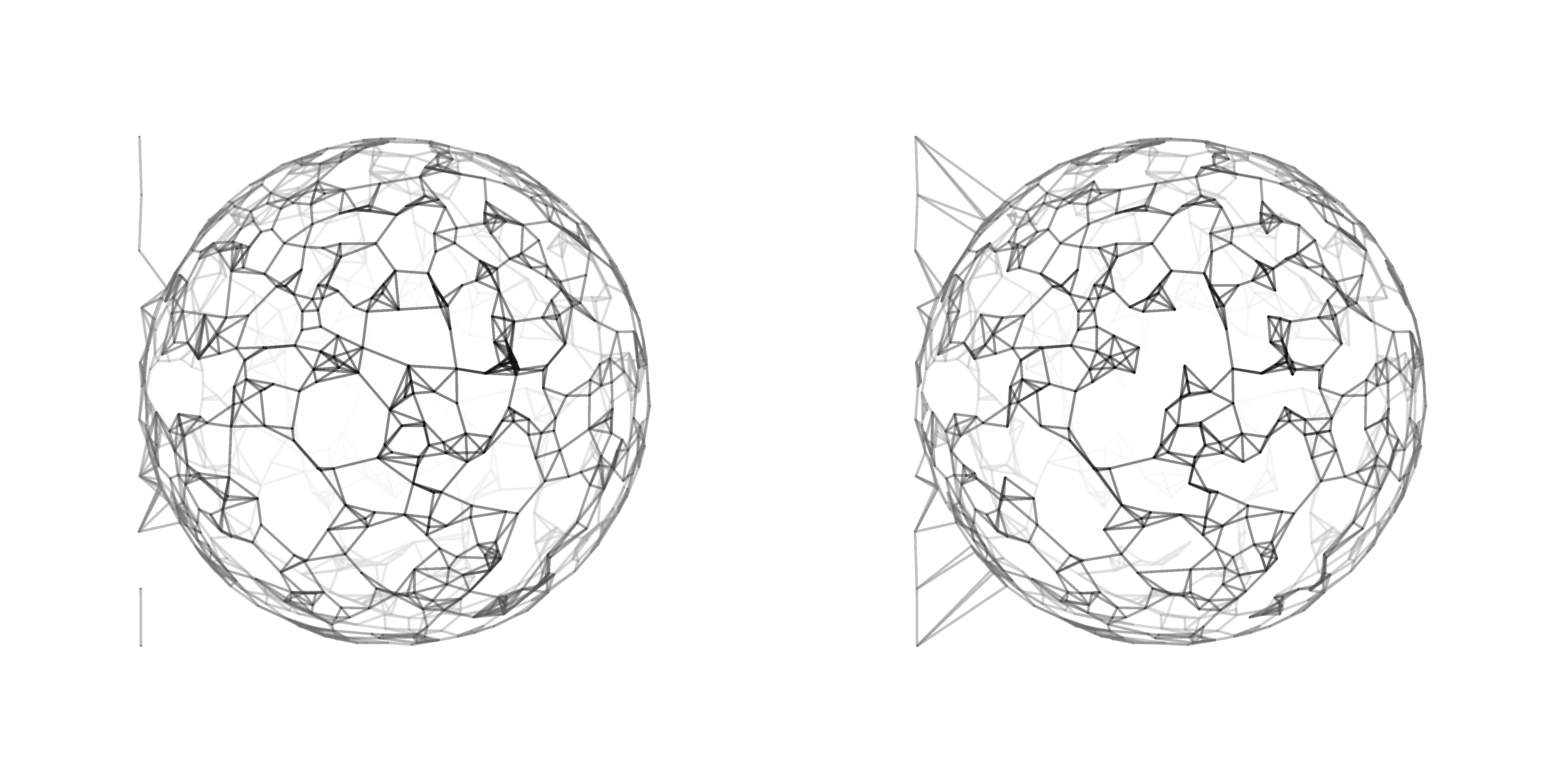}
  \includegraphics[trim = 15mm 0mm 15mm 0mm, clip, width=\textwidth,keepaspectratio]{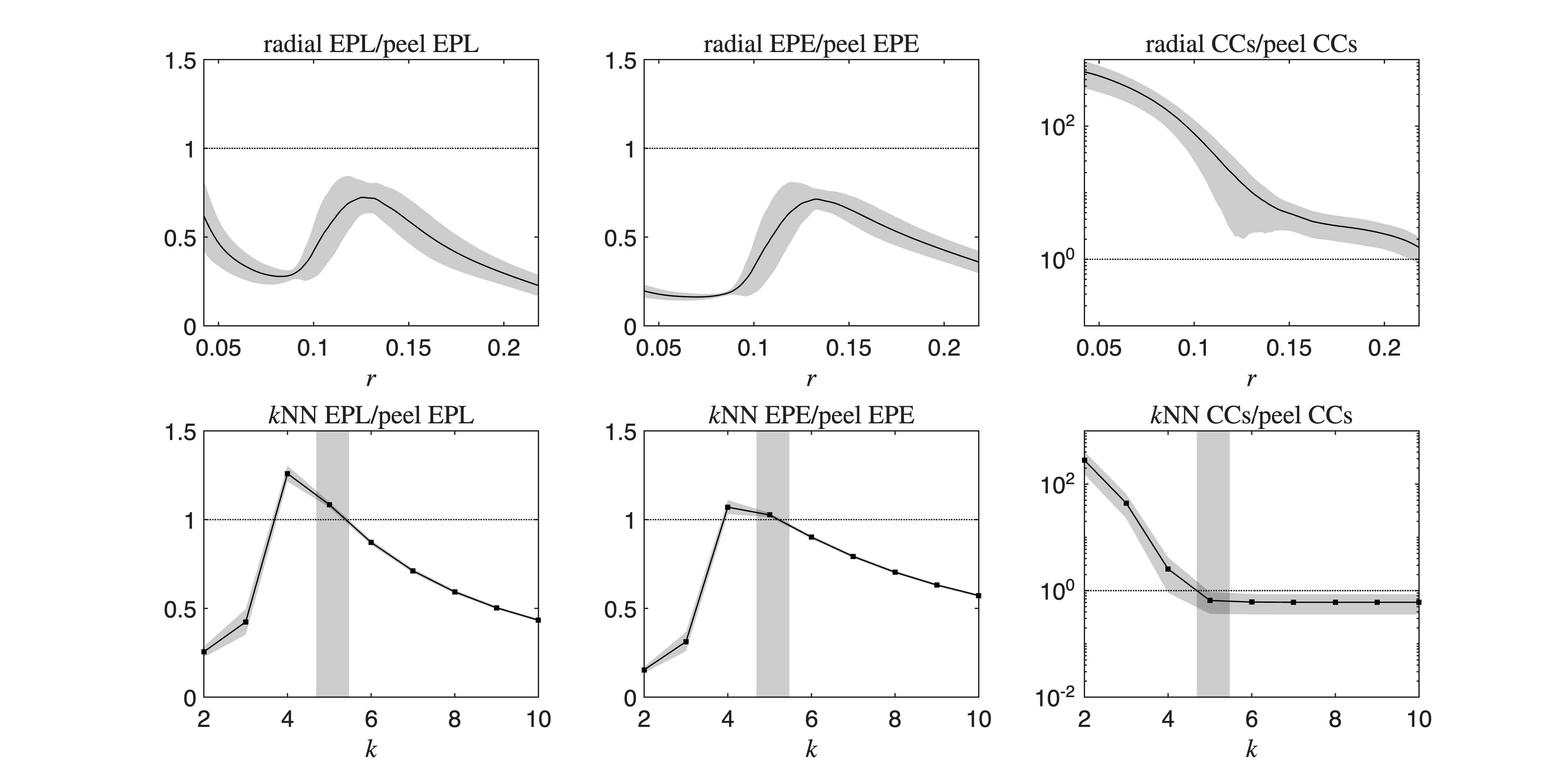}
\caption{Top left: the graph with edges given by peel neighborhoods on a uniform sample of approximately $1000$ points from a sphere $S^2 \subset \mathbb{R}^3$, plus $11$ equispaced points slightly to the left, embedded in 10 dimensions with small Gaussian noise added. 
The exact neighborhoods and their heuristic approximations are always identical here.
Edges are shaded according to their average in the third dimension. Top right: the similar graph with edges given by $k$-nearest neighbors with the minimal $k$ such that the graph has a single connected component (here, $k = 5$). Middle: the relative efficiencies per edge and per length of radial graphs \emph{versus} peel neighborhood graphs, along with the relative numbers of connected components, all as functions of radius, with means and standard deviations over 100 trials indicated. Bottom: as in the middle panels, but for $k$NN graphs. Initial values of $k$ for which the $k$NN graph is fully connected are indicated by the vertical patch of $\pm$ a standard deviation about the mean.} 
  \label{fig:neighborhood_graphs_sphere}
\end{figure}

Figure \ref{fig:neighborhood_graphs_sphere} (and a similar figure for the torus in Appendix \S \ref{sec:another51})
illustrates that the efficiencies per edge and per length are usually higher for peel neighborhoods than for radial or $k$-nearest neighborhoods, and that the number of connected components is usually lower for peel neighborhoods for most reasonable radii and for small $k$. Meanwhile, for larger $k$, efficiencies decrease and spurious connectivity emerges. Importantly, the peel neighborhoods require no hard parameter choice. As Figure \ref{fig:letters_peel_neighborhood_graph} illustrates above and Figure \ref{fig:peelNeighborhoodsAnnular} illustrates below, the thresholds only matter for a relatively few boundary-like points, so they are robust.

\subsection{Samples from annuli}\label{sec:annuli}

Figure \ref{fig:peelNeighborhoodsAnnular} shows covers consisting of peel neighborhoods with the default threshold discussed in \S \ref{sec:peelNeighborhoods}, and Figure \ref{fig:peelNeighborhoodsHoles} shows a proxy $$\min_{x \in X} (\|x\|-\rho(x))$$ for the distance from the origin to the same sort of cover of a noisy embedding of a sample from an annulus for various sample sizes, embedding dimensions, and noise levels. 
Appendix \S \ref{sec:annuli_radii} contains additional statistical detail.
Note that the first panel of Figure \ref{fig:peelNeighborhoodsAnnular} can indicate how a negative value for this proxy does not automatically imply that the topology of the cover is trivial: on the other hand, a positive value for this proxy practically means that the topology of the cover has first Betti number equal to 1. It is evident that these covers can accurately reflect the underlying topology of noisy samples in high embedding dimension, though eventually sparsity, noise, and embedding dimension can conspire to overwhelm any construction.

\begin{figure}[htbp]
    \centering
    \includegraphics[width=\textwidth, trim={0 150 0 0mm}, clip]{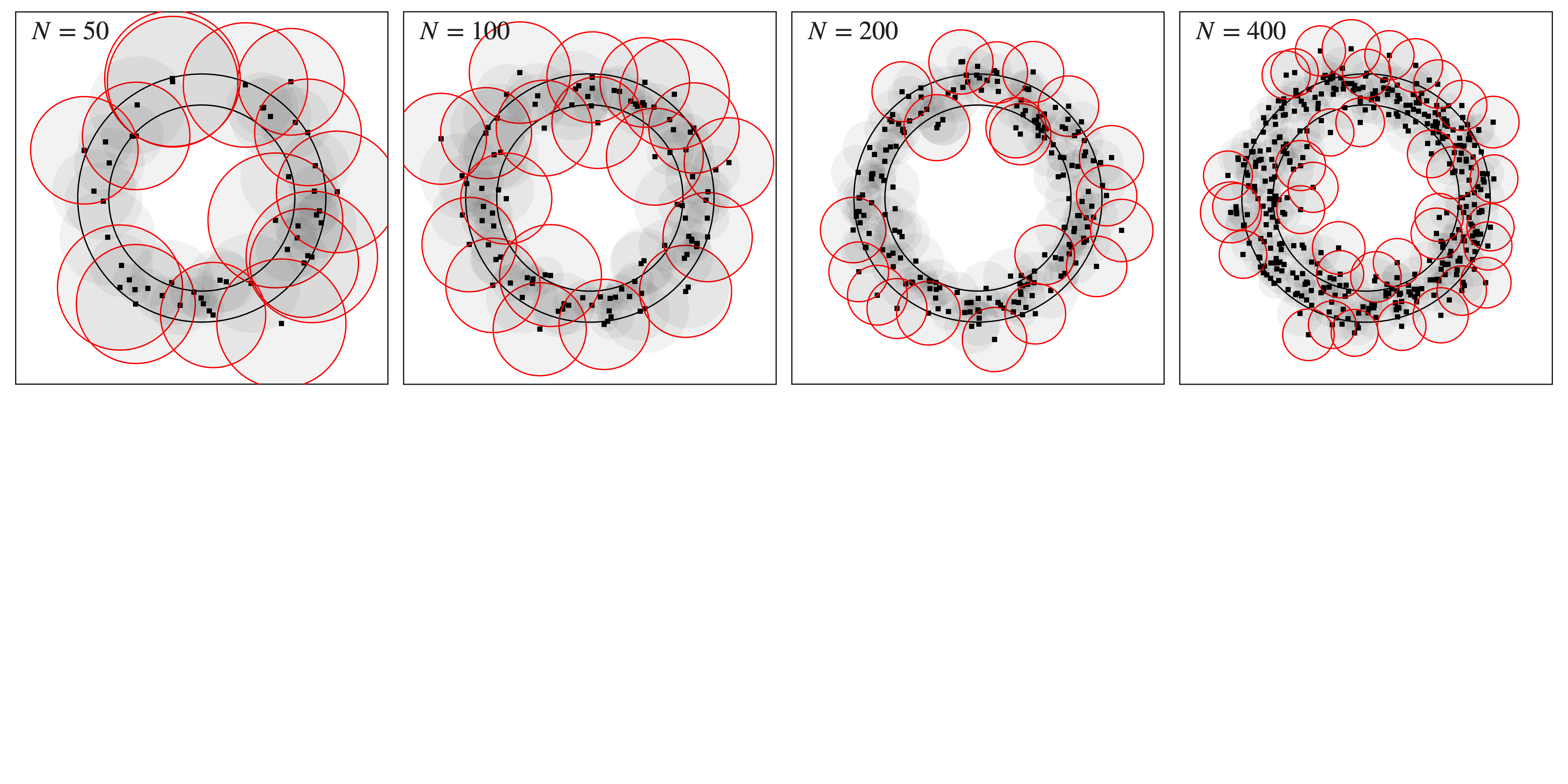}
        \caption{Peel neighborhoods of various points indicated by shaded disks, with the threshold on the radius described in \S \ref{sec:peelNeighborhoods} applied; {\color{red}neighborhoods that saturate the threshold are circled in red}.
        The exact neighborhoods and their heuristic approximations are always identical here.
        Note that saturating neighborhoods are at locally outlying points. Each sample is initially uniform over a plane annulus with outer radius 1, inner radius 0.75, and with isotropic Gaussian noise of standard deviation $0.1$ subsequently added. From left to right, top to bottom, the samples are of 50, 100, 200, and 400 points. 
        }
    \label{fig:peelNeighborhoodsAnnular}
\end{figure}

\begin{figure}[htbp]
    \centering
    \includegraphics[width=\textwidth, trim={0 0 0 0mm}, clip]{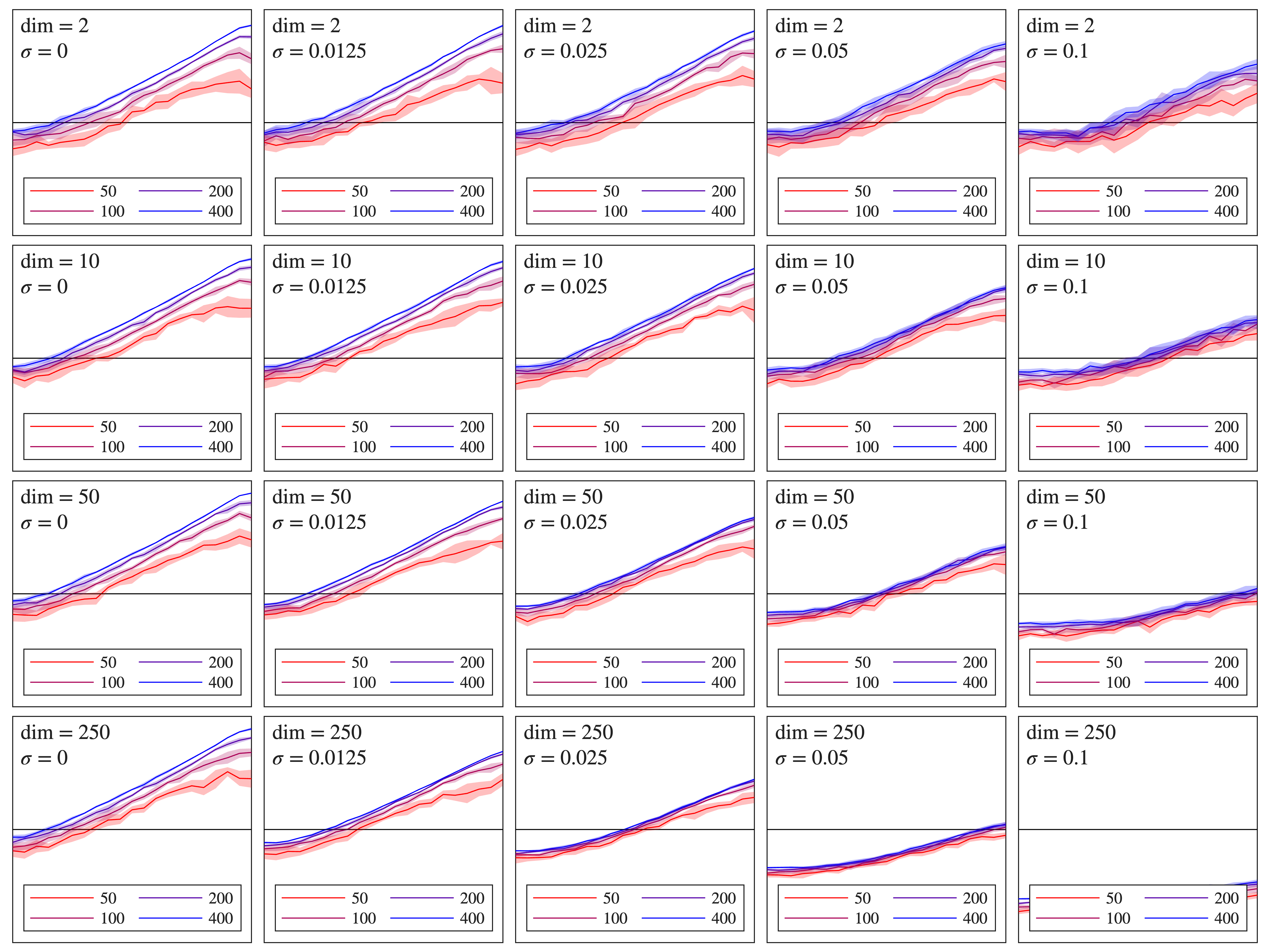}
        \caption{A proxy $\min_{x \in X} (\|x\|-\rho(x))$ for the distance from the origin to the cover of $X$ defined by its thresholded peel neighborhood, as a function of the inner radius of an annulus. Each panel is over $[0,1] \times [-1,1]$ and shows the average and standard deviation for this proxy taken over 10 realizations of $X$ sampled initially uniformly over a plane annulus with outer radius 1 and variable inner radius parametrizing the horizontal axes. Subsequently, $X$ is embedded into the number of dimensions indicated and isotropic Gaussian noise of standard deviation $\sigma$ is added. Sample sizes of 50, 100, 200, and 400 points are indicated with colors ranging from {\color{red}red} to {\color{blue}blue}. Exact and heuristic peel neighborhoods (the latter not shown) yield visually indistinguishable results, with the differences concentrated in dimension $\ge 10$ and more prevalent as noise increases. Of the 16800 total realizations, only 918 had any differences between exact and heuristic peel neighborhoods at all; of those, only eight had any heuristic peel neighborhood radii smaller than 80 percent of the actual values. 
        }
    \label{fig:peelNeighborhoodsHoles}
\end{figure}

\subsection{Peel neighborhood scaling and relation to peels}\label{sec:scaling}

We have seen that peel neighborhoods robustly encode locality, including local boundaries. A natural question is if (and how) they can encode global boundaries, including peels themselves. While Lemma \ref{lem:inclusionBound} and the propositions that follow it are the best formal results we have obtained in this direction, computational demonstrations give more insight.

Figure \ref{fig:neighborhoodPeelApproximation} compares peel distributions $p_*$ to approximations $\hat p_*$ obtained by first computing thresholded peel neighborhoods, then restricting the metric to the points whose neighborhood radii exceed the threshold. The figure shows the normalized energy distance (NED) \cite{lyons2013distance} defined as $$\textnormal{NED}(p_*,\hat p_*) := \frac{2p_*^T d \hat p_*-p_*^Tdp_*-\hat p_*^Td\hat p_*}{2 \max_{j,k} d_{jk}}.$$ The NED saturates the unit interval on the space of probability distribution pairs (to see this, use convexity of the quadratic form and the fact that distances between point masses are saturating). The figure shows that the approximation $\hat p_*$ is good in dimension two, but degrades as the dimension increases (in particular but not shown here, the NED takes values comparable to that between the uniform distribution on $\textnormal{peel}(d)$ and the uniform distribution on $[N]$, though this is still much less than 1). Still, this approximation is very interesting for potential large-scale applications.

\begin{figure}[htbp]
    \centering
    \includegraphics[width=.32\textwidth, trim={35 70 25 10mm}, clip]{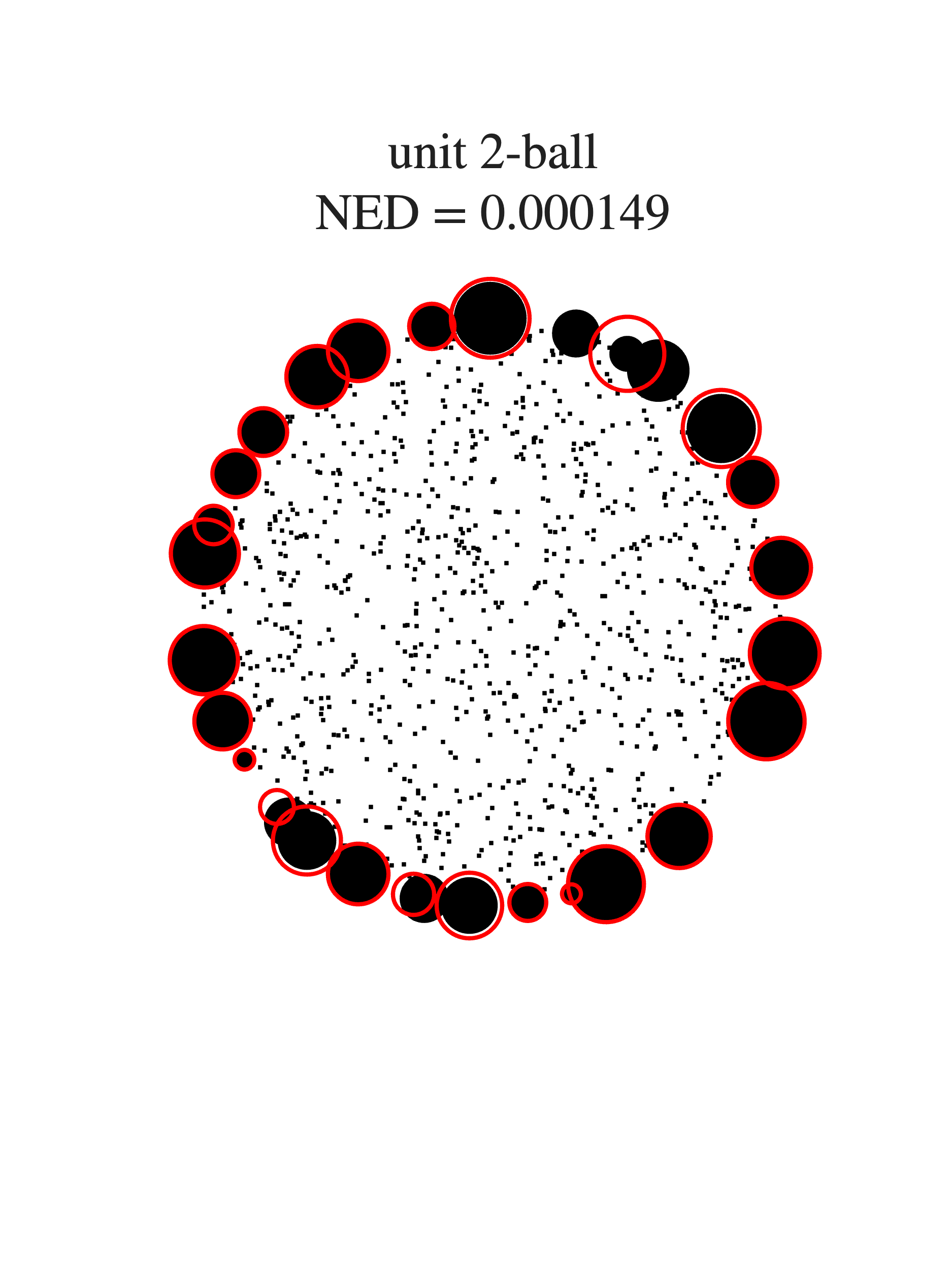}
    \includegraphics[width=.32\textwidth, trim={35 70 25 10mm}, clip]{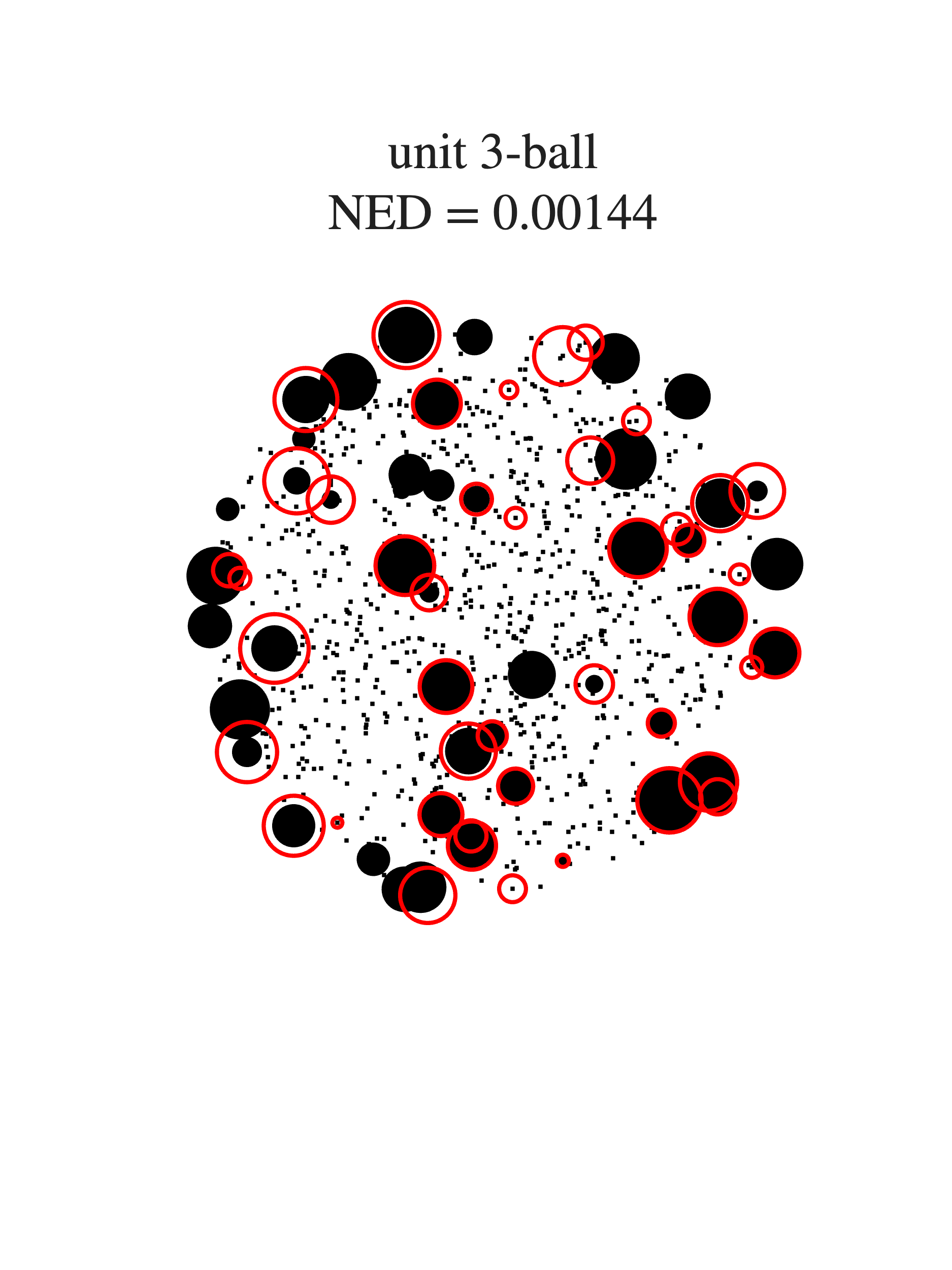}
    \includegraphics[width=.32\textwidth, trim={35 70 25 10mm}, clip]{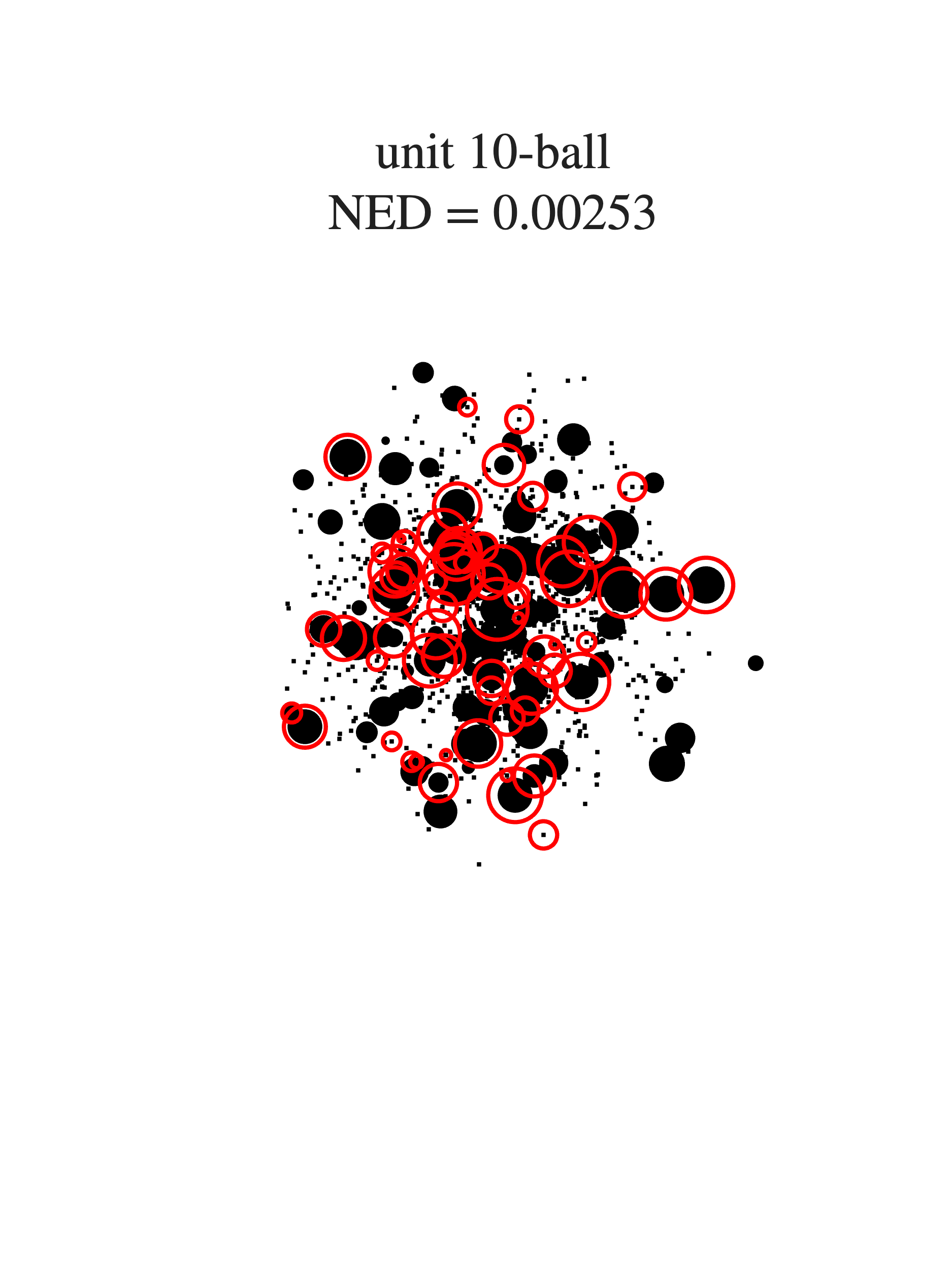}
    \includegraphics[width=.32\textwidth, trim={5 0 0 0mm}, clip]{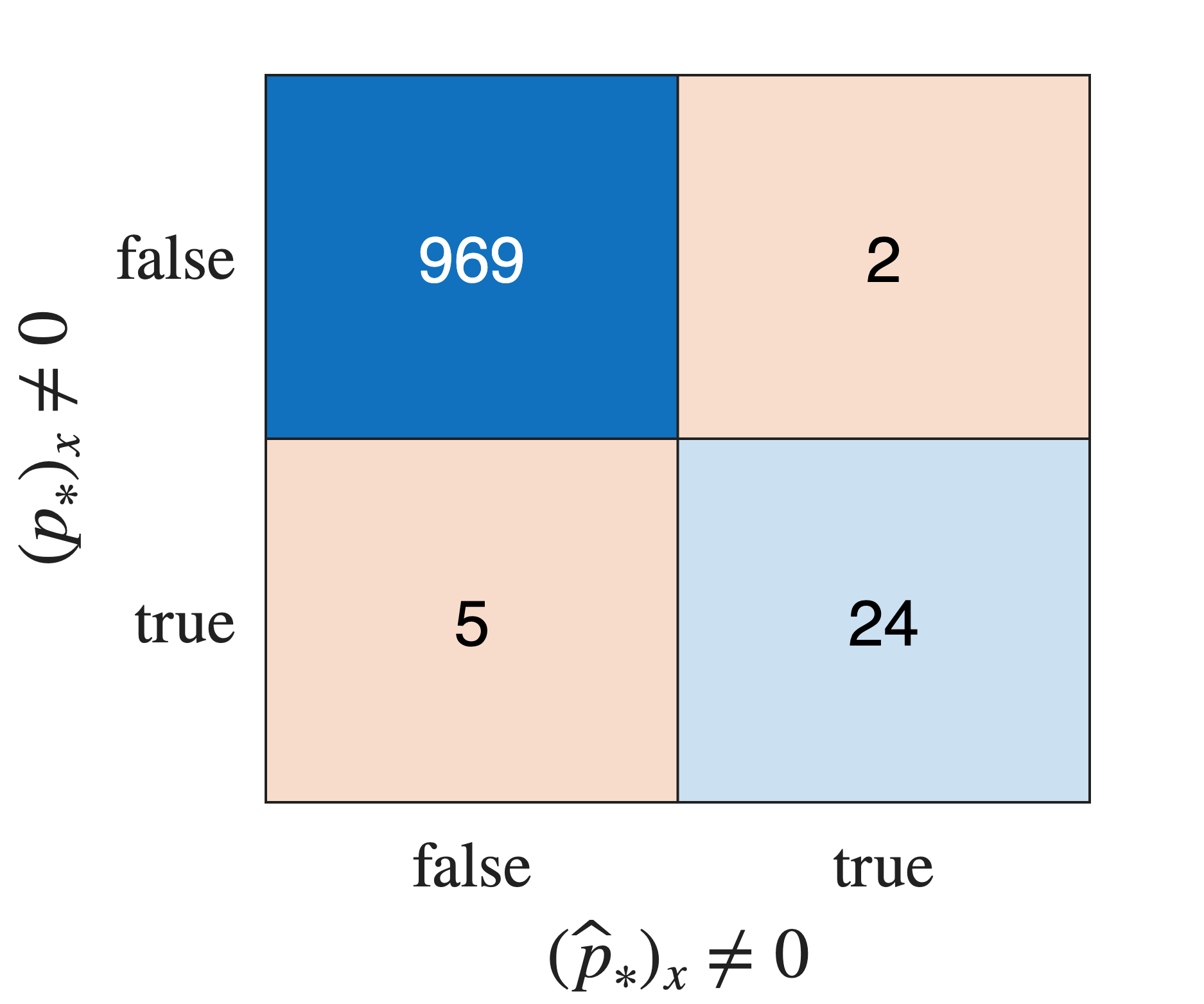}
    \includegraphics[width=.32\textwidth, trim={5 0 0 0mm}, clip]{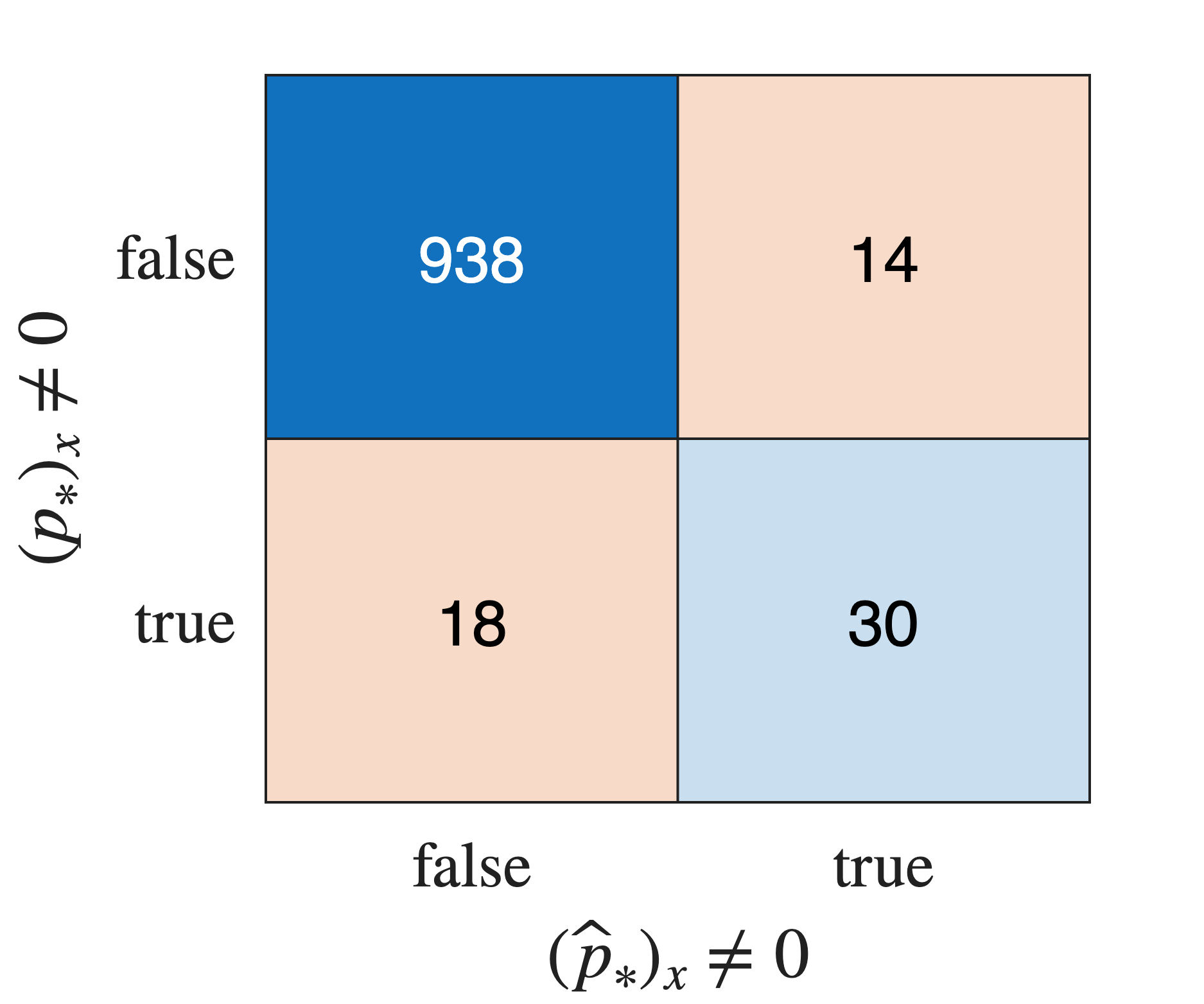}
    \includegraphics[width=.32\textwidth, trim={5 0 0 0mm}, clip]{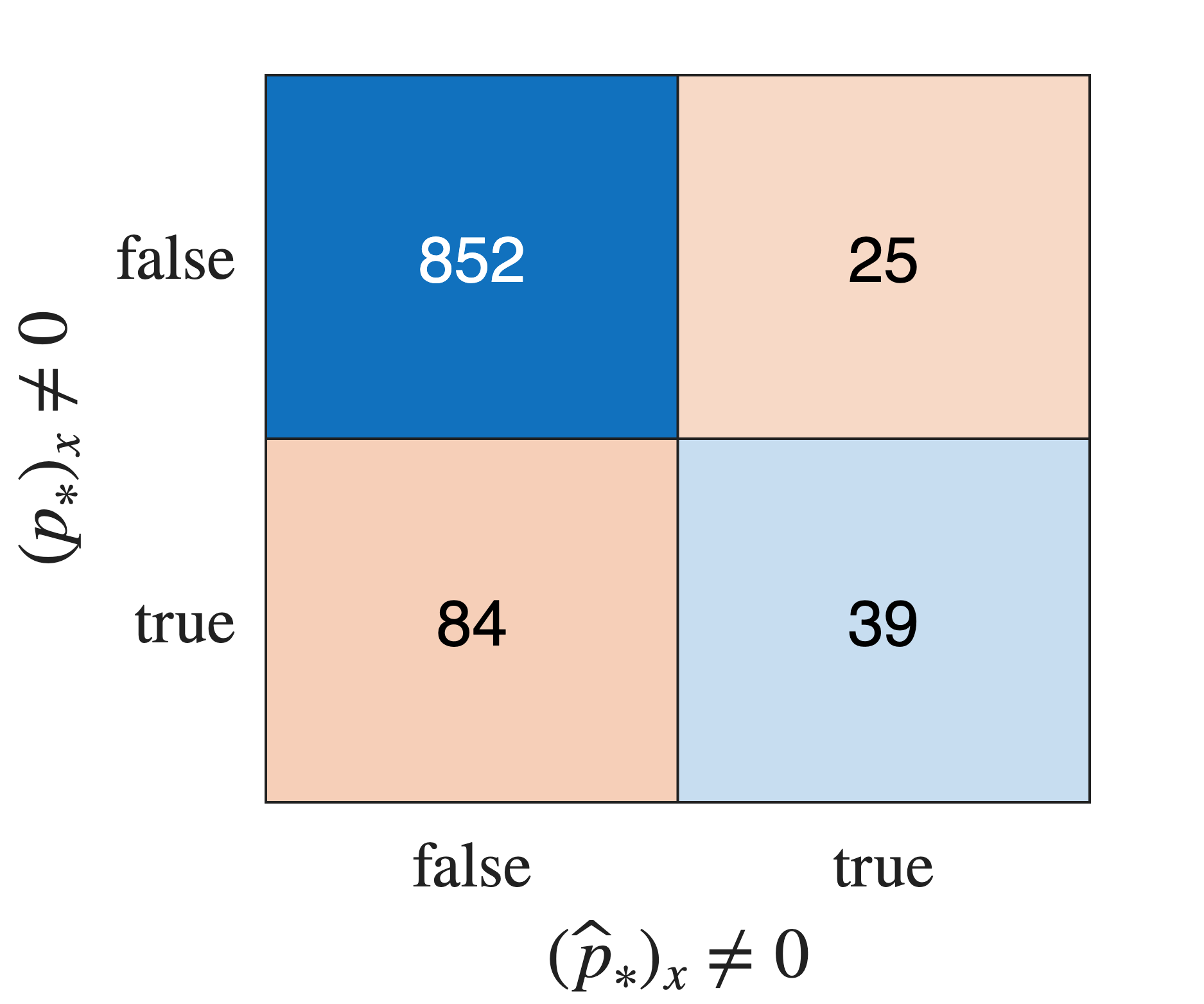}
        \caption{Top left: a uniform sample from the unit 2-ball, with peel distribution $p_*$ indicated by black disks and {\color{red}the neighborhood approximation $\hat p_*$ indicated by overlaid red circles}. The NED is very small. Lower left: the confusion matrix for zeros of $p_*$ and $\hat p_*$. Center and right: as in the left panels, but for (projections of) the 3- and 10-ball, respectively. 
        The heuristic and exact peel neighborhoods are the same except that dimension 10 case exhibits $2/1000$ discrepancies, as in Figures \ref{fig:rho_ball_10} and \ref{fig:rho_ball_histograms}.
        }
    \label{fig:neighborhoodPeelApproximation}
\end{figure}

\subsection{Local dimension estimation in sampled stratified manifolds}\label{sec:localDimension}

As a warmup, we first consider local dimension estimation for low-dimensional sampled manifolds. For $m \in \{2,3,4\}$, we sample $N = 10^m$ IID uniform points from the sphere $S^m$, embed the sample in the first dimensions of $\mathbb{R}^{10}$, and add isotropic Gaussian noise of expected norm $10^{-2}$ to every point. For various notions of local neighborhood, we compute local dimension estimates using expected simplex skewness of type a (ESSa) \cite{johnsson2014low}, representing the current state of the art \cite{binnie2025survey}.

The results of $k$NN estimates along these lines are shown in Figure \ref{fig:knn_sphere_dim}, which shows that reasonable values of $k$ lead to reasonable if noisy dimension estimates.

\begin{figure}[htbp]
    \centering
    \includegraphics[width=\textwidth, trim={0 125 0 0mm}, clip]{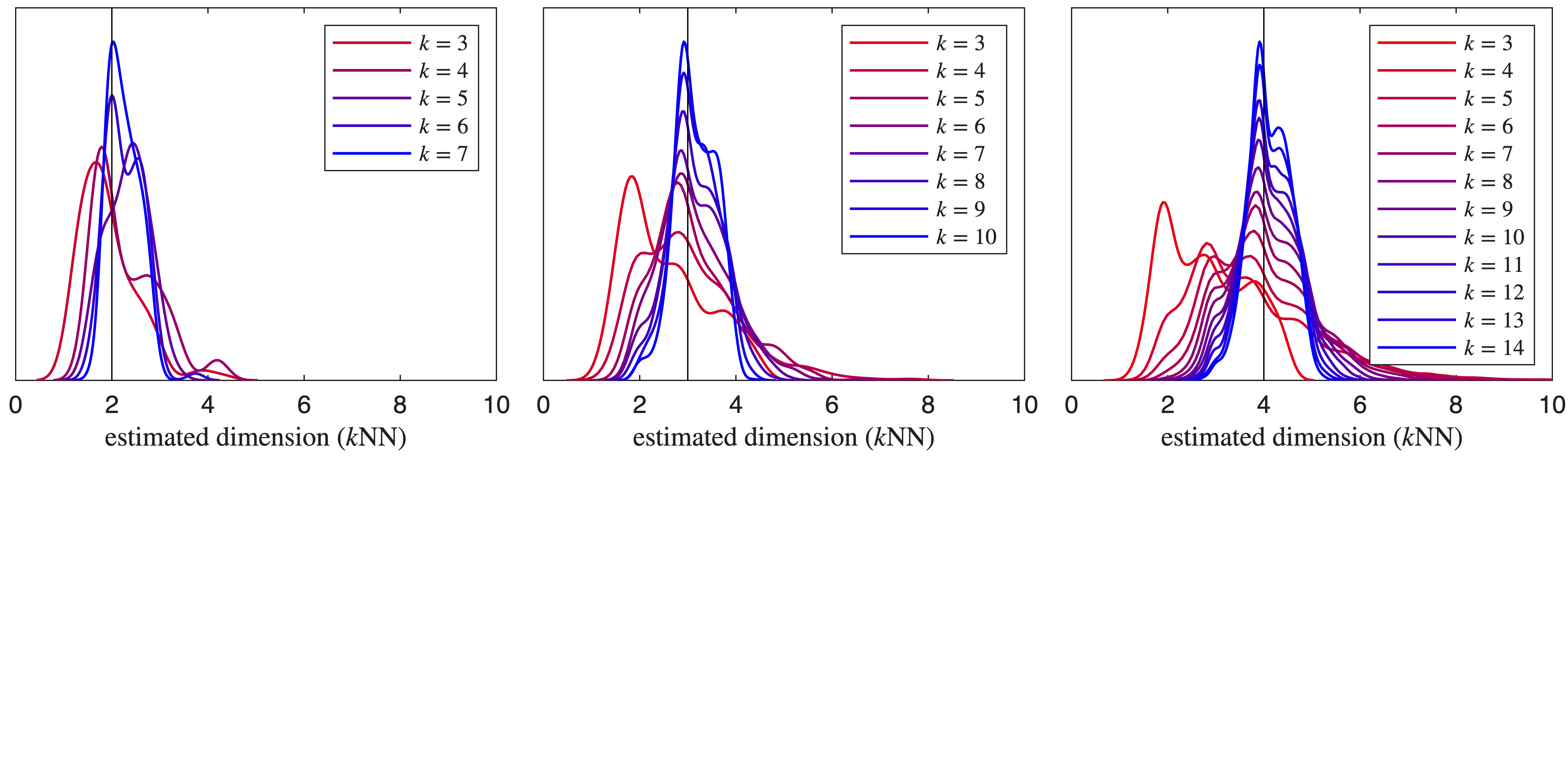}
        \caption{Left: kernel density estimates of estimated local dimension for $k$NN neighborhoods of a IID sample of size $N = 10^m$ from $S^m$ with $m = 2$, embedded in $\mathbb{R}^{10}$ with small isotropic Gaussian noise. $k$ varies from $3$ to $\lceil \log_2 N \rceil$. $m$ is indicated with a vertical line. Center and right panels: as in the left panel, but for $m = 3$ and $m = 4$, respectively. 
        }
    \label{fig:knn_sphere_dim}
\end{figure}

By comparison, the values obtained for peel neighborhoods are worse, though for related constructs they are competitive or better, as shown in Figure \ref{fig:local_sphere_dim}. The reason that peel neighborhoods are worse in this regard is because ESSa is concerned with angles, not distances \emph{per se}, and the peel neighborhood is in a sense the smallest ball that ``envelops'' its base point, so that the resulting dimension estimate is fatally biased. In other words, $\rho(x)$ is actually a reasonably hard \emph{lower} bound on the radius of a neighborhood suitable for estimating local dimension using angles.\footnote{In fact, an analysis along the lines of \S \ref{sec:aUsefulModel} suggests that ESSa will tend to saturate at $\approx 13$ on individual peel neighborhoods (cf. Figure \ref{fig:EssVsVGT50}), so larger neighborhoods are necessary. Such an analysis is possible because the ESSa statistic is a function of the relevant Gram matrix. 
}

This motivates larger constructs involving peel neighborhoods, e.g., the \emph{iterated peel neighborhood} $\nu_2(x) := \cup_{y \in \nu(x)} \nu(y)$ and the larger ball $B_{2\rho(x)}(x)$. The latter produces excellent estimates, but scales exponentially with data dimension in practice. On the other hand, the iterated peel neighborhood is easy, relatively small, and efficient to construct. While the iterated peel neighborhood tends to underestimate all but the smallest dimensions, it still performs adequately when $k$NN methods fail spectacularly, as we shall see below.

\begin{figure}[htbp]
    \centering
    \includegraphics[width=\textwidth, trim={0 125 0 0mm}, clip]{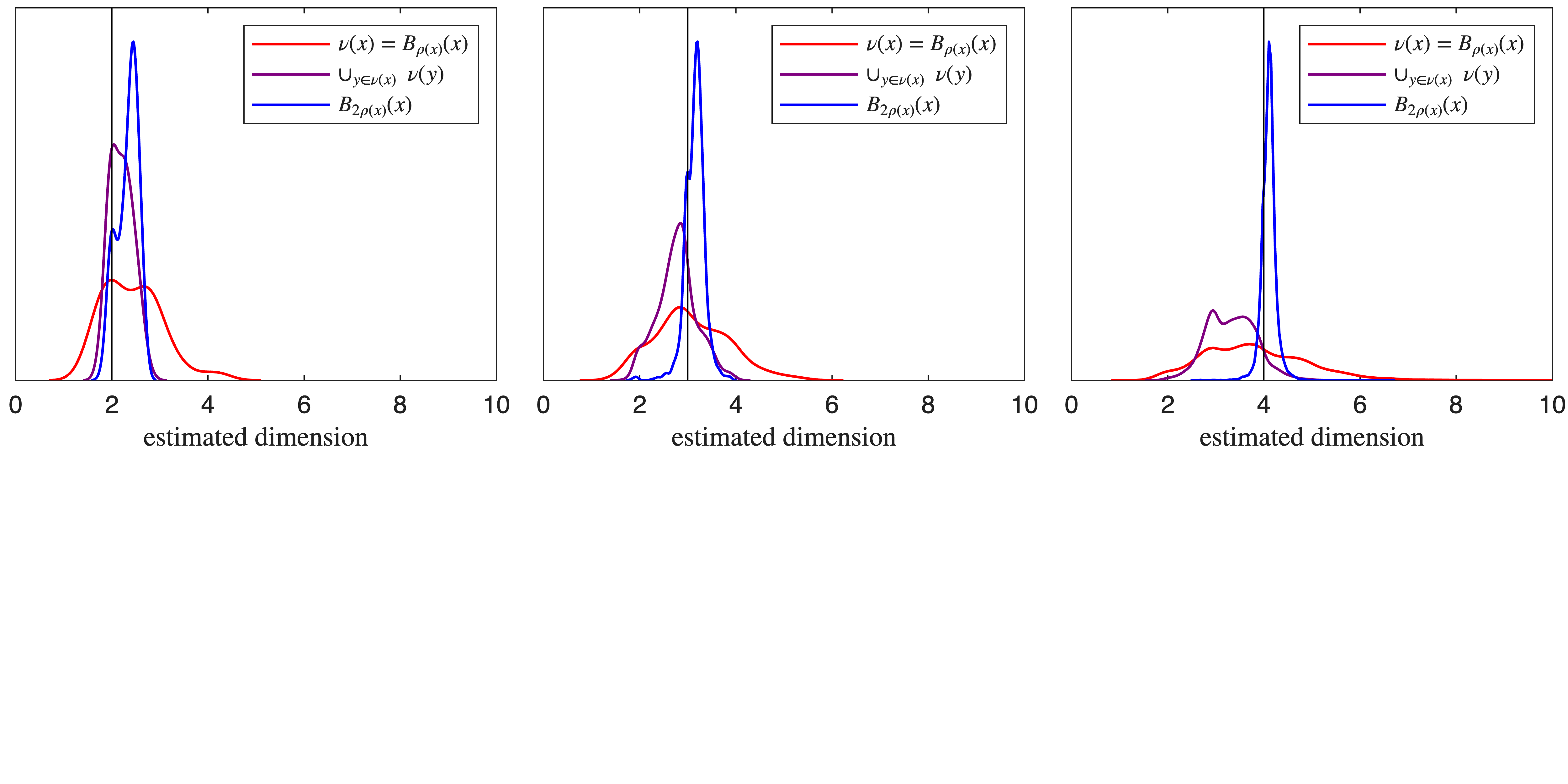}
        \caption{As in Figure \ref{fig:knn_sphere_dim} but for three different neighborhood families: {\color{red}(red) peel neighborhoods $\nu(x) = B_{\rho(x)}(x)$}; {\color[rgb]{0.5, 0.0, 0.5}(purple) iterated peel neighborhoods $\cup_{y \in \nu(x)} \nu(y)$}; and {\color{blue}(blue) $B_{2\rho(x)}(x)$}. 
        The heuristic and exact peel neighborhoods are identical except for one neighborhood in dimension 4, with that heuristic neighborhood radius $0.7705$ times the exact value. This has no visible effect on the plots.
        }
    \label{fig:local_sphere_dim}
\end{figure}
     
Figure \ref{fig:stratified_dim} shows that $k$NN neighborhoods slightly overestimate local dimension on a sample from a stratified manifold formed from a circle, a 2-ball, and a 3-ball; peel neighborhoods give very noisy estimates; iterated peel neighborhoods slightly underestimate local dimension; and ``double-radius'' neighborhoods of the form $B_{2\rho(x)}(x)$ give very accurate estimates of local dimension, but these are expensive to compute. 

We can compare the various \emph{relative} local dimension estimates here using a Kendall $\tau$-b coefficient relative to ground truth local dimension: the $k$NN coefficient is $0.308$; the peel neighborhood coefficient is $0.223$; the iterated peel neighborhood coefficient is $0.287$; and the $B_{2\rho(x)}(x)$ coefficient is $0.308$. That is, the iterated peel neighborhood coefficient is nearly as good as the best competition in this regard.

\begin{figure}[htbp]
    \centering
    \includegraphics[width=.24\textwidth, trim={5 30 5 5mm}, clip]{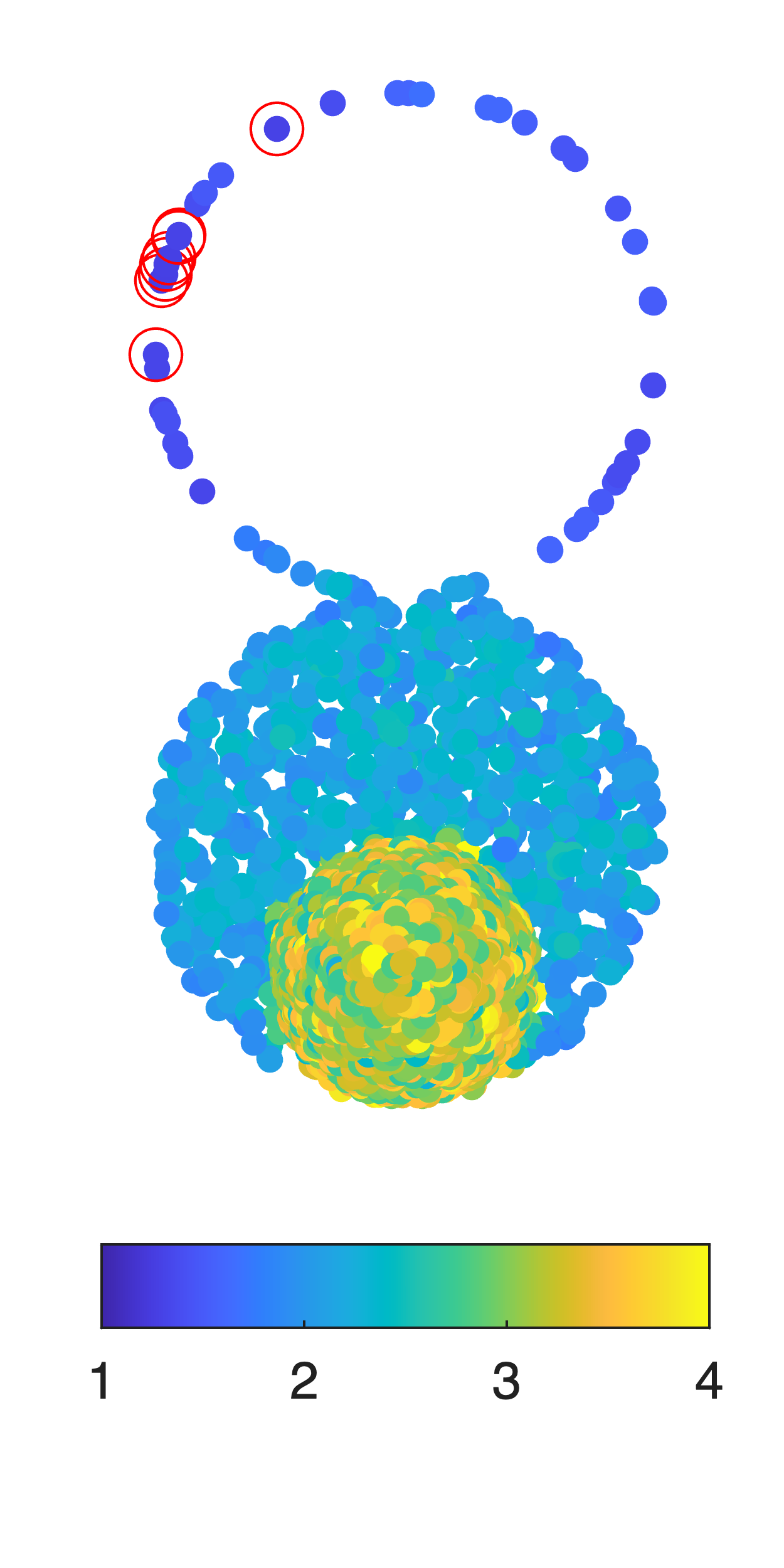}
    \includegraphics[width=.24\textwidth, trim={5 30 5 5mm}, clip]{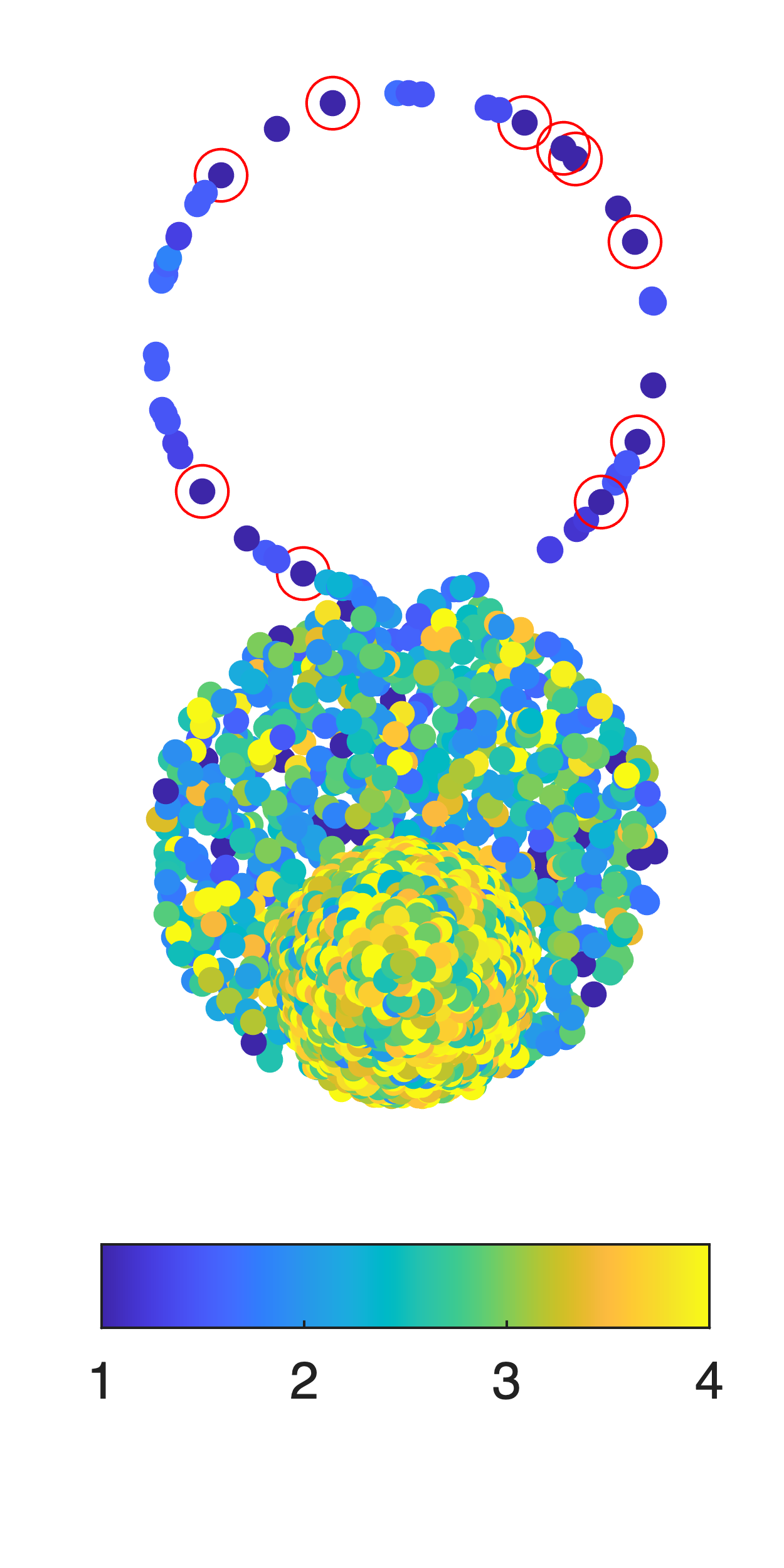}
    \includegraphics[width=.24\textwidth, trim={5 30 5 5mm}, clip]{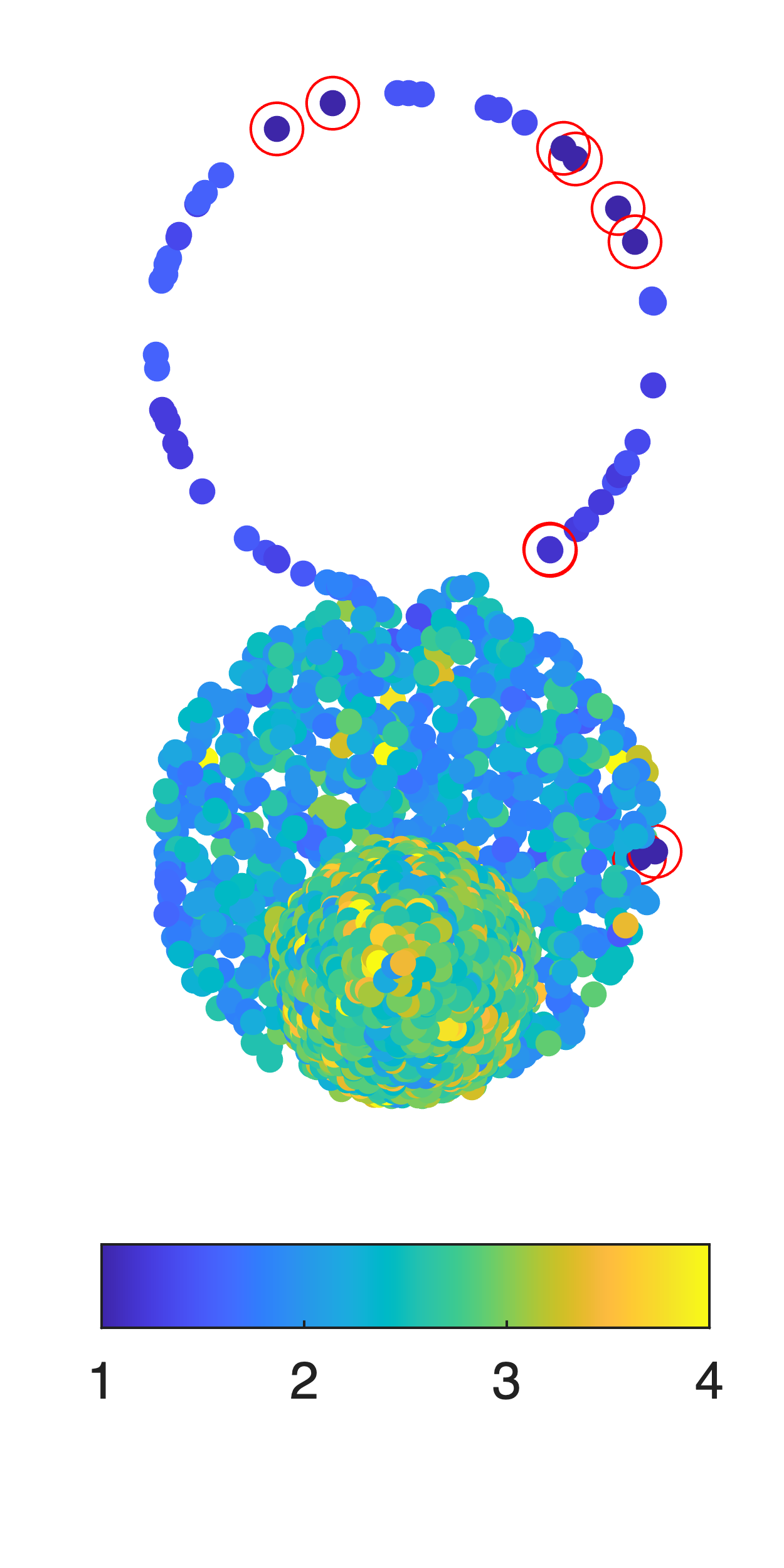}
    \includegraphics[width=.24\textwidth, trim={5 30 5 5mm}, clip]{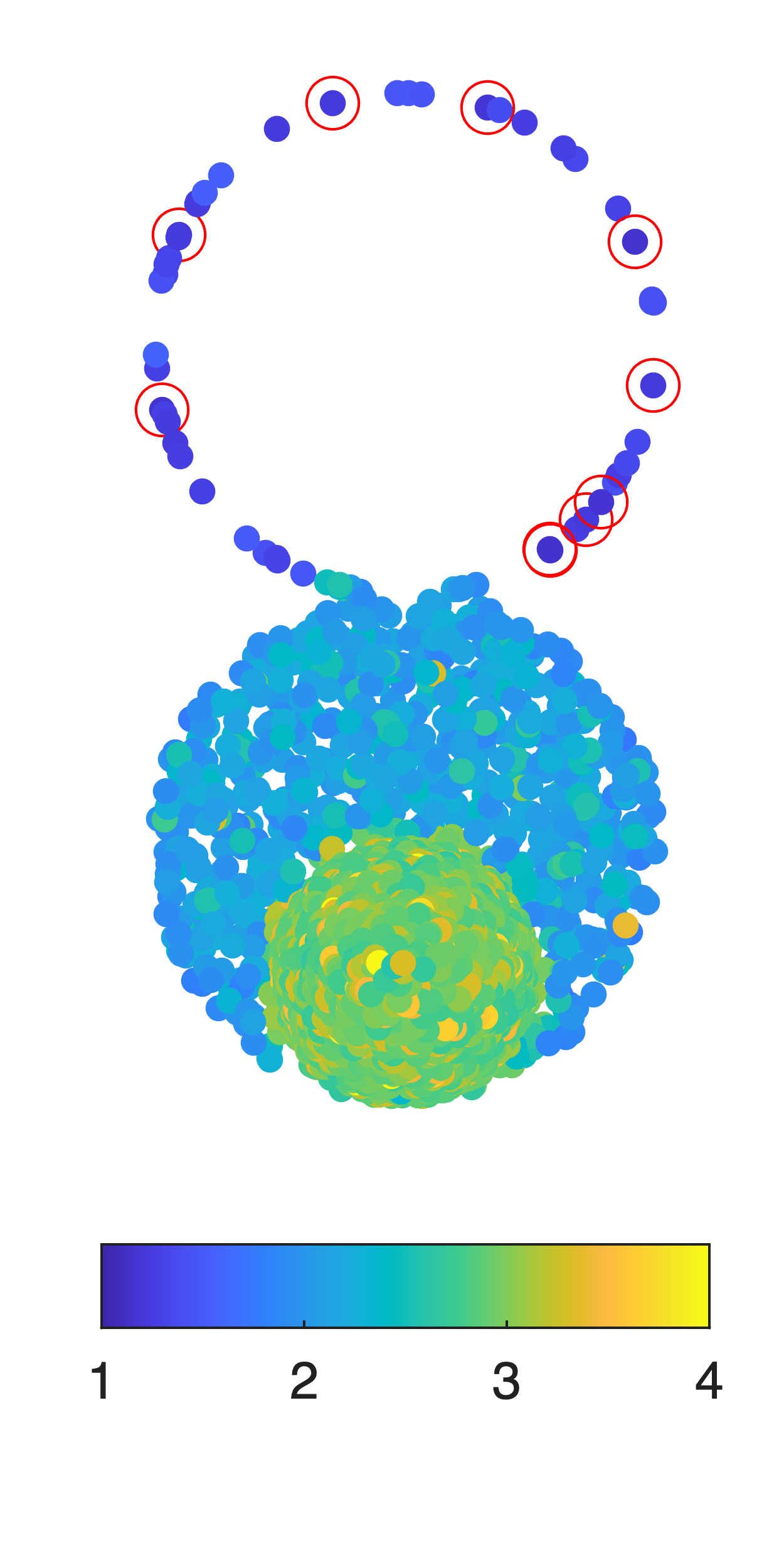}
        \caption{Left to right: local dimension estimates for $k$NN neighborhoods with $k = \lceil \log_2 N \rceil = 14$ with $N = 11294$; for peel neighborhoods $\nu(x) = B_{\rho(x)}(x)$; for iterated peel neighborhoods $\nu_2(x) = \cup_{y \in \nu(x)} \nu(y)$; and for $B_{2\rho(x)}(x)$. Ambient dimension and noise are the same as in Figures \ref{fig:knn_sphere_dim} and \ref{fig:local_sphere_dim}. 
        {\color{red}The 10 points with least estimated dimension are circled in red.} Only one heuristic peel neighborhood differs from the exact one: its radius is $\approx 0.9951$ times the exact value; this has no further discernible effect.
        }
    \label{fig:stratified_dim}
\end{figure}

While a simple $k$NN-based approach can lead to good local dimension estimates for stratified manifolds where the strata are reasonably sampled and dimensions vary slowly, such an approach also fails dramatically as behavior deteriorates even slightly, as Figure \ref{fig:strand} shows: the $B_{2\rho(x)}$ estimates also fail on the true low-dimensional points.

\begin{figure}[htbp]
    \centering
    \includegraphics[width=.24\textwidth, trim={5 30 5 5mm}, clip]{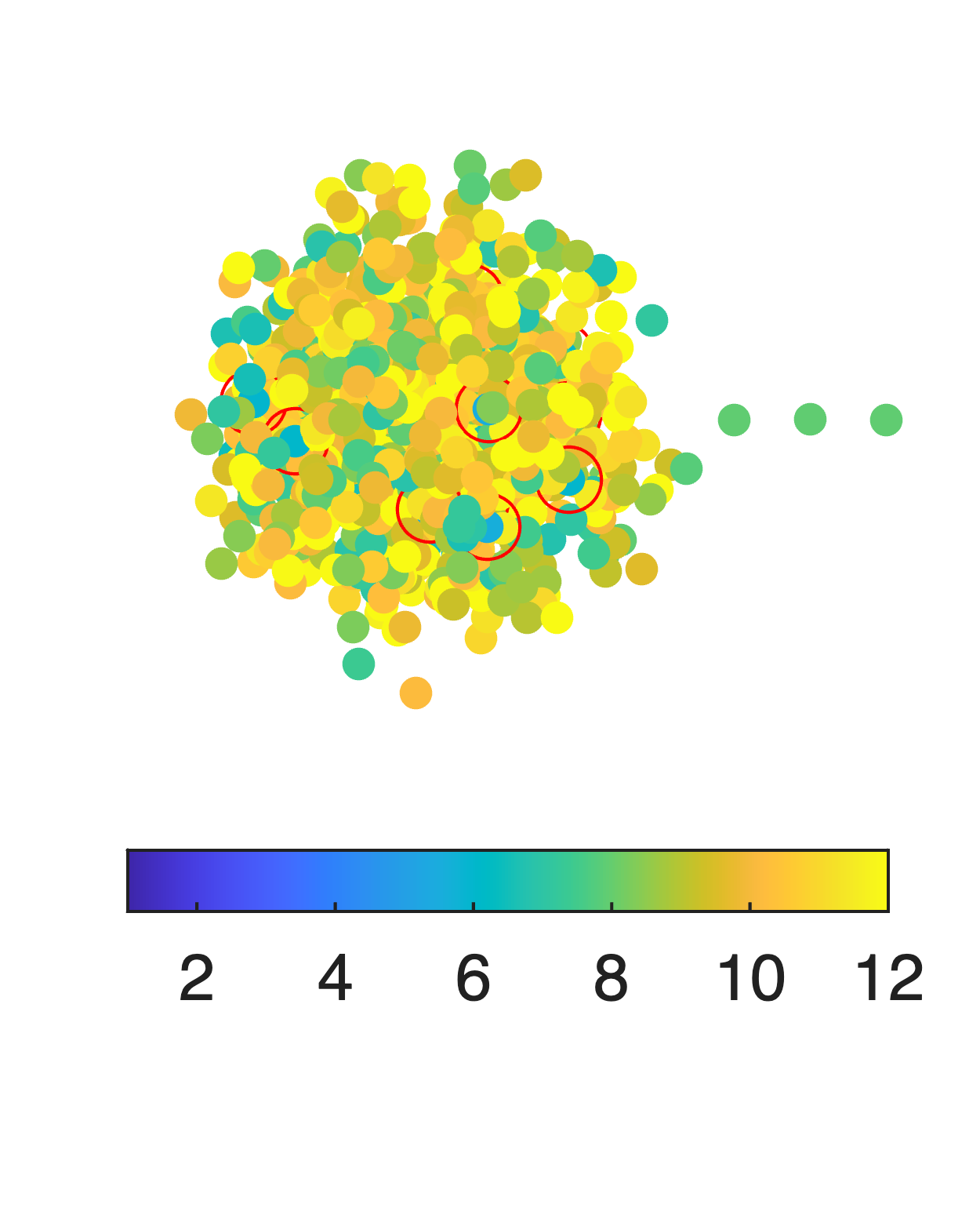}
    \includegraphics[width=.24\textwidth, trim={5 30 5 5mm}, clip]{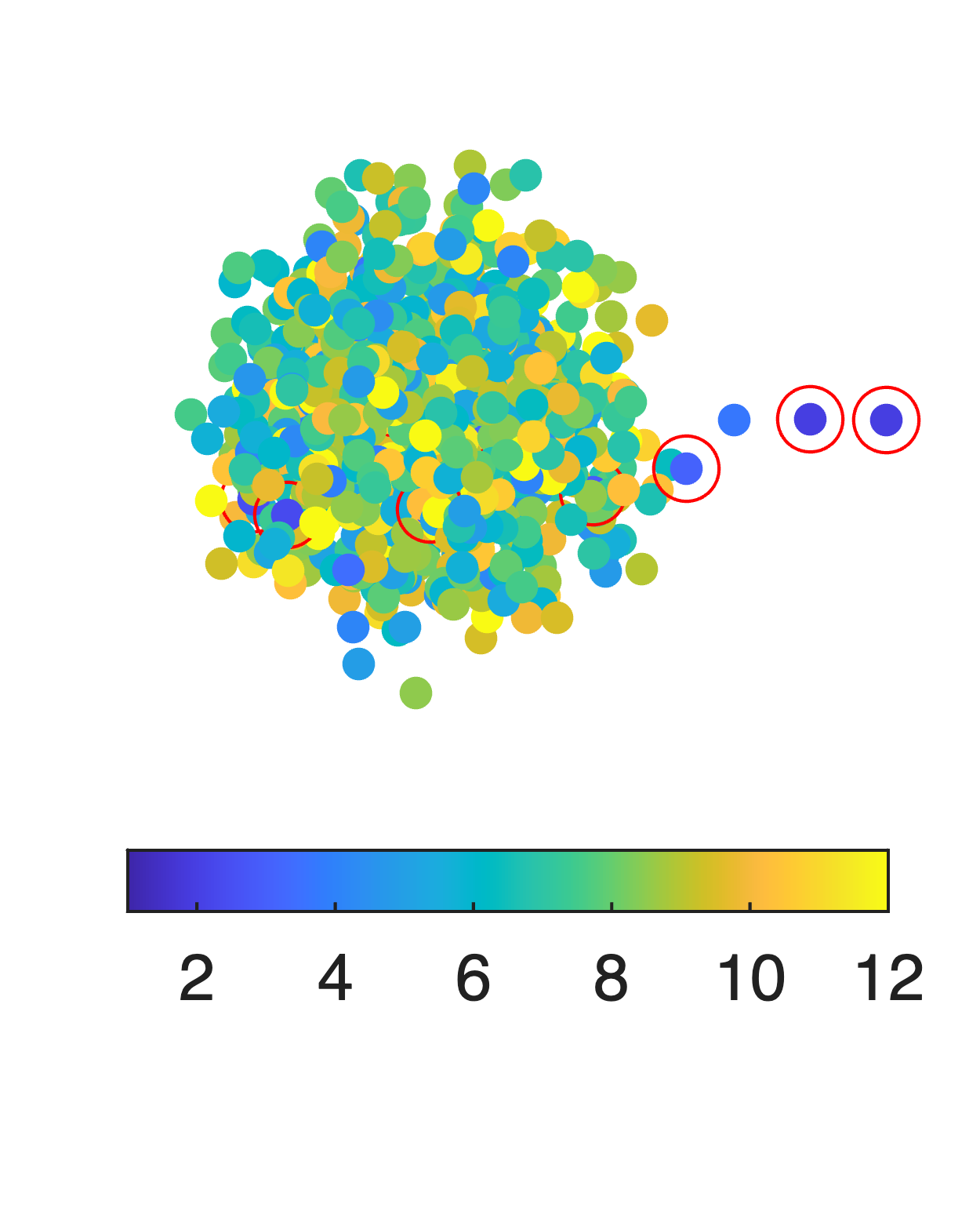}
    \includegraphics[width=.24\textwidth, trim={5 30 5 5mm}, clip]{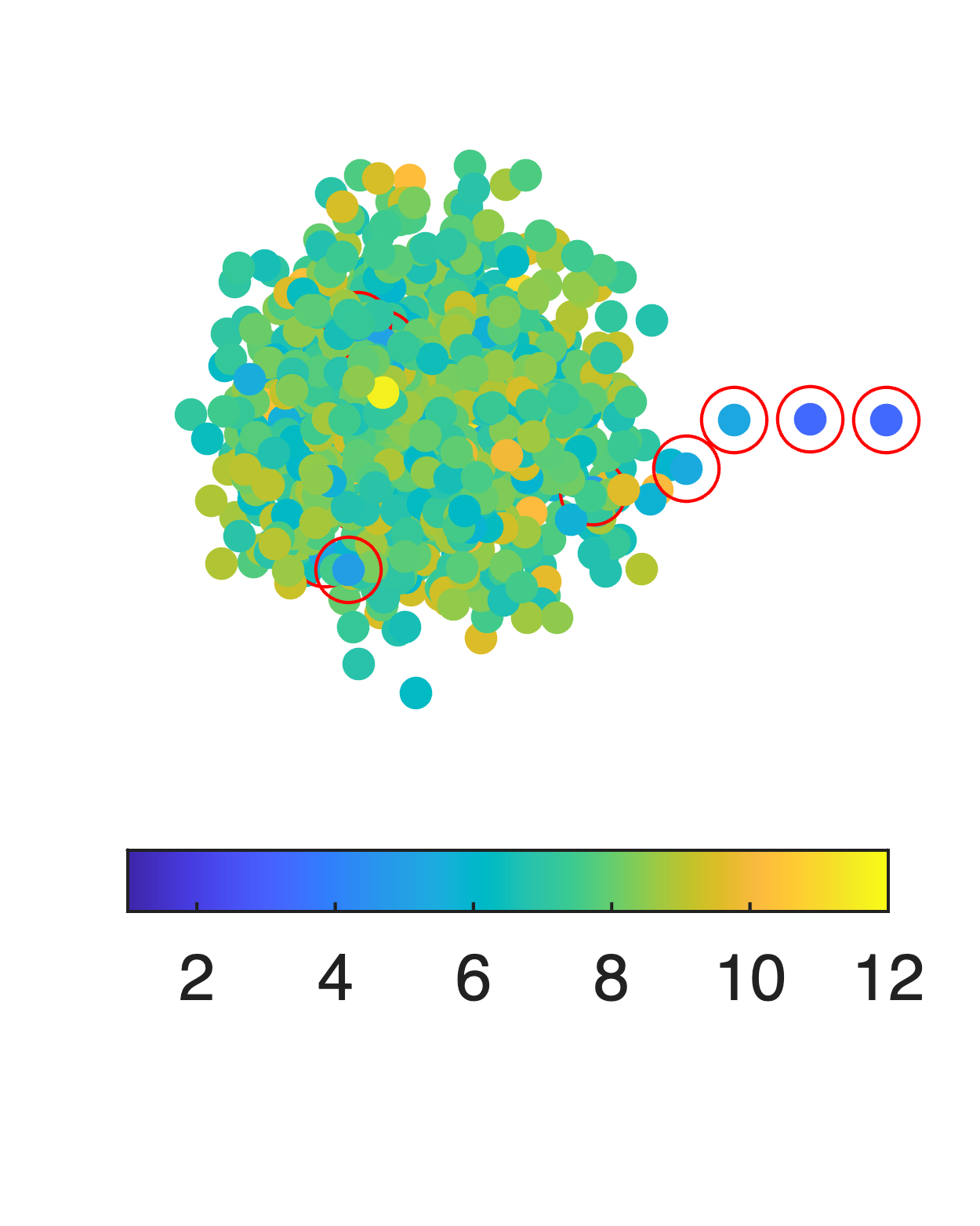}
    \includegraphics[width=.24\textwidth, trim={5 30 5 5mm}, clip]{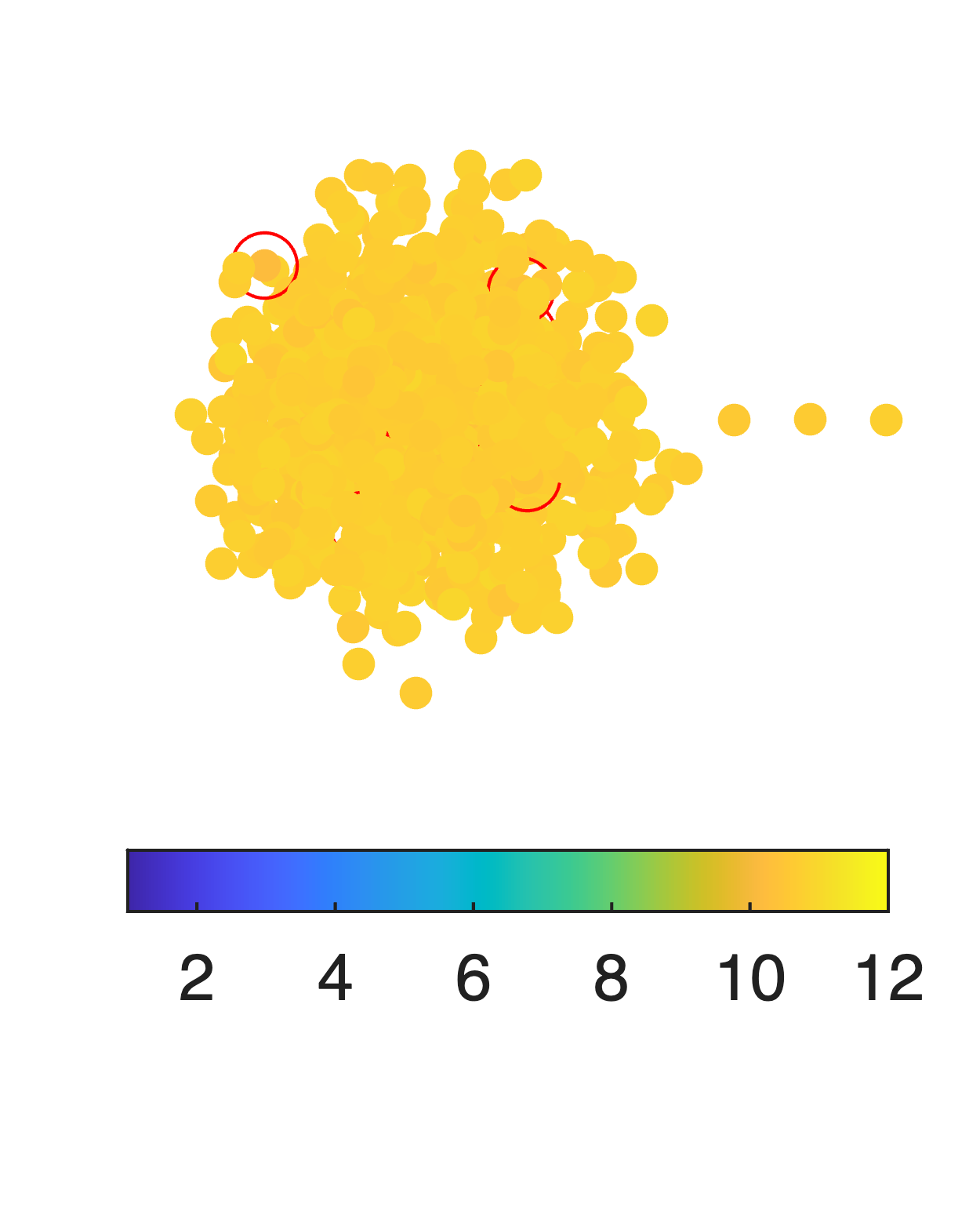}
        \caption{As in Figure \ref{fig:stratified_dim}, but for 1000 points sampled uniformly from $S^{10}$, along with three nearby points. Here, the data are embedded in $\mathbb{R}^{100}$ with small isotropic Gaussian noise added. 
        The exact neighborhoods and their heuristic approximations are always identical here.
        For each of the three points on the right, the estimated dimension with $k$NN neighborhoods is higher than that for over 10 percent of the points sampled from the sphere. The Kendall $\tau$-b coefficient between the estimates in the middle two panels is $0.3410$, while the coefficients between the first two panels and between the first and third panels are respectively $0.2409$ and $0.1976$.
        }
    \label{fig:strand}
\end{figure}

\subsubsection{Gradient norms of dimension estimates}\label{sec:gradients}

Many data of practical interest appear to be effectively sampled from stratified manifolds \cite{brown2023verifying,aamari2024theory,robinson2024structure,wang2024cw,curry2025exploring,li2025unraveling,sun2025sparsification}. Meanwhile, state of the art methods for identifying singularities in stratified manifolds such as \cite{stolz2020geometric,von2023topological,lim2025hades} still leave much performance to be desired. We have seen that for stratified data it is imperative to avoid $k$NN techniques, but peel neighborhoods are relatively usable.

With this in mind, for $X \subset \mathbb{R}^m$ finite, $f : X \rightarrow \mathbb{R}$, and $\{x\} \subset \lambda(x)$ for all $x \in X$, we can form the gradient estimate
\begin{equation}\label{eq:gradientEstimate}
    \hat \nabla_\lambda f(x) := \sum_{x' \in \lambda(x)-\{x\}} p_*(d|_{\lambda(x)})_{x'} \cdot \frac{f(x')-f(x)}{|x'-x|} \frac{x'-x}{|x'-x|}.
\end{equation}
Writing $\nu_0(x) := \{x\}$ and $\nu_j(x) := \cup_{y \in \nu_{j-1}(x)} \nu(y)$ for the $(j-1)$-fold iterated peel neighborhood, in practice we will restrict consideration to $\nu_1(x) = \nu(x)$ and $\nu_2(x)$. We write $\hat \nabla_j := \hat \nabla_{\nu_j}$ for convenience below.

For $f : X \rightarrow \mathbb{R}$ and $\alpha$ a summary statistic (e.g., mean, median, max), let $[f(y): y \in Y]$ be the multiset of values taken by $f(y)$ for $y \in Y \subseteq X$, so that $\alpha([f(y): y \in Y])$ is the summary statistic of $f$ on $Y$. Write $\hat m_j$ for a (local) dimension estimate of $\nu_j(x)$, nominally using ESSa. We demonstrate below that the \emph{gradient norm score}
\begin{equation}\label{eq:gradientNorm}
    s_j(x) := \left \| \hat \nabla_j \log \alpha([\hat m_j(x'): x' \in \nu_j(x)]) \right \|
\end{equation} 
with $\alpha = \textnormal{median}(\cdot)$ is capable of identifying singularities in sampled stratified manifolds, with $j = 2$ improving over $j = 1$ at the cost of (not infeasibly) more computation. The idea behind \eqref{eq:gradientNorm} is that the norm of the log-gradient measures the size of relative changes in median local dimension estimates. 
\footnote{The logarithm in \eqref{eq:gradientNorm} is inessential and included (only) because we decided to highlight relative \emph{versus} absolute dimension changes. We performed exactly the same experiments with and without the logarithm. We do not show the latter data because there is very little quantitative difference between it and the former, and there is no qualitative difference.}
Using iterated peel neighborhoods provides a robust local dimension estimate and corresponding gradient estimate via \eqref{eq:gradientEstimate}. Figure \ref{fig:nbhd2_pinchedTorus} shows $\hat m_2$ and $s_2$ on samples from a pinched torus.

\begin{figure}[htbp]
    \centering
    \includegraphics[width=.48\textwidth, trim={65 0 20 0mm}, clip]{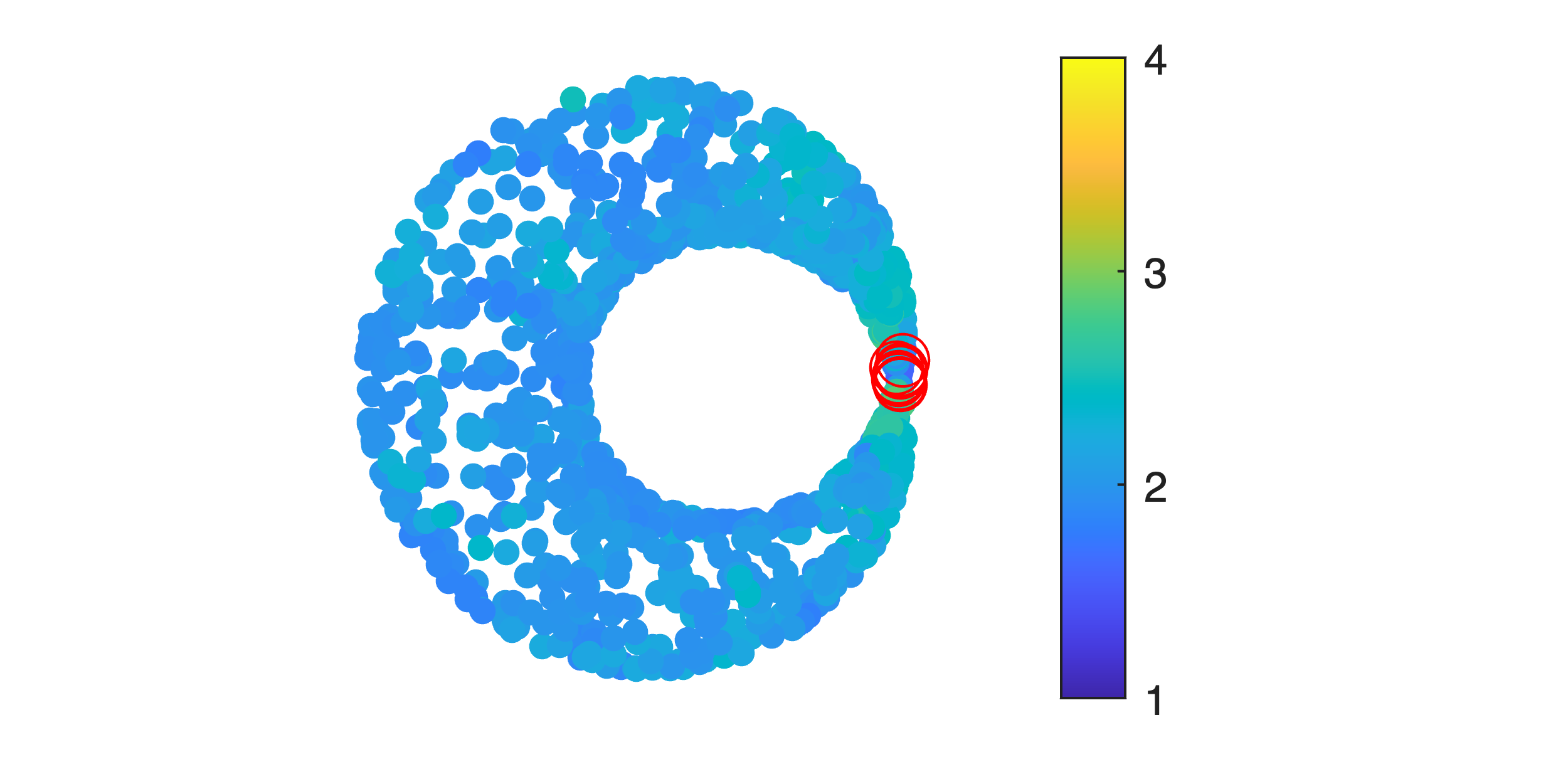}
    \includegraphics[width=.48\textwidth, trim={65 0 20 0mm}, clip]{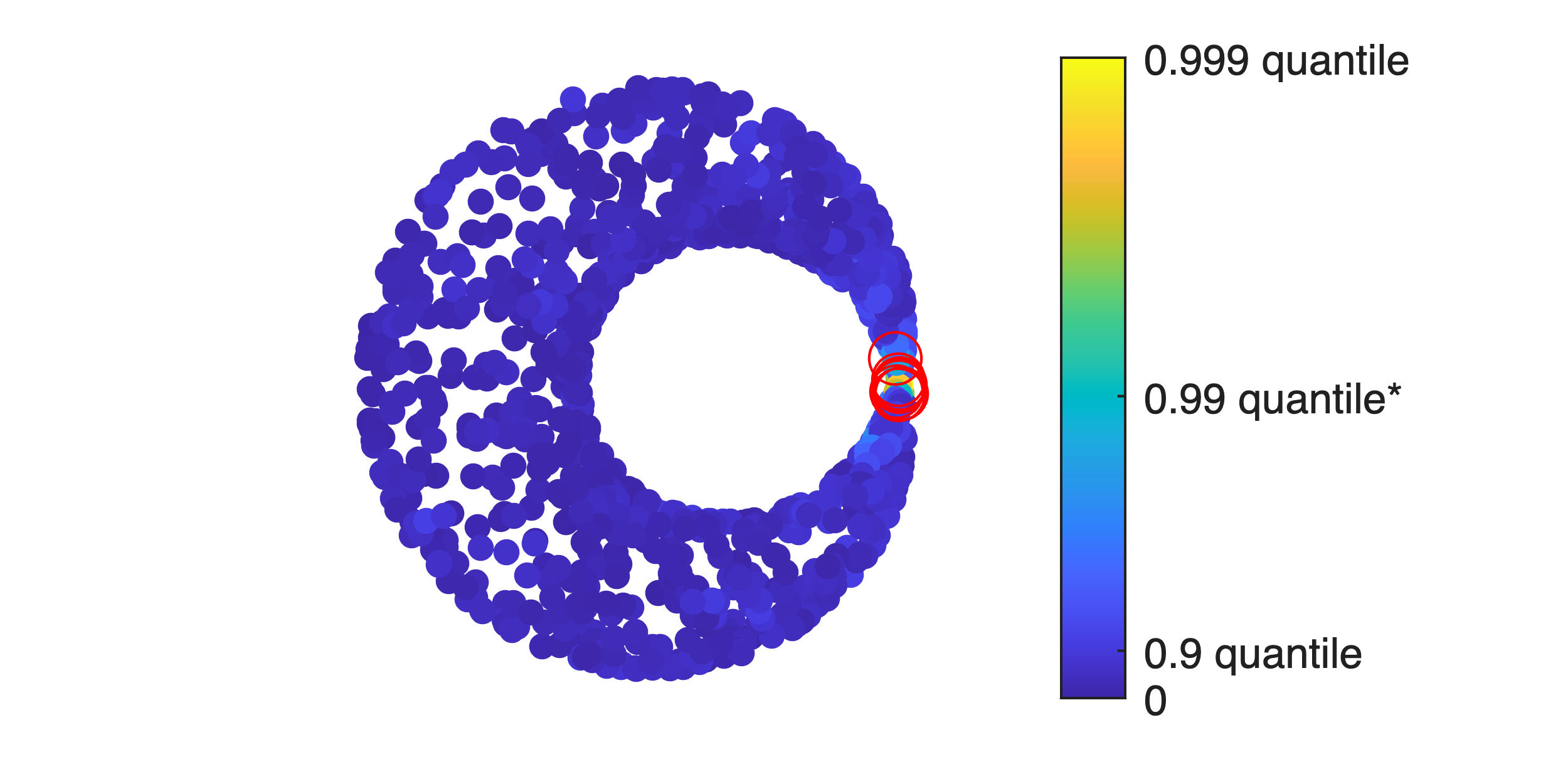}
        \caption{Left: $\hat m_2$ using ESSa for a sample from a pinched torus embedded in $\mathbb{R}^{10}$ with small but nonzero isotropic Gaussian noise added. {\color{red}The 10 lowest values are circled in red}. Right: $s_2$, with {\color{red}the values above the 0.99 quantile indicated by a * on the colorbar circled in red}. 
        The exact neighborhoods and their heuristic approximations are always identical here.
        }
    \label{fig:nbhd2_pinchedTorus}
\end{figure}

For a more controlled and statistically meaningful experiment, we want to be able to adjust the dimensionality of strata and (optionally) to make the expected spacing between points the same on different strata. The former requirement is easily met by sampling from the union of a unit sphere and an interval. The latter requirement is more demanding. To satisfy it, we need to find the distribution of angular distances between nearest neighbors from a uniform IID sample on $S^{m-1}$ of size $N$. Suppose that the angular distance between two uniform IID points has cumulative distribution $F(\theta)$. Then the cumulative distribution that any two given points in the sample are nearest neighbors is $1-(1-F(\theta))^{N-1}$.
 
Let $x \sim U(S^{m-1})$ and let $x_*$ be any fixed unit vector in $\mathbb{R}^m$. Now $$F(\theta) = \mathbb{P}(\langle x,x_* \rangle \le \theta) = I_{\sin^2 \theta} \left (\tfrac{m-1}{2},\tfrac{1}{2} \right ),$$ where the regularized incomplete beta function is indicated. \cite{li2011concise} For $N \gg 1$, the probability mass will be concentrated near $\theta = 0$. To leading order, $I_t(a,b) = t^a/(aB(a,b)) + o(t^{a+1})$, where the beta function is indicated. So $$F(\theta) \approx \frac{\theta^{m-1}}{(m-1) \cdot B(\tfrac{m-1}{2},\tfrac{1}{2})}.$$

Now $1-(1-F)^{N-1} \approx 1-\exp(-[N-1]F)$. The right hand side is approximately a Weibull distribution with shape parameter $m-1$ and scale parameter $$\left ( \frac{m-1}{N-1} \cdot B(\tfrac{m-1}{2},\tfrac{1}{2}) \right )^{1/(m-1)}.$$ The mean of this Weibull distribution is therefore $$\Gamma \left ( \frac{m}{m-1} \right ) \cdot \left ( \frac{m-1}{N-1} \cdot B(\tfrac{m-1}{2},\tfrac{1}{2}) \right )^{1/(m-1)}.$$ This expression gives the approximate expected spacing between nearest neighbors for a uniform IID sample of size $N$ on $S^{m-1}$, as desired.

There are now two reasonable ways to glue a sample from an interval to a sample from a sphere while preserving the expected spacing. We can either attach a zero-dimensional singular point at the intersection of the sphere and the interval, or not. Both choices are easy to implement, so we do both. Likewise, we can either choose to match the expected spacing in the interval, or take an interval of fixed length. Again, both choices are easy to implement, so we do both. The expected spacing between $N'$ uniform IID points on an interval of length $L$ is $L/(N'+1)$. Given $m$, $N$, and $N'$, matching the expected spacing just means choosing $L$ so that this expected spacing equals the mean of the Weibull distribution above. Separately, we take $L = 1$. 

Figures \ref{fig:hair_ball_dim_11_attach_false_match_false}-\ref{fig:hair_ball_dim_11_attach_true_match_true} show that the gradient norm score $s_2$ is consistently able to perform comparably to or better than HADES \cite{lim2025hades} for detecting actual stratification-induced singularities. In particular, when strata have equal inter-point spacing, $s_2$ clearly outperforms HADES close to their intersection. Furthermore, unlike $s_2$, HADES shows a visible bias towards global outliers that can be induced by very few points. 

For each of the four experiments illustrated in Figures \ref{fig:hair_ball_dim_11_attach_false_match_false}-\ref{fig:hair_ball_dim_11_attach_true_match_true}, we ran both tools 100 times. 
\footnote{All of the experiments in the paper were performed in MATLAB: we used a wrapper here for the HADES Python code.
}
We ran HADES 10 times in fully automatic mode with the hyperparameter search provided as an example (i.e., 5 values of radius and 3 PCA thresholds), then used the resulting median parameters for a subsequent 90 runs. The 10 fully automatic runs took about 365 seconds each, and the 90 other runs took about 25 seconds each, with variations of $\approx$ 1-2\%. Meanwhile, the corresponding $s_2$ calculations took about 0.6 seconds, with very little variation. That is, $s_2$ was conservatively over 40 times faster to compute than the HADES score, while the experiments themselves were roughly 100 times faster for $s_2$ than for the HADES score because of the (amortized) hyperparameter search overhead.

For lower data dimensions of bulk strata $S^m$, both techniques experience degradations in performance, as detailed in Appendix \S \ref{sec:analogues54}.
Peel neighborhoods considered broadly still enable competitive singularity detection capability, with faster runtime.

\begin{figure}[htbp]
    \centering
    \includegraphics[width=.48\textwidth, trim={0 50 0 15mm}, clip]{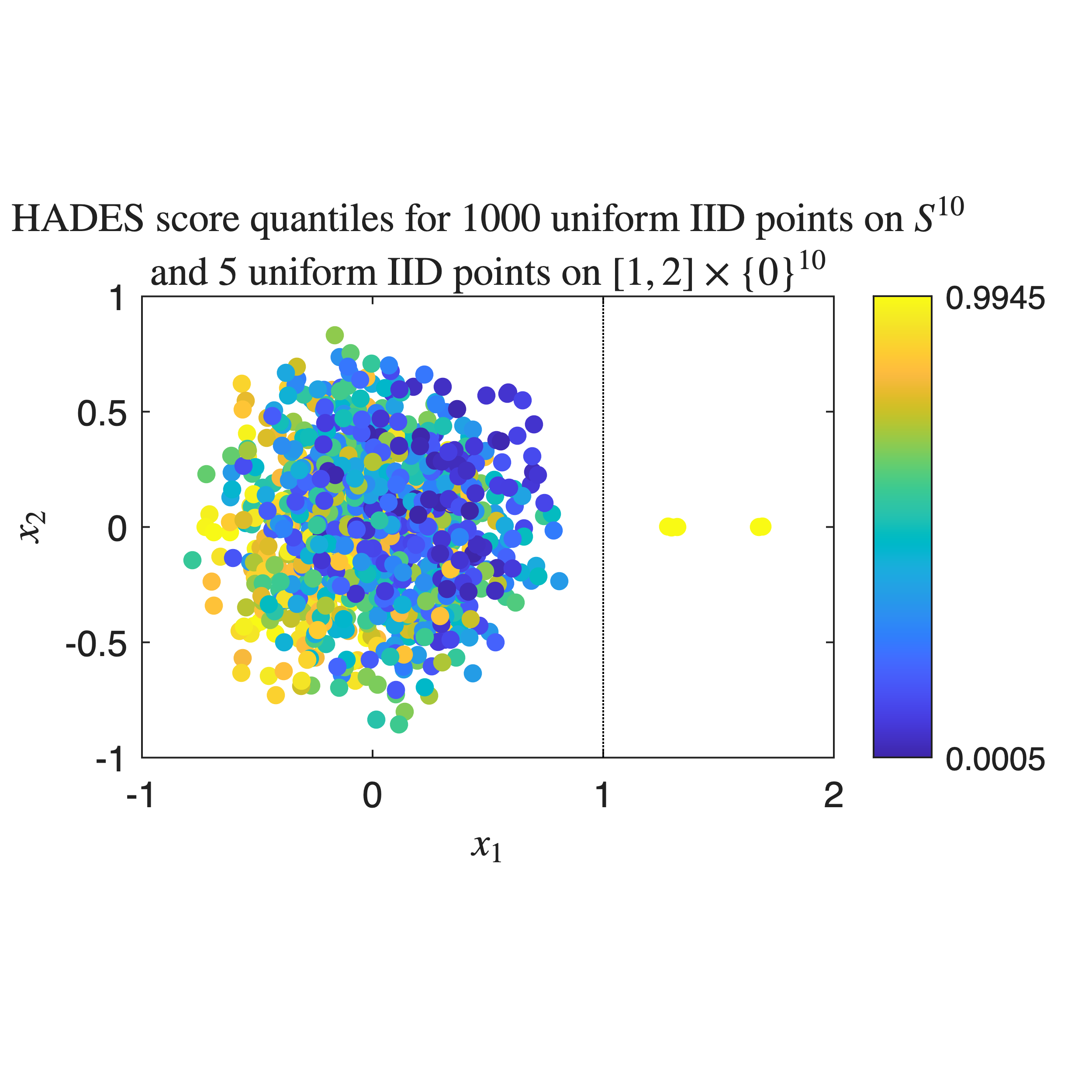}
    \includegraphics[width=.48\textwidth, trim={0 50 0 15mm}, clip]{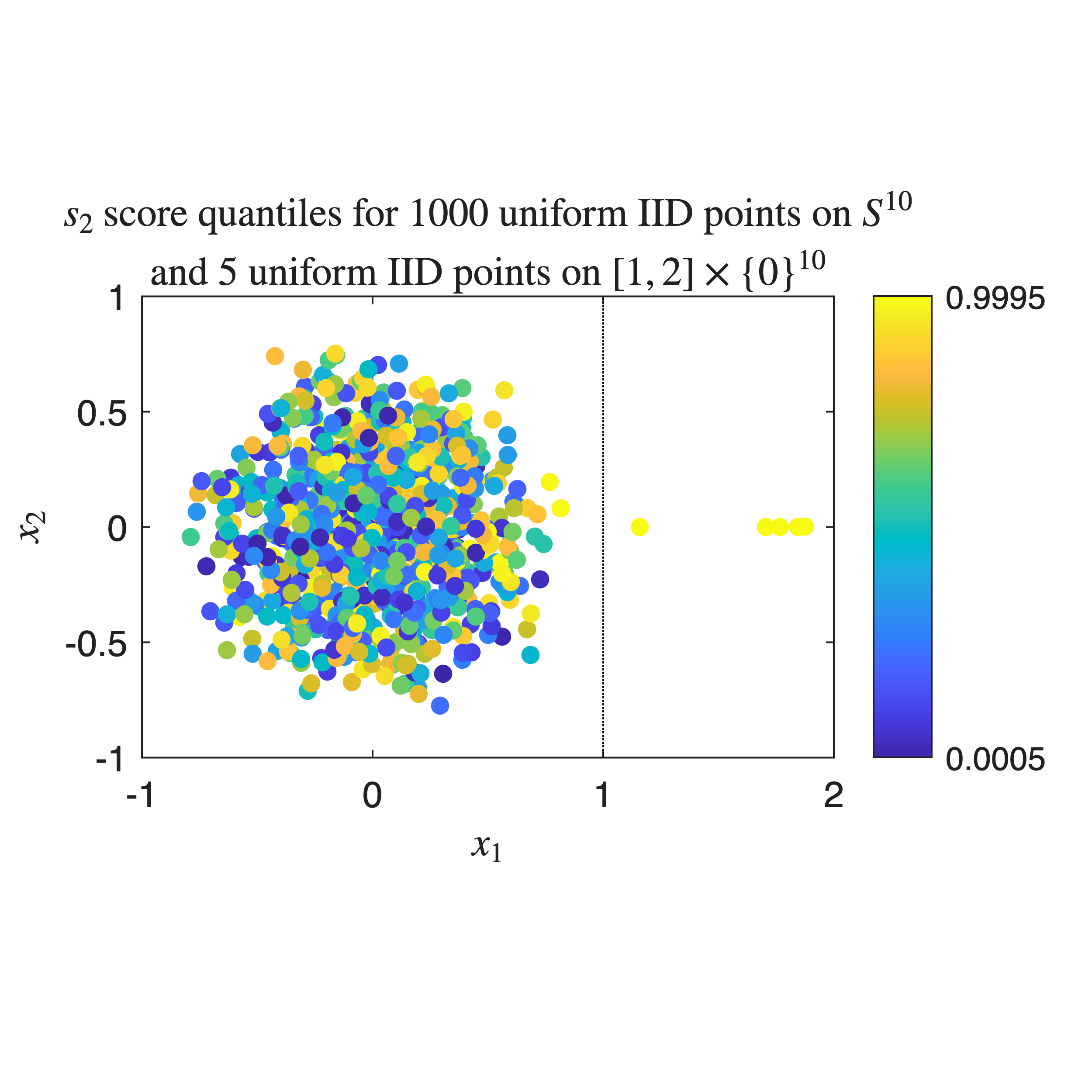}
    \includegraphics[width=.48\textwidth, trim={0 0 0 0mm}, clip]{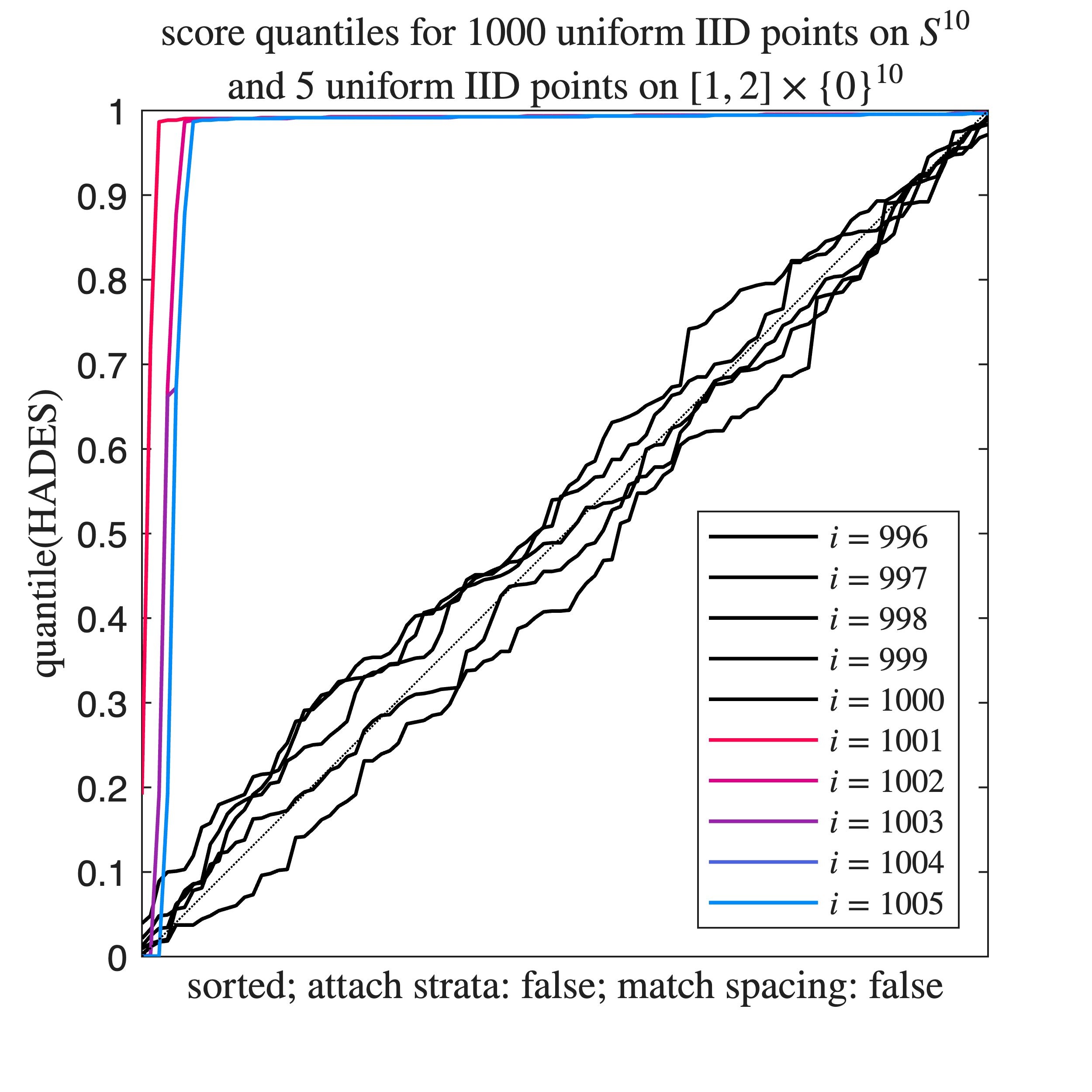}
    \includegraphics[width=.48\textwidth, trim={0 0 0 0mm}, clip]{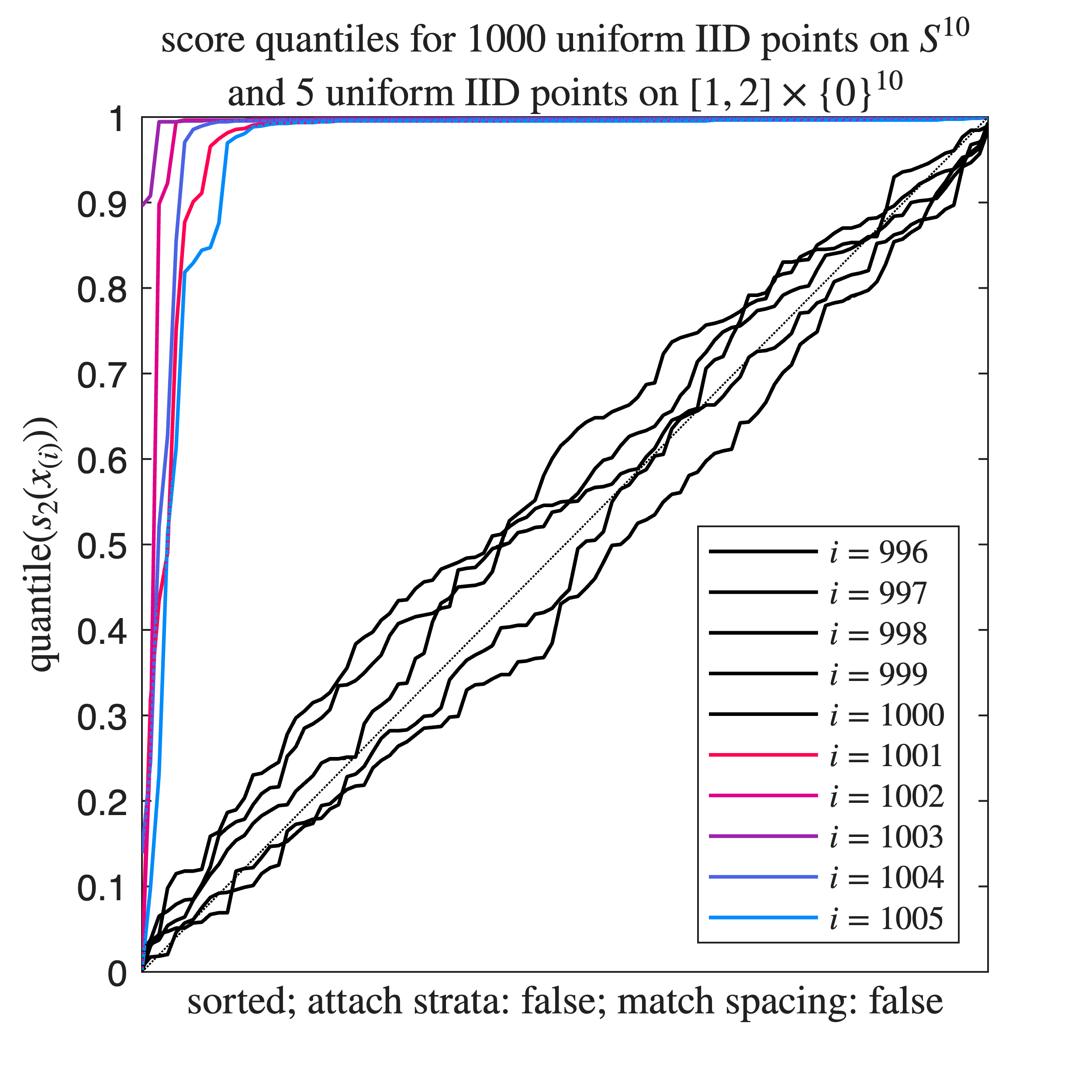}
        \caption{Top left: HADES score quantiles for a sample of size 1000 from $S^m$ for $m = 10$ and an ``unattached'' sample of size 5 from an interval of unit length, plus small but nonzero isotropic Gaussian noise after embedding in $\mathbb{R}^{100}$. Note that the scores are highly nonuniform on the sphere stratum. Top right: as in the top left panel, but for $s_2$. Note that unlike the HADES score, $s_2$ does not show any obvious nonuniformity on the sphere stratum. Bottom left: HADES score quantiles computed over 100 sample realizations for five points on the sphere stratum (in black) and the five points on the interval stratum, indexed in order of their distance from the sphere. Bottom right: as in the bottom left panel, but for $s_2$. 
        The $s_2$ plots actually show both exact and heuristic peel results (the latter as dotted lines), which differ slightly in a few places but are visually indistinguishable.
        }
    \label{fig:hair_ball_dim_11_attach_false_match_false}
\end{figure}

\begin{figure}[htbp]
    \centering
    \includegraphics[width=.48\textwidth, trim={0 50 0 15mm}, clip]{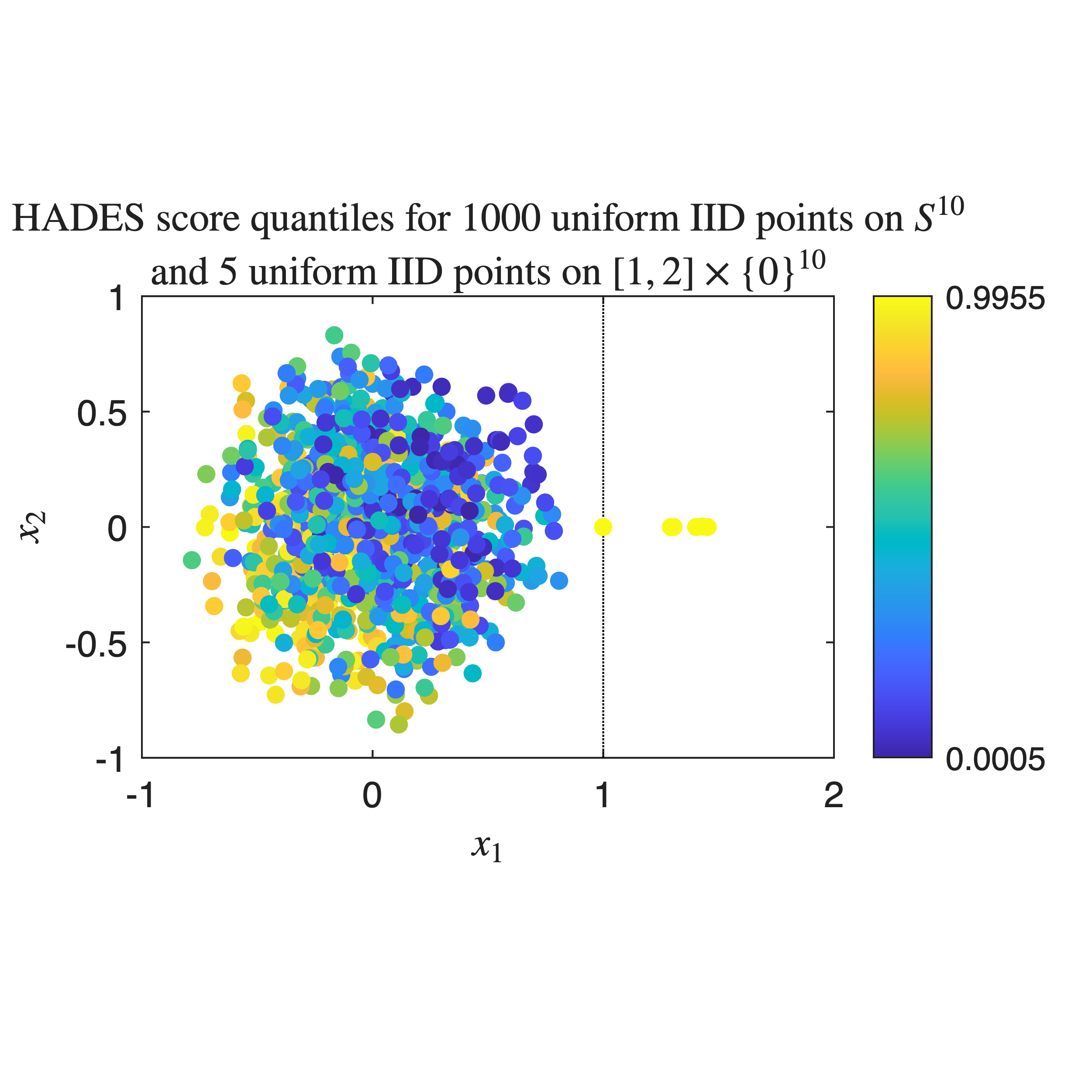}
    \includegraphics[width=.48\textwidth, trim={0 50 0 15mm}, clip]{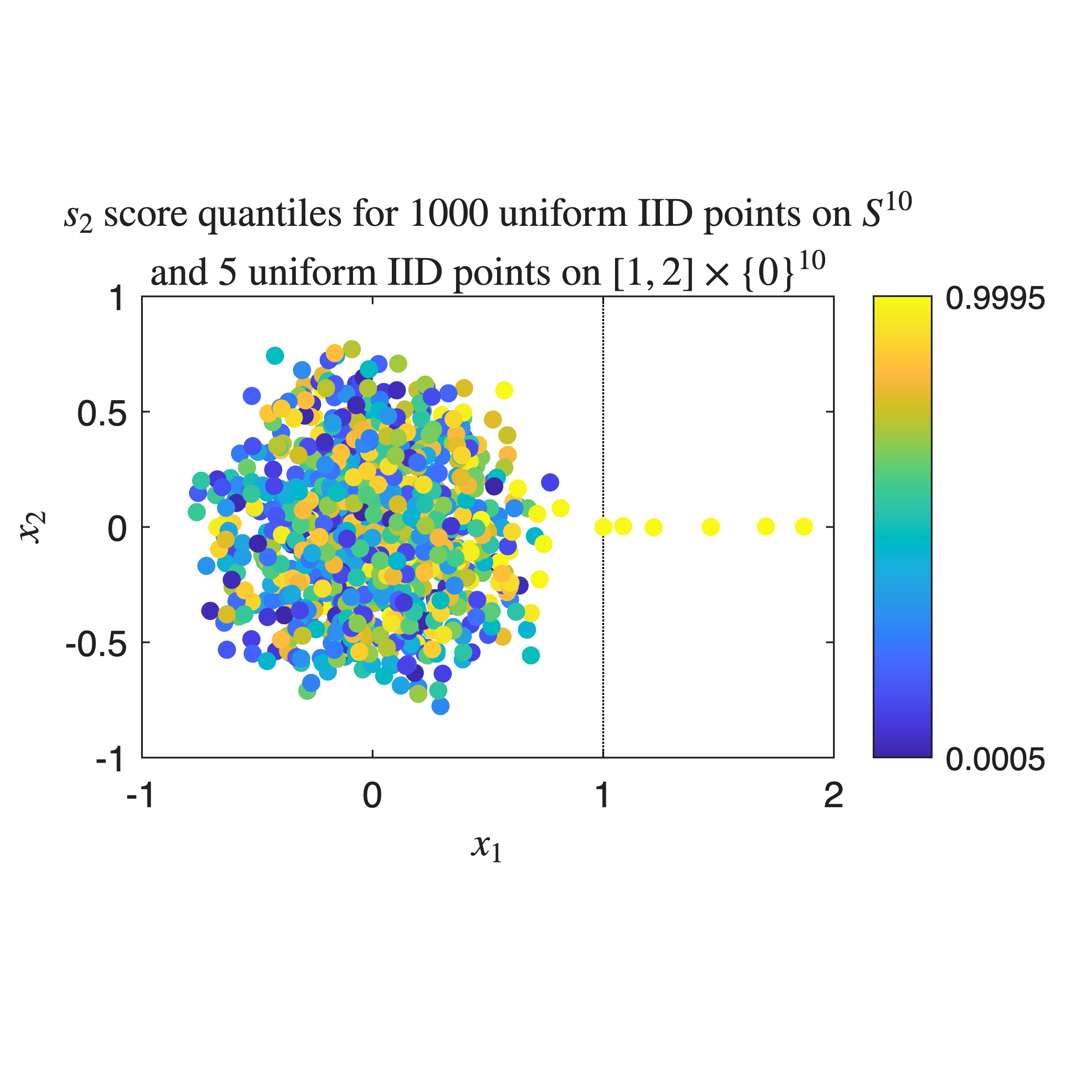}
    \includegraphics[width=.48\textwidth, trim={0 0 0 0mm}, clip]{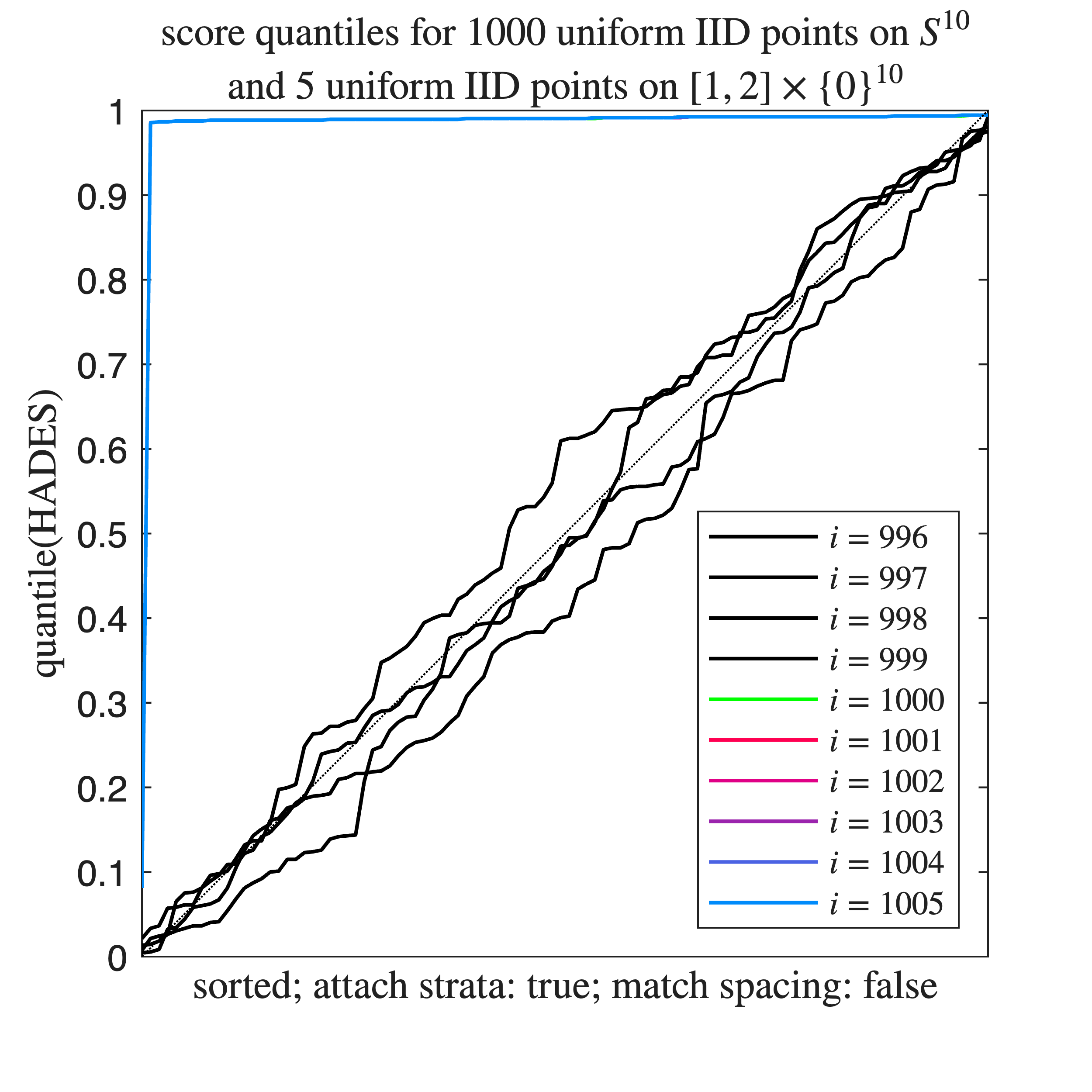}
    \includegraphics[width=.48\textwidth, trim={0 0 0 0mm}, clip]{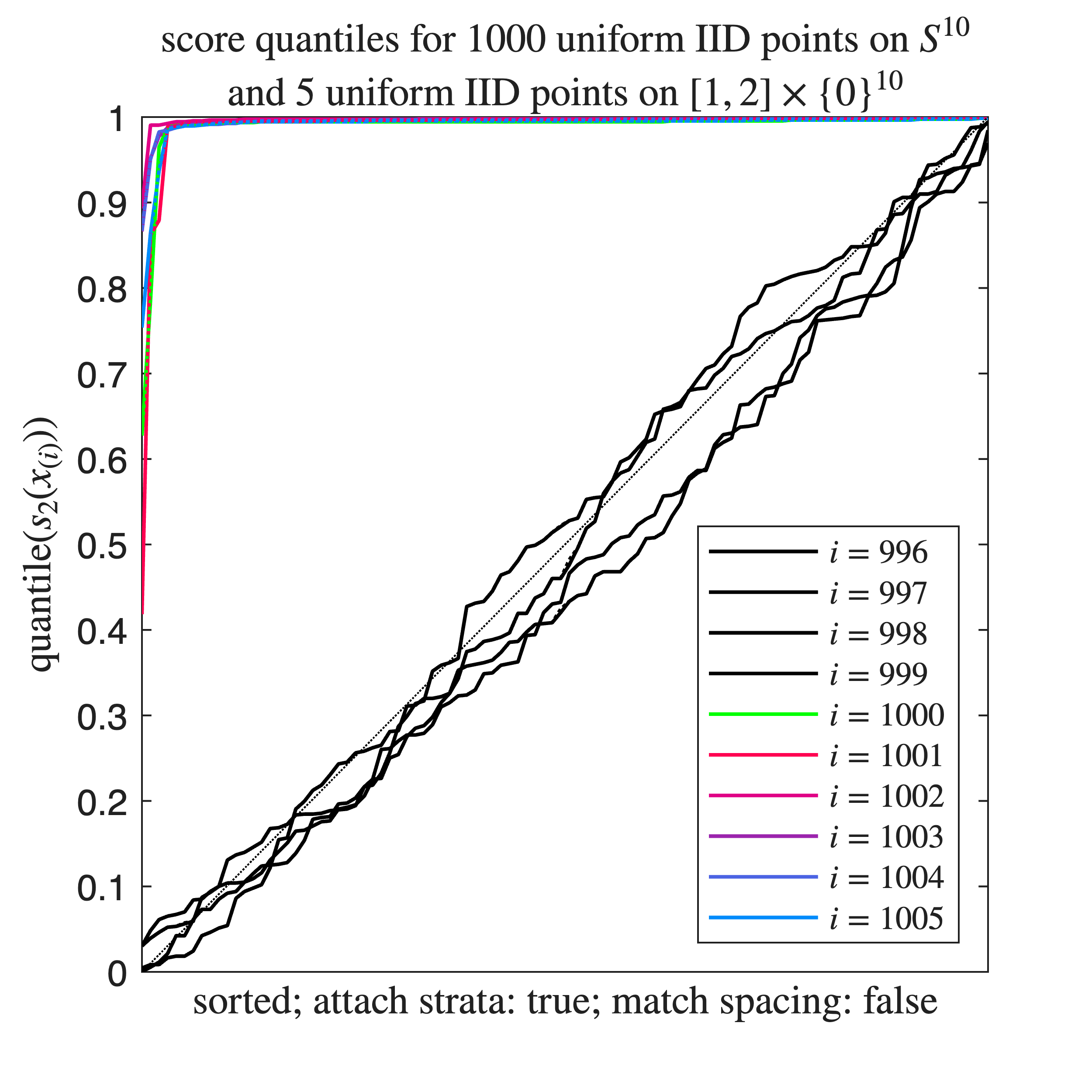}
        \caption{As in Figure \ref{fig:hair_ball_dim_11_attach_false_match_false}, but for an ``attached'' sample (with the attachment point indicated by a dashed line in the upper panels and {\color{green}in green in the lower panels}). 
        The $s_2$ plots actually show both exact and heuristic peel results, which differ slightly in a few places but are visually indistinguishable.
        }
    \label{fig:hair_ball_dim_11_attach_true_match_false}
\end{figure}

\begin{figure}[htbp]
    \centering
    \includegraphics[width=.48\textwidth, trim={0 80 0 15mm}, clip]{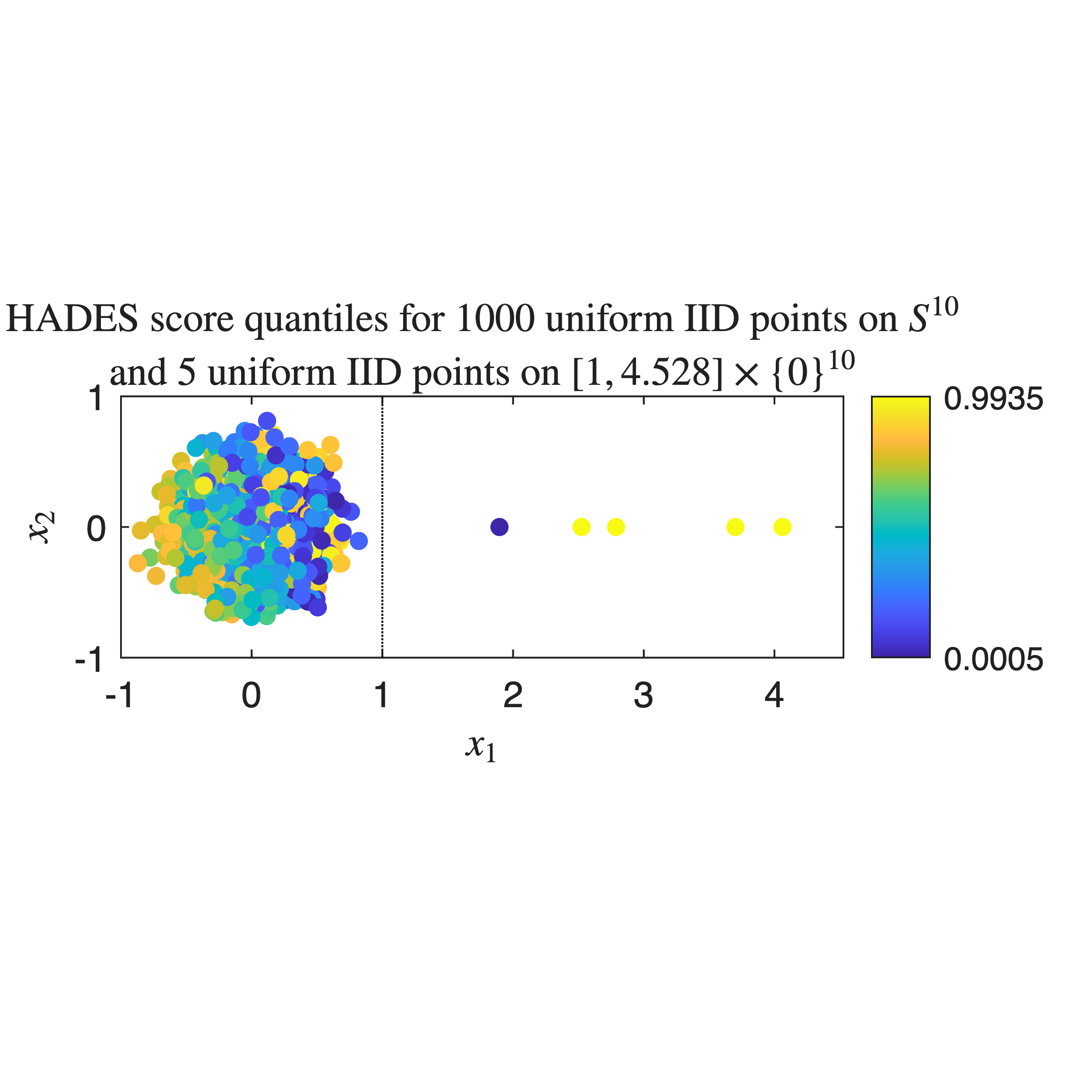}
    \includegraphics[width=.48\textwidth, trim={0 80 0 15mm}, clip]{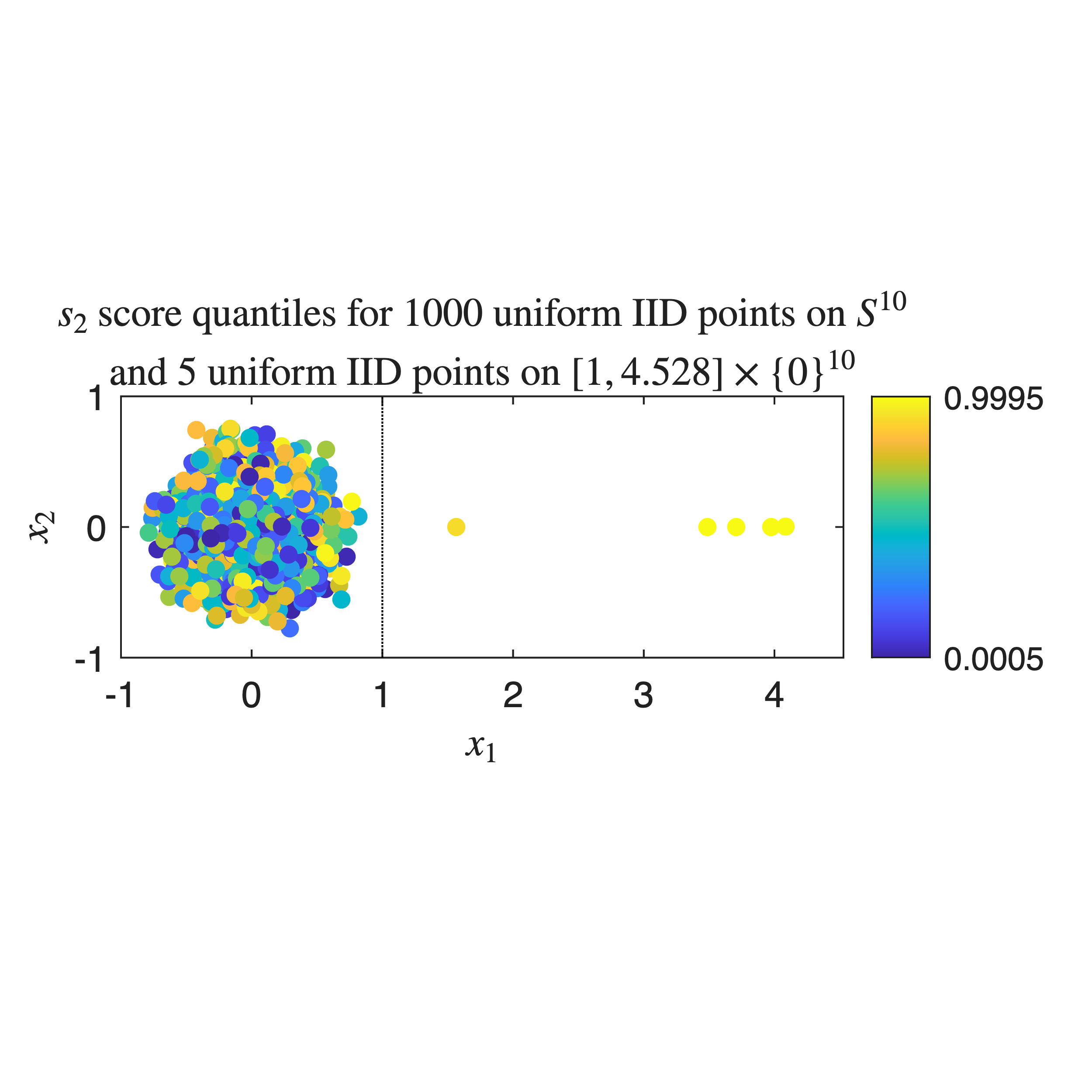}
    \includegraphics[width=.48\textwidth, trim={0 0 0 0mm}, clip]{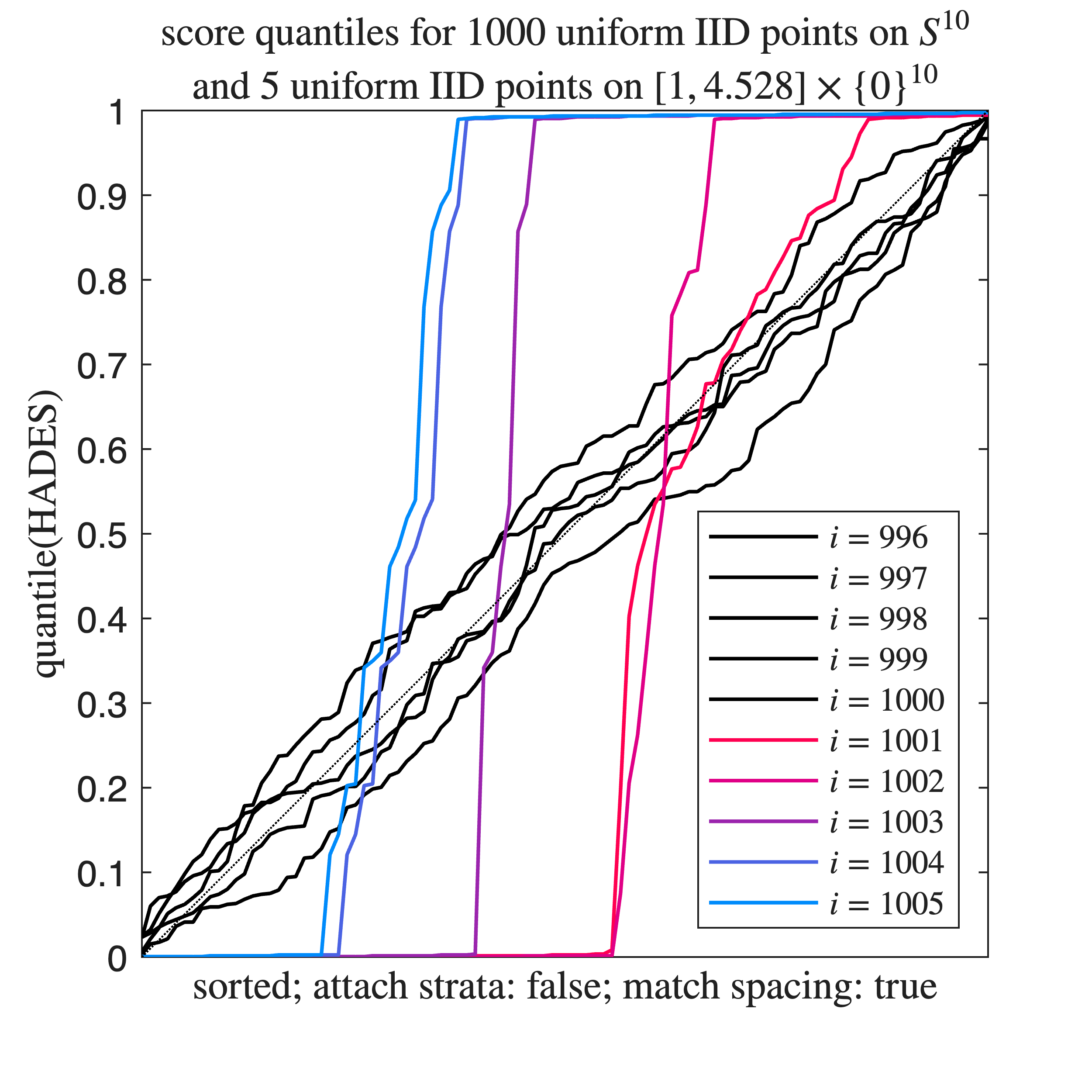}
    \includegraphics[width=.48\textwidth, trim={0 0 0 0mm}, clip]{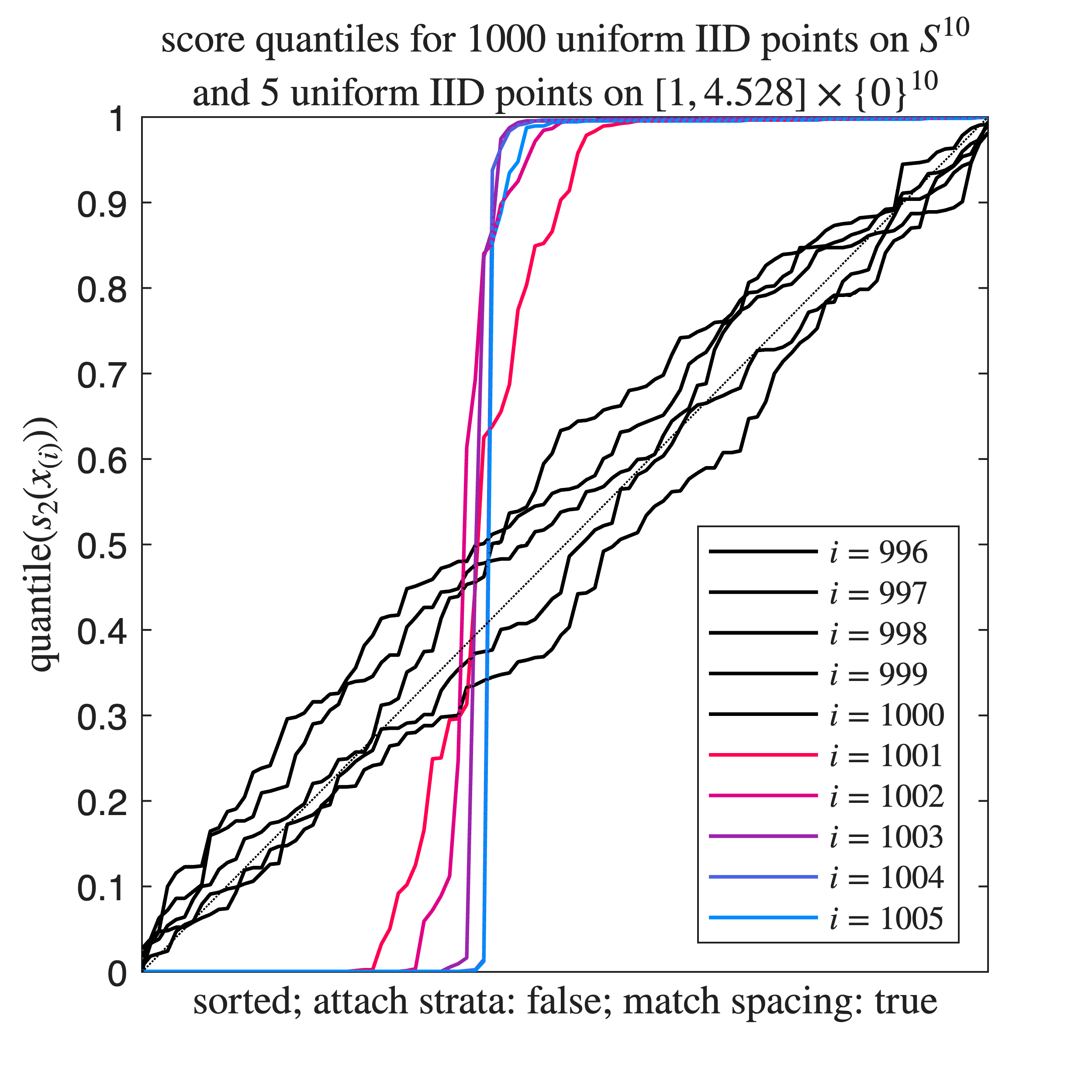}
        \caption{As in Figure \ref{fig:hair_ball_dim_11_attach_false_match_false}, but for a spacing-matched interval. Note that HADES scores consistently identifies interval stratum points closer to the sphere as much less singular than outliers, while $s_2$ identify interval stratum points closer to the sphere as slightly more singular than outliers. 
        The $s_2$ plots actually show both exact and heuristic peel results, which differ slightly in a few places but are visually indistinguishable.
        }
    \label{fig:hair_ball_dim_11_attach_false_match_true}
\end{figure}

\begin{figure}[htbp]
    \centering
    \includegraphics[width=.48\textwidth, trim={0 80 0 15mm}, clip]{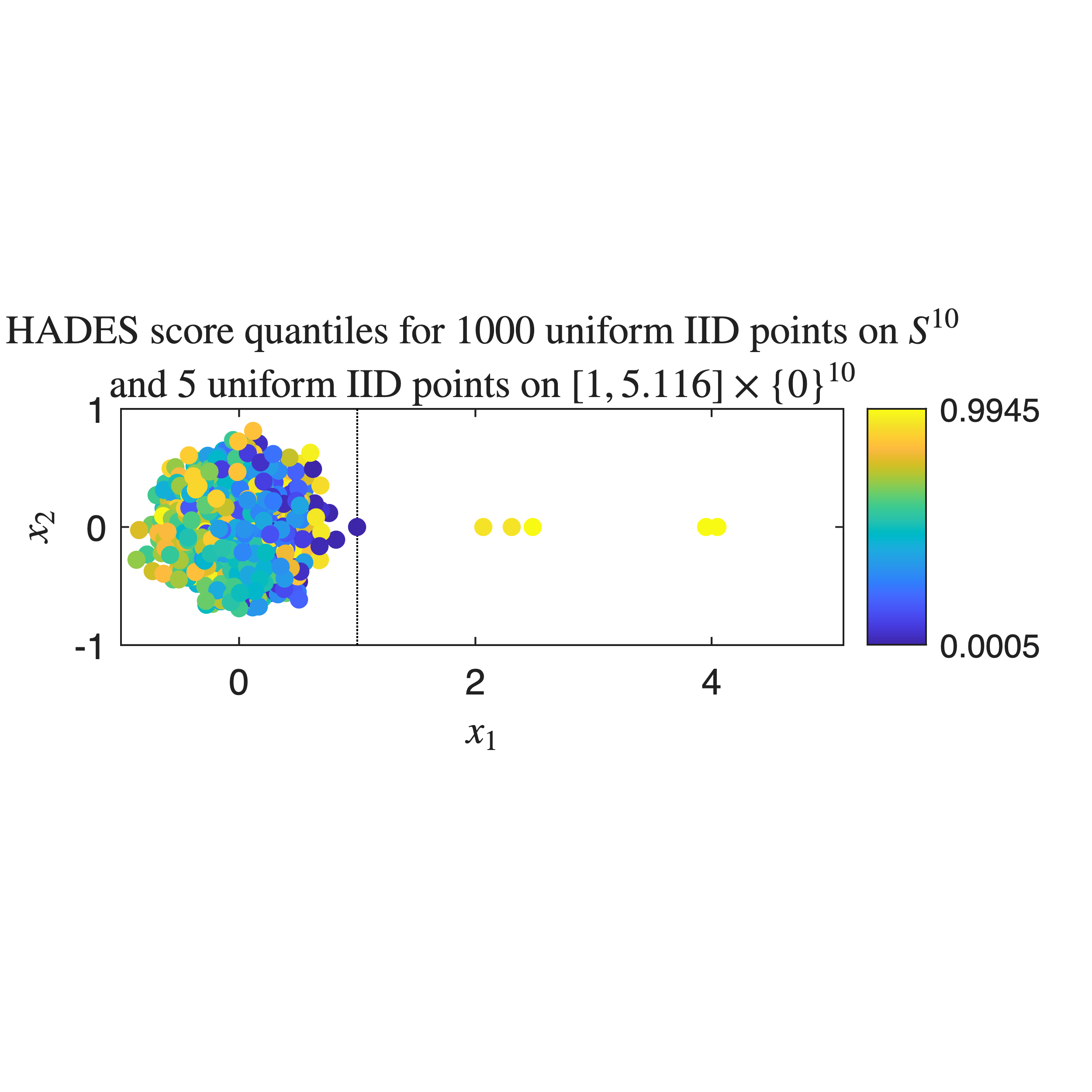}
    \includegraphics[width=.48\textwidth, trim={0 80 0 15mm}, clip]{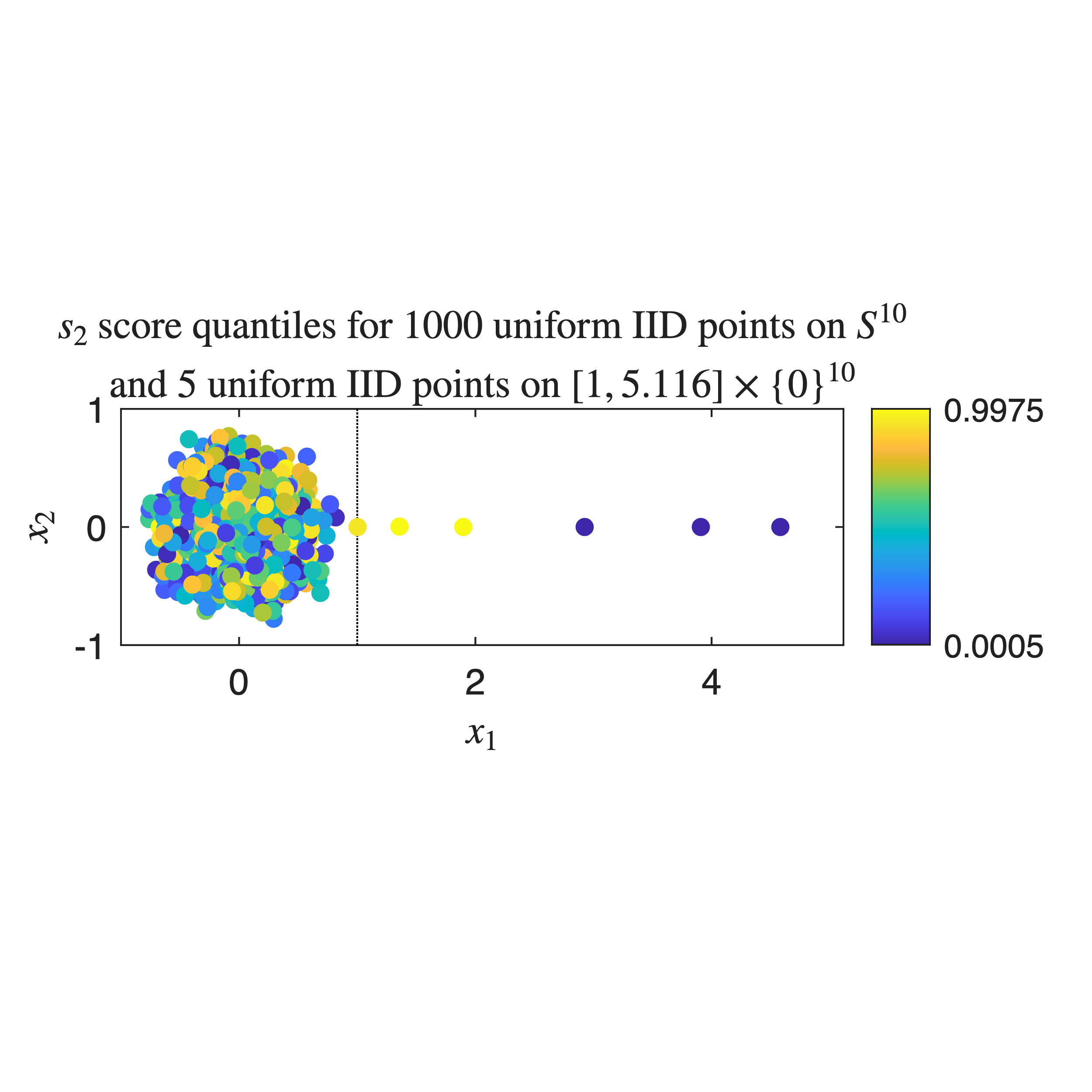}
    \includegraphics[width=.48\textwidth, trim={0 0 0 0mm}, clip]{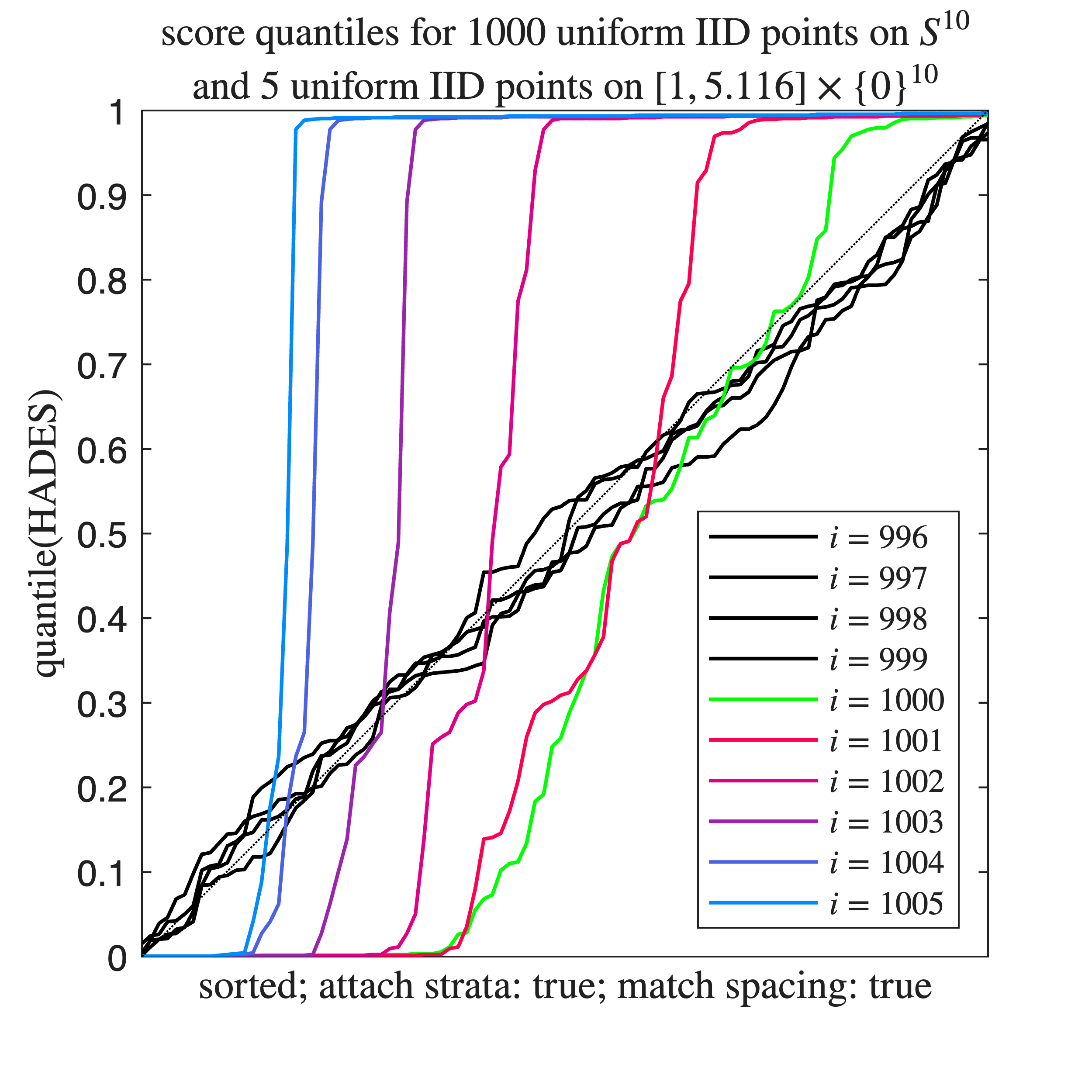}
    \includegraphics[width=.48\textwidth, trim={0 0 0 0mm}, clip]{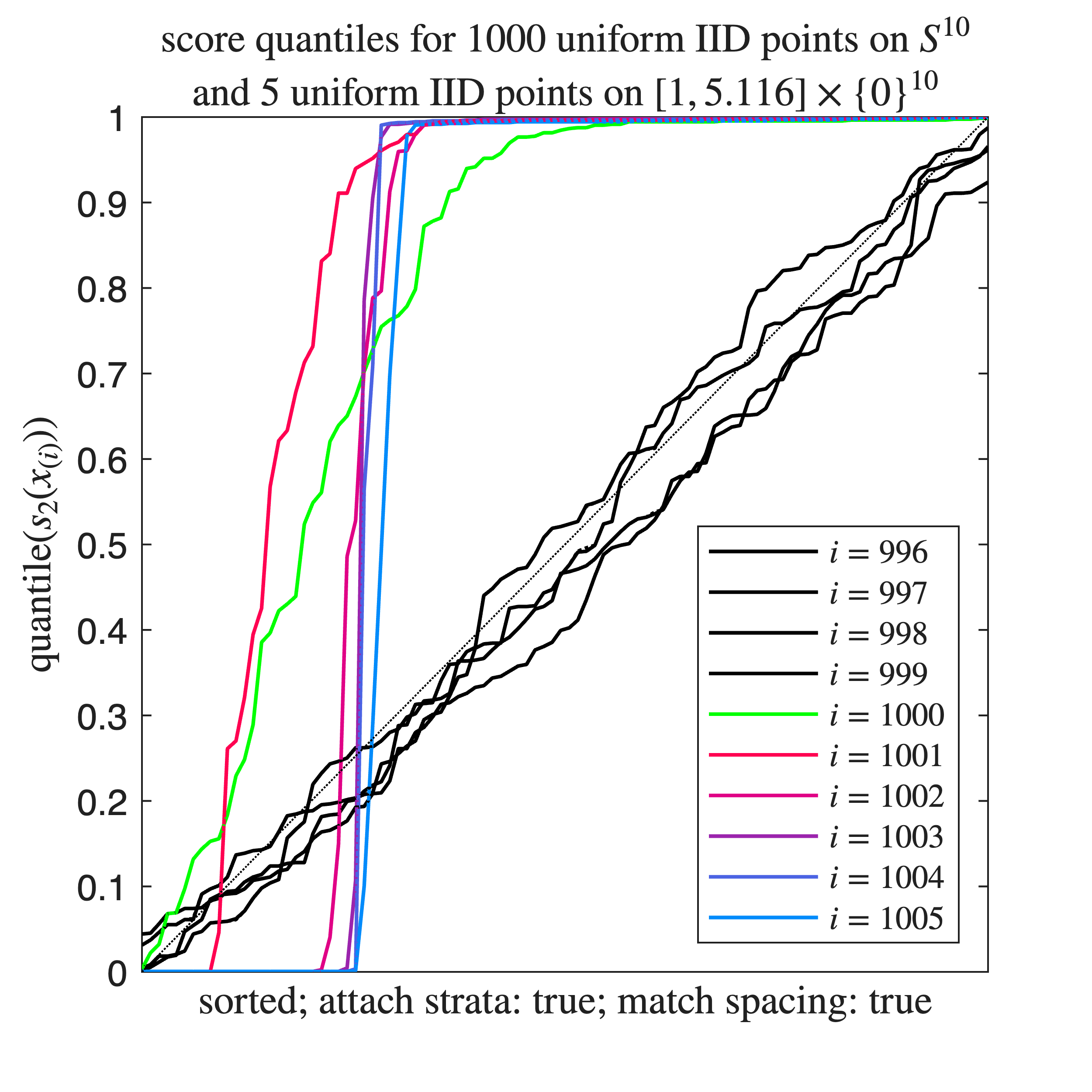}
        \caption{As in Figure \ref{fig:hair_ball_dim_11_attach_false_match_false}, but for an ``attached'' sample (with the attachment point indicated by a dashed line in the upper panels and {\color{green}in green in the lower panels}) and a spacing-matched interval. 
        The $s_2$ plots actually show both exact and heuristic peel results, which differ slightly in a few places but are visually indistinguishable.
        }
    \label{fig:hair_ball_dim_11_attach_true_match_true}
\end{figure}

\subsection{Comparison of dimension estimates on peel neighborhoods with the volume growth transform}\label{sec:VGT}

The \emph{volume growth transform} (VGT) of \cite{robinson2024structure,curry2025exploring} produces local dimension and scalar curvature estimates from a point cloud. The VGT is elegant in theory, but it is computationally expensive and it requires significant (typically manual) finesse in practice. The VGT practically requires a full distance matrix, and a good result depends on selecting a good window of points for regression, even when using a robust variant with iteratively reweighted least squares as we do here \cite{street1988note}. By comparison, we have seen that peel neighborhoods offer a way to efficiently compute local dimension estimates that are reliable in relative terms even if noisy in absolute terms. However, the dimension and curvature estimates of the VGT have only been subject to preliminary evaluations, and corroboration is useful. We proceed here with uniform sampling from compact manifolds of scalar curvature $-2$ (the Bolza surface), $0$ (the flat torus $(\mathbb{R}/\mathbb{Z})^2$), and $+2$ (the sphere $S^2$), respectively.

The first case is hardest to set up, and we relegate the details to Appendix \S \ref{sec:bolza}.
The second case works with the distance $$d_{(\mathbb{R}/\mathbb{Z})^m}(x,x') = \left ( \sum_{j=1}^m \left [ \frac{1}{2} - \left | \frac{1}{2} - |x_j - x'_j| \right | \right ]^2 \right )^{1/2}$$ on samples from $[0,1]^m$, and the third case works by normalizing samples from a standard Gaussian, computing cosine distances, and subsequently transforming these to angular distances.

Figure \ref{fig:EssVsVGT2} shows that multidimensional scaling (MDS) followed by ESSa consistently outperforms the robust VGT over comparable regions, without substantial dependence on curvature. However, without manually tuned bands, the VGT curvature estimates themselves are noisy and unreliable (not shown). These observations hold for IID uniform samples over surfaces of negative, zero, and positive curvature.

\begin{figure}[htbp]
    \centering
    \includegraphics[width=.32\textwidth, trim={0 0 0 0mm}, clip]{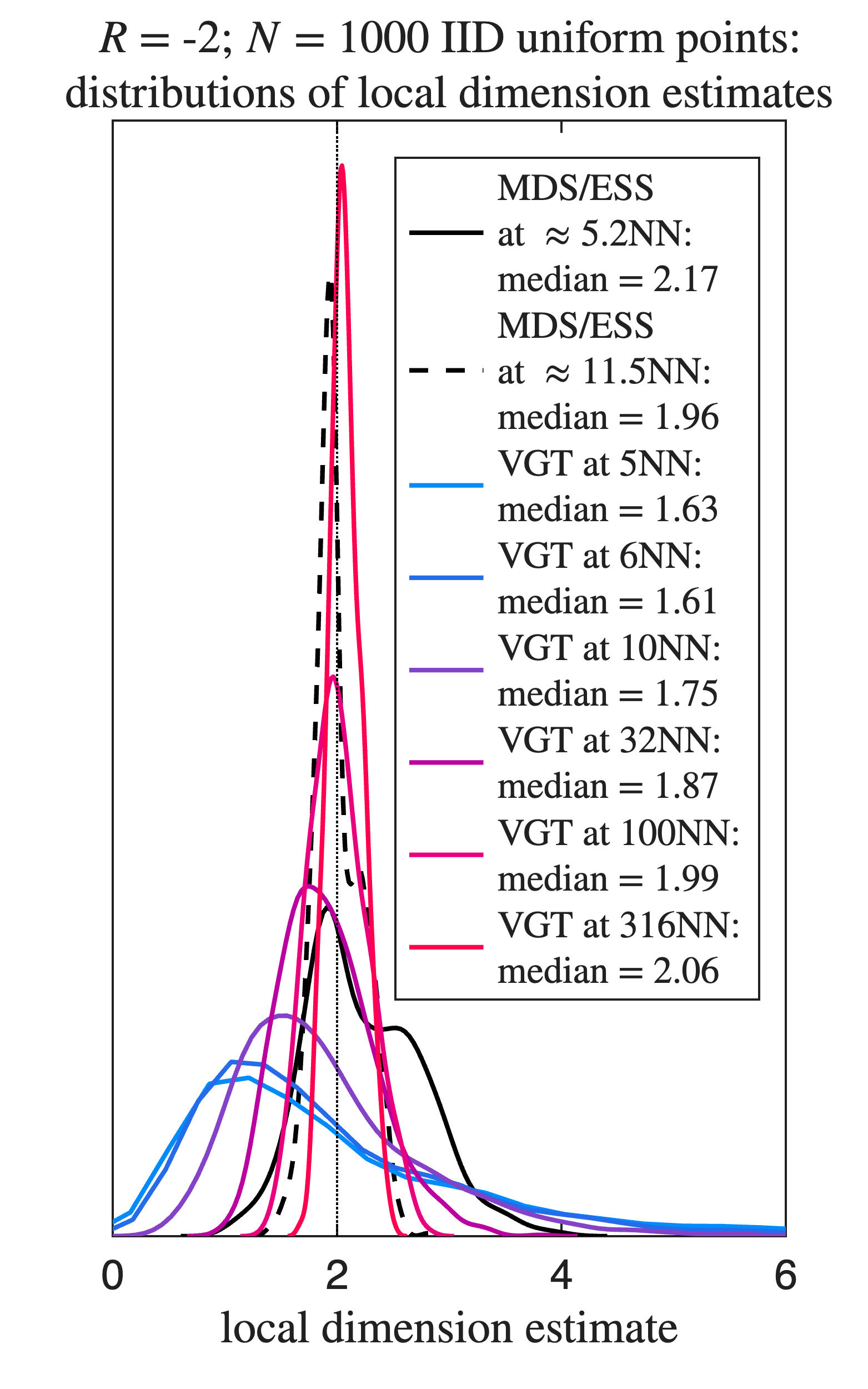}
    \includegraphics[width=.32\textwidth, trim={0 0 0 0mm}, clip]{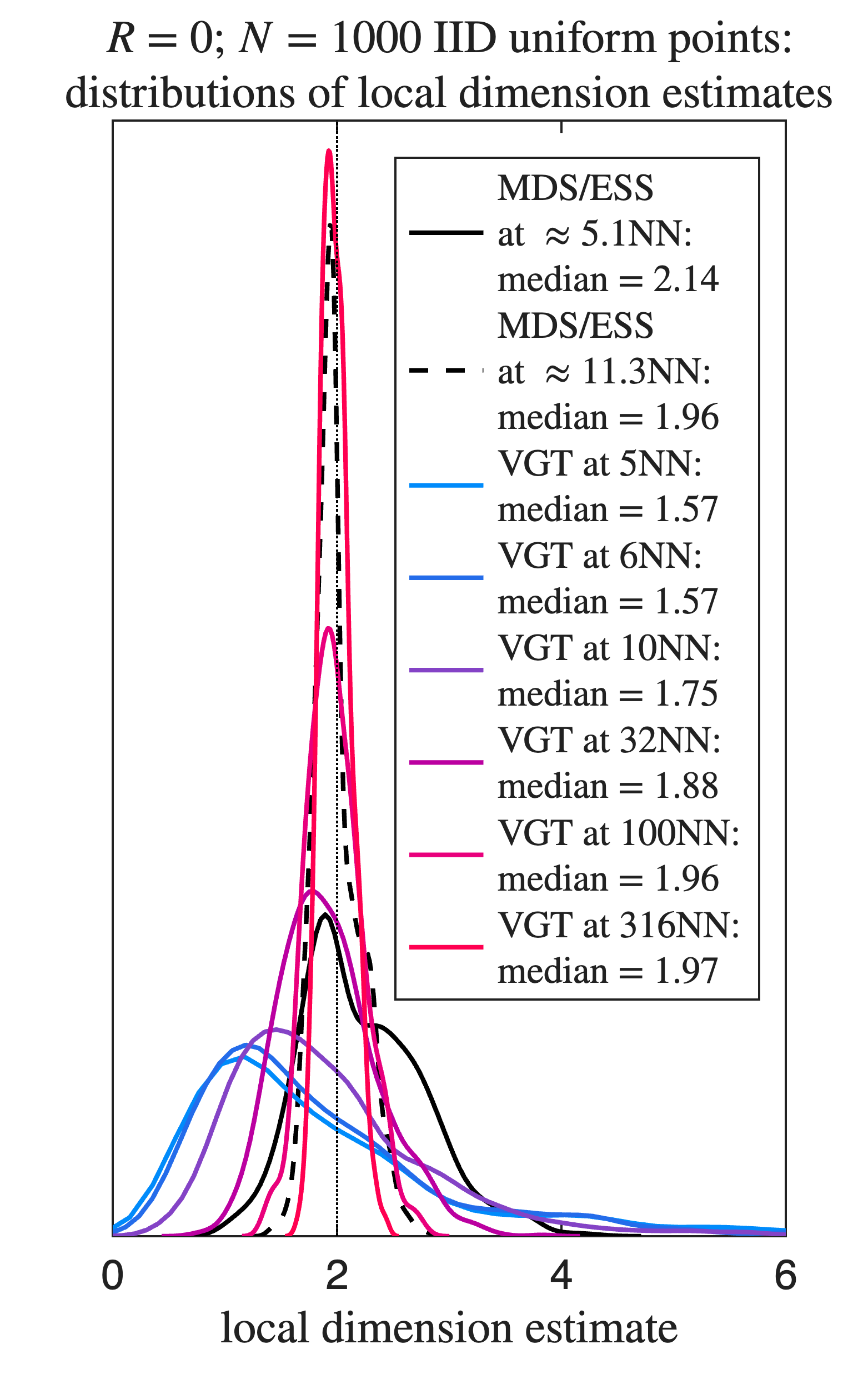}
    \includegraphics[width=.32\textwidth, trim={0 0 0 0mm}, clip]{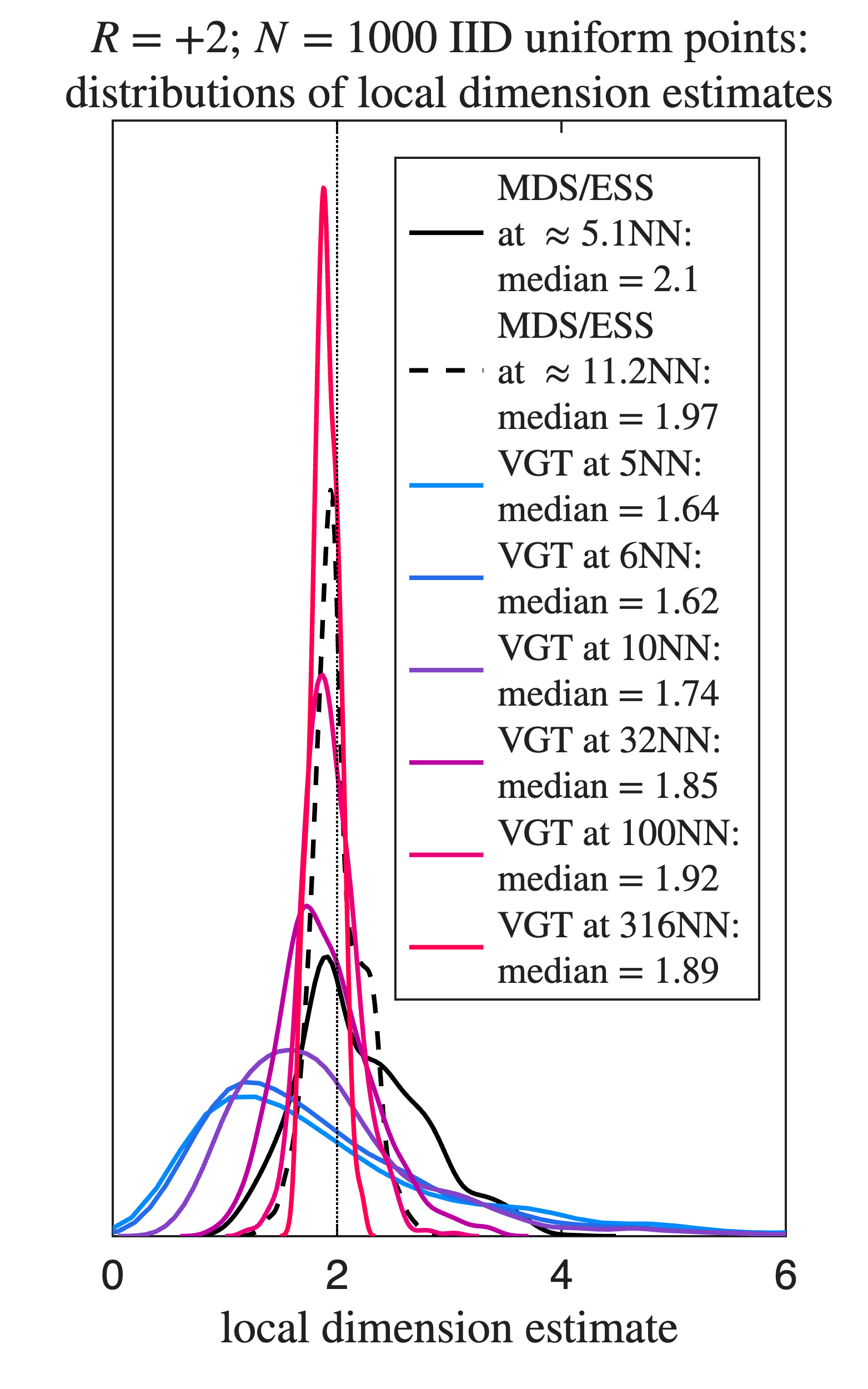}
        \caption{Left: local dimension estimates for a IID uniform sample from the Bolza surface. The panel shows the results of MDS followed by ESSa on the peel neighborhood $\nu(x) = B_{\rho(x)}(x)$ as a black solid line as well as on the iterated peel neighborhood $\nu_2(x) = \cup_{y \in \nu(x)} \nu(y)$ as a black dashed line (both with the median number of elements indicated), along with a robust VGT on the first $k$ nearest neighbors for indicated values of $k$. (We use a soft threshold of twice the median distance instead of $(1+1/m)$ times here, because our code consumes distance matrices and not point data, and in general only the former might be available without resorting to a full-dimensional Schoenberg embedding using negative type.) Center and right: as in the left panel, but for the flat 2-torus $(\mathbb{R}/\mathbb{Z})^2$ and $S^2$, respectively. 
        In these cases, data turn out to be identical for exact peels and heuristic approximations.
        }
    \label{fig:EssVsVGT2}
\end{figure}

With some evidence in hand that dimension estimates are insensitive to the sign of curvature, for convenience we restrict attention to the flat torus and sphere in higher dimensions. Figures \ref{fig:EssVsVGT3}-\ref{fig:EssVsVGT50} show that while the results for MDS followed by ESSa on (iterated) peel neighborhoods may not be good in an absolute sense for sparse, high-dimensional data, probably no other technique will deliver results that are good in an absolute sense in such regimes. This approach still delivers results that are relatively good, i.e., at least as good if not better than a robust VGT absent manual tuning (which is not guaranteed to improve performance). Moreover, our peel-oriented technique is computationally efficient at scale.

\begin{figure}[htbp]
    \centering
    \includegraphics[width=.48\textwidth, trim={0 0 0 0mm}, clip]{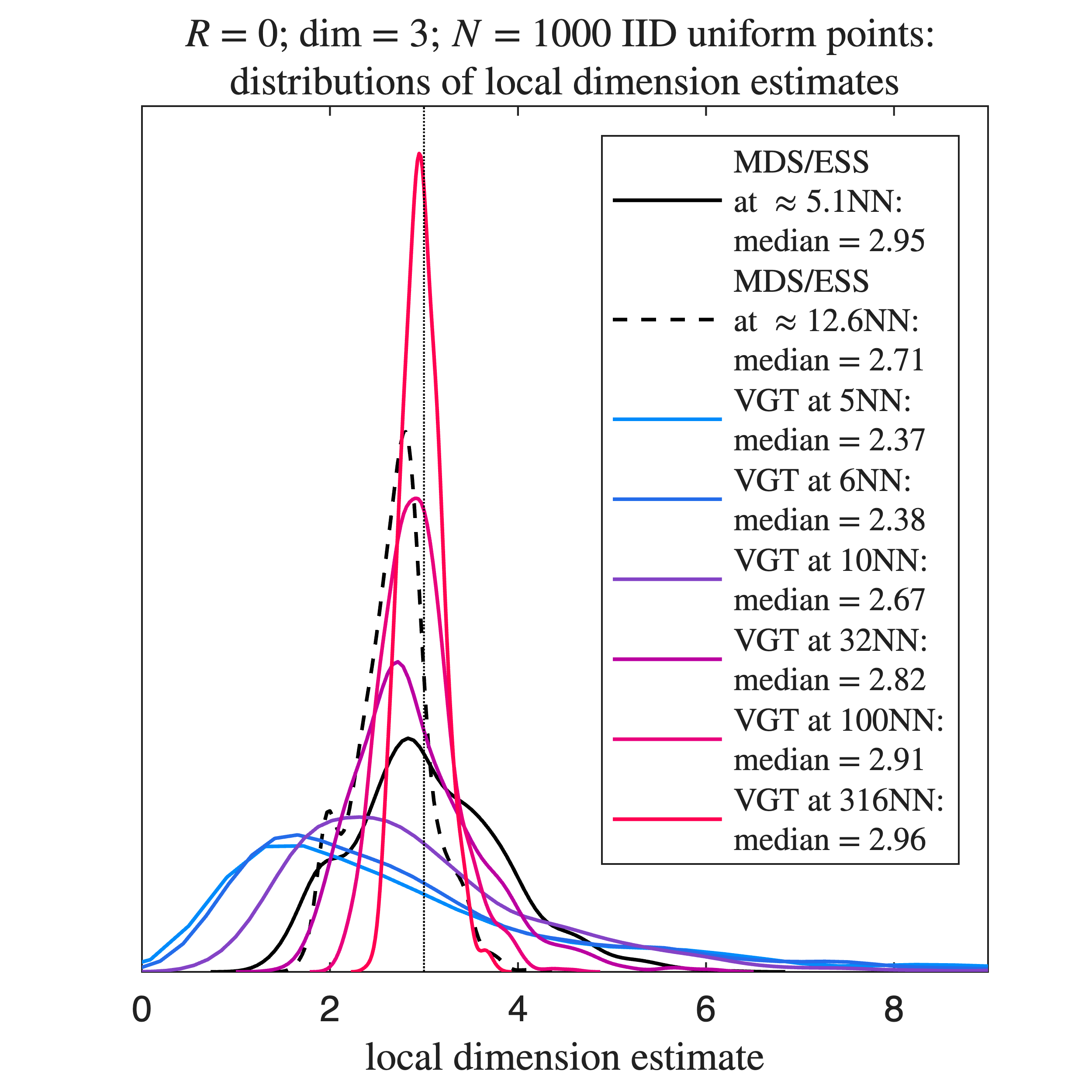}
    \includegraphics[width=.48\textwidth, trim={0 0 0 0mm}, clip]{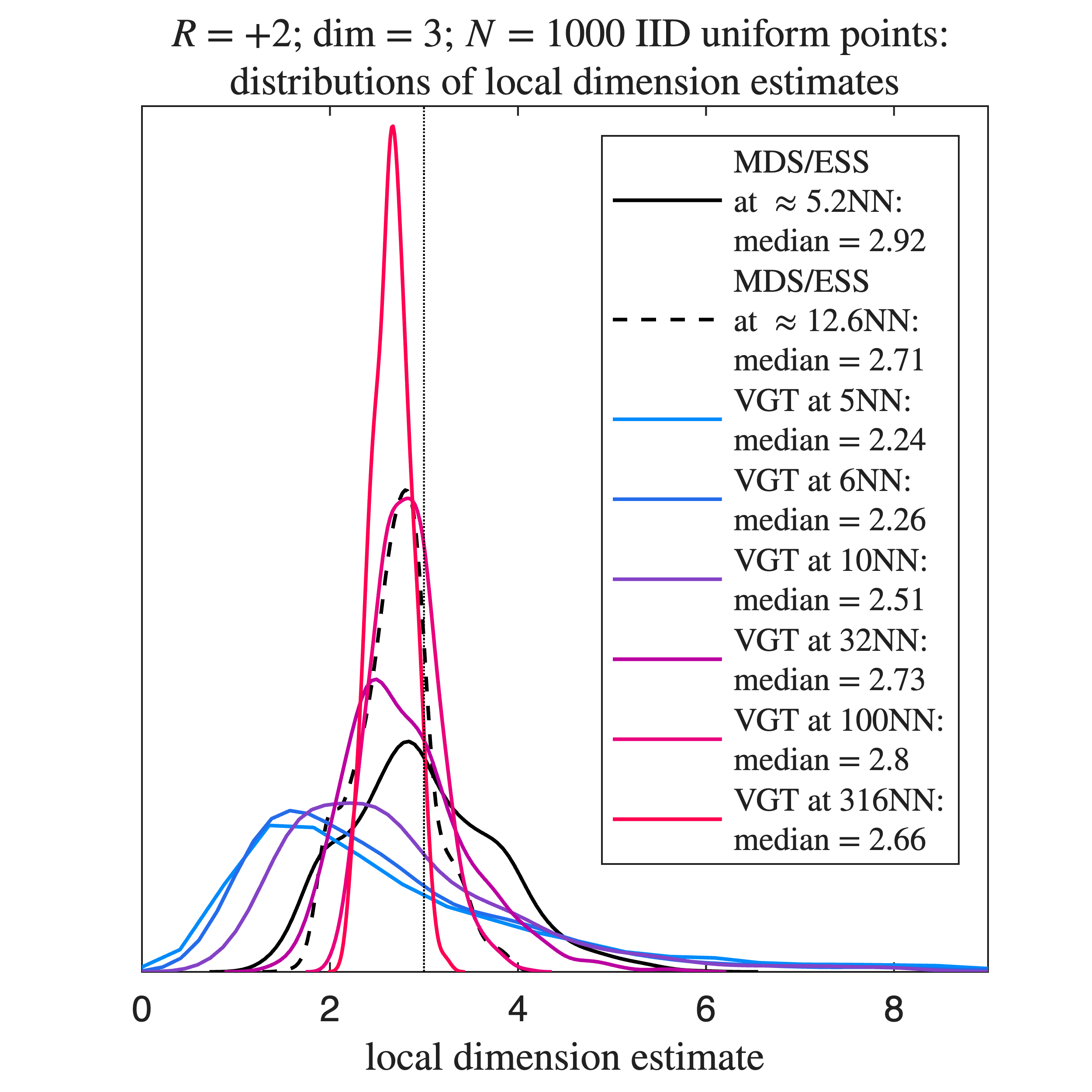}
        \caption{As in Figure \ref{fig:EssVsVGT2} but for $(\mathbb{R}/\mathbb{Z})^3$ (left panel) and $S^3$ (right panel). 
        In these cases, data turn out to be identical for exact peels and heuristic approximations.
        }
    \label{fig:EssVsVGT3}
\end{figure}

\begin{figure}[htbp]
    \centering
    \includegraphics[width=.48\textwidth, trim={0 0 0 0mm}, clip]{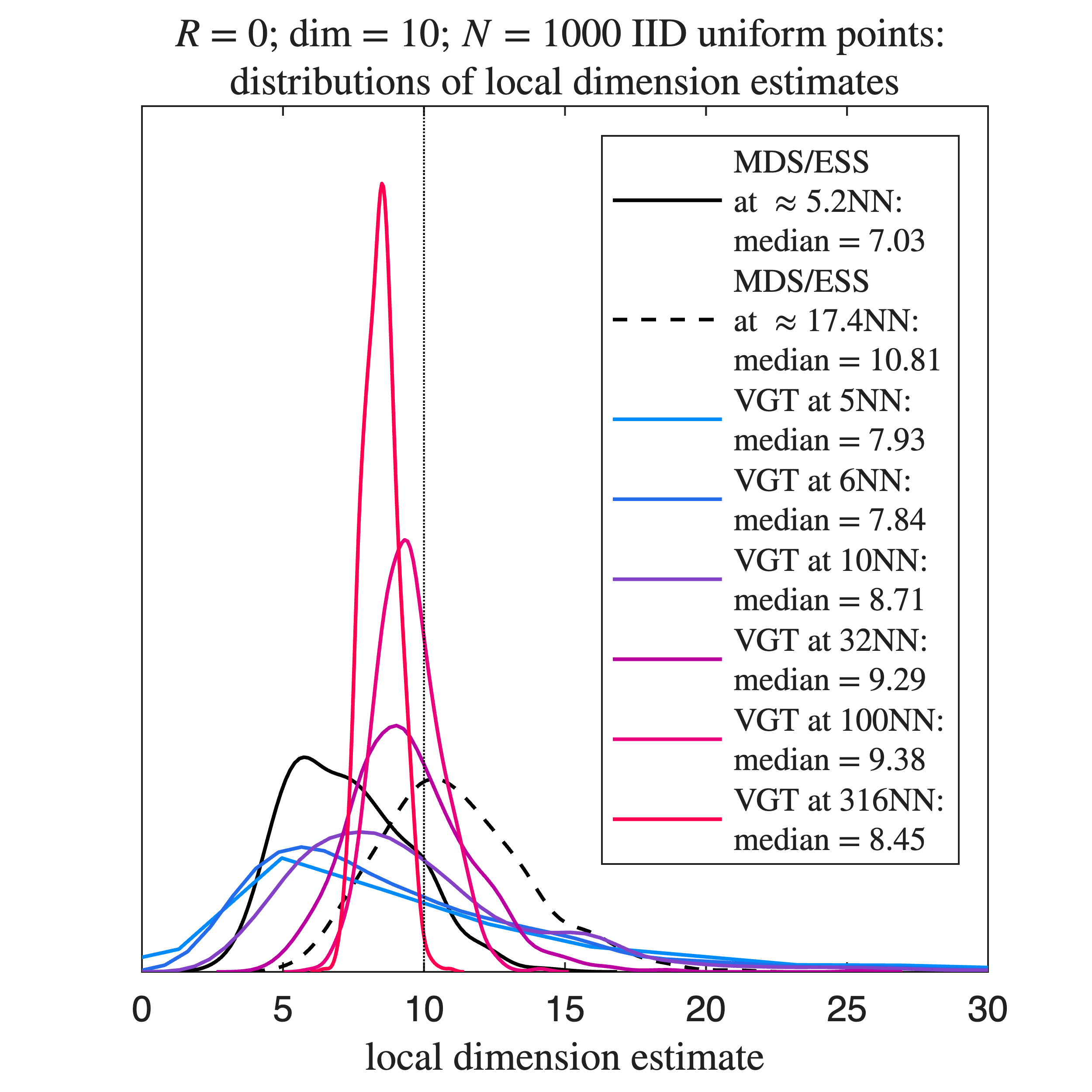}
    \includegraphics[width=.48\textwidth, trim={0 0 0 0mm}, clip]{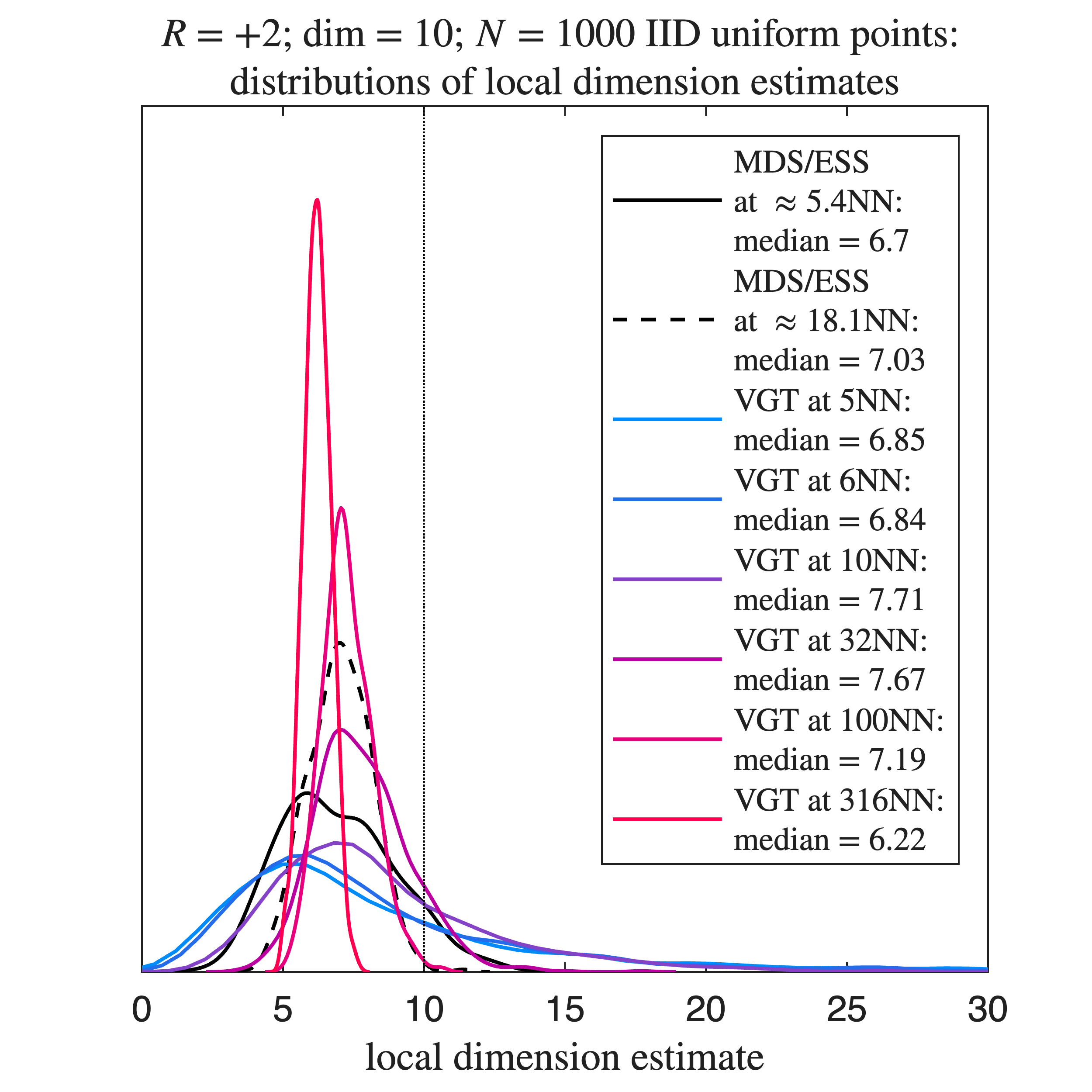}
        \caption{As in Figure \ref{fig:EssVsVGT2} but for $(\mathbb{R}/\mathbb{Z})^{10}$ (left panel) and $S^{10}$ (right panel). 
        In these cases, data turn out to be identical for exact peels and heuristic approximations.
        }
    \label{fig:EssVsVGT10}
\end{figure}

\begin{figure}[htbp]
    \centering
    \includegraphics[width=.48\textwidth, trim={0 0 0 0mm}, clip]{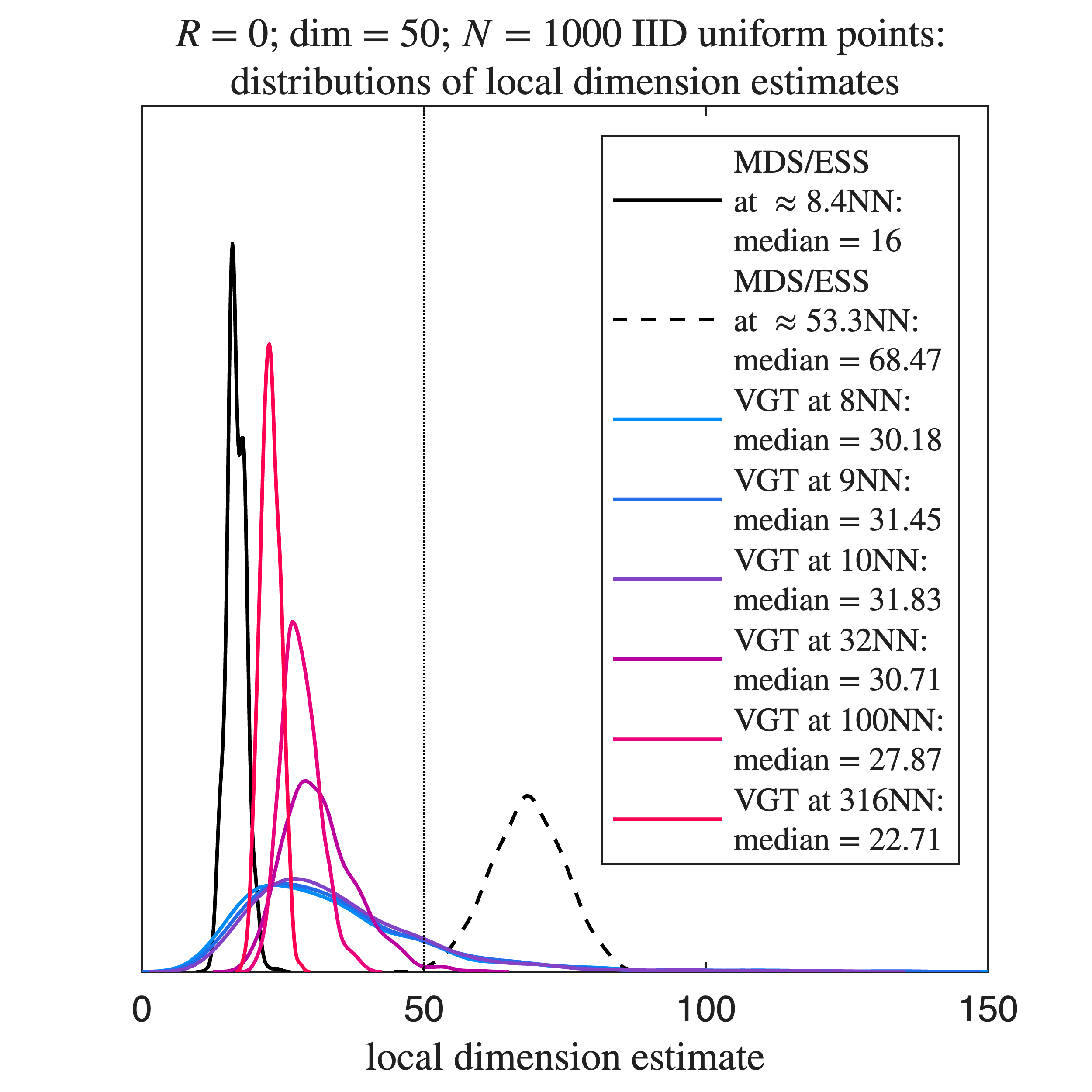}
    \includegraphics[width=.48\textwidth, trim={0 0 0 0mm}, clip]{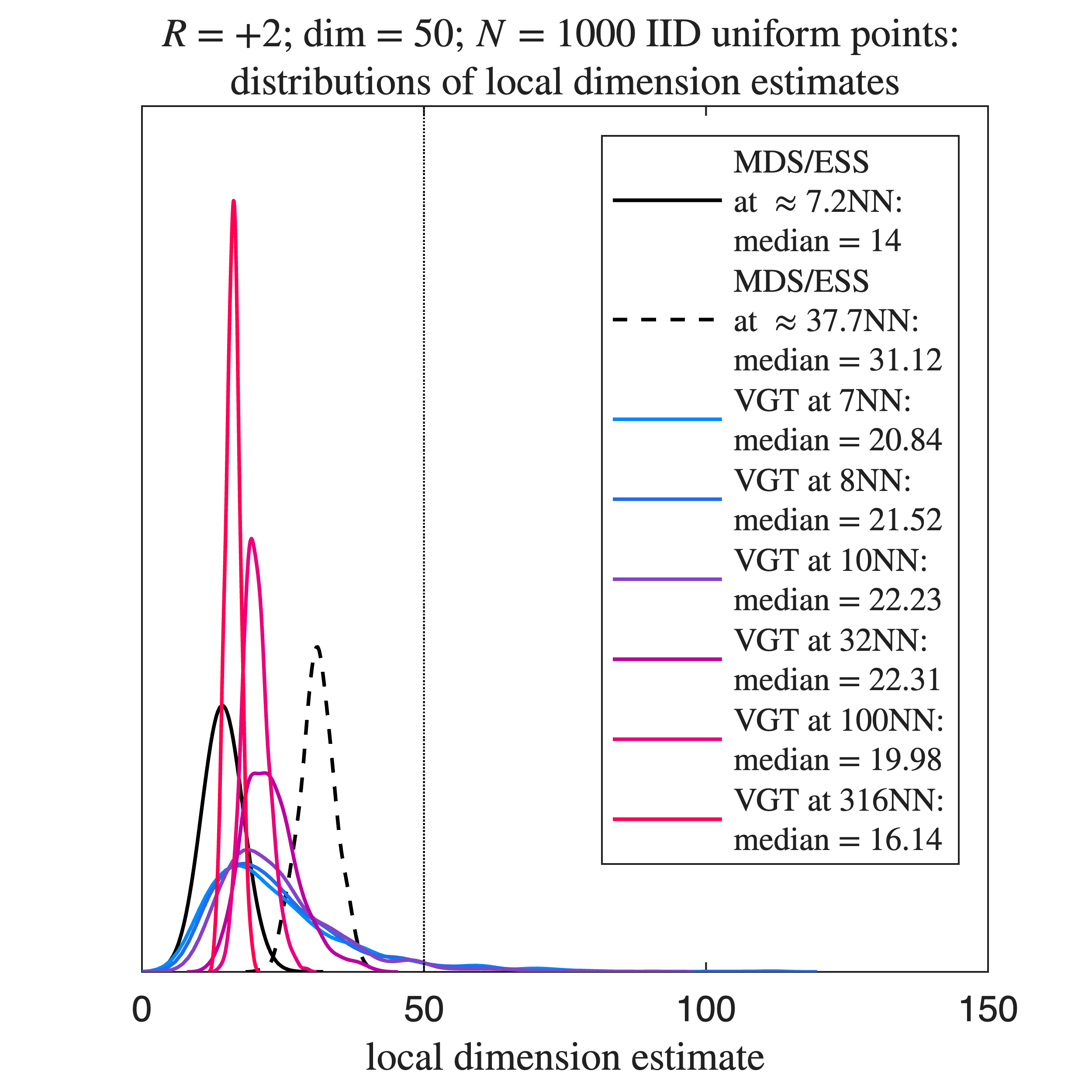}
        \caption{As in Figure \ref{fig:EssVsVGT2} but for $(\mathbb{R}/\mathbb{Z})^{50}$ (left panel) and $S^{50}$ (right panel). In these cases, data turn out to be identical for exact peels and heuristic approximations.
        }
    \label{fig:EssVsVGT50}
\end{figure}

\section{Conclusions}
\label{sec:conclusions}

Peel neighborhoods provide a canonical, parameter-free notion of locality via convexity in finite metric spaces of strict negative type. Their definition in terms of a peel (i.e., the support of a diversity-maximizing distribution at scale zero) inherits the simplicity of Algorithm \ref{alg:ApproximatePeel} and the still fairly simple exact variants described in \S G of \cite{huntsman2025peeling}. With a soft or adaptive threshold, peel neighborhoods are also very efficient to compute at scale (both of cardinality and dimension), and they reveal boundary-like points.

The experiments of \S \ref{sec:experiments} reveal that peel neighborhoods have utility along multiple axes. \S \ref{sec:comparison} shows they encode locality more flexibly and more efficiently than $k$-nearest neighbor or fixed-radius neighborhoods. \S \ref{sec:annuli} demonstrates that they provide a practical alternative to topological persistence. \S \ref{sec:scaling} illustrates how peel neighborhoods can approximate peels that are much more expensive to compute exactly. Finally, \S \ref{sec:localDimension} and \ref{sec:VGT} show that peel neighborhoods can help efficiently estimate local dimension and thereby identify singularities in stratified manifolds using the gradient norm score $s_2$ of \S \ref{sec:gradients}. This technique is robust, efficient, and performant. 

It is natural to ask if there is a useful nontrivial result involving peels of subsets of peels (or of peel neighborhoods of elements of peels). In the framing of Lemma \ref{lem:inclusionBound}, we want to be able to understand when $\Delta = d|_\mathcal{I}^{-1} \delta > 0$, i.e., when the peel of a subset is a subset of a peel. Numerical experiments with Gaussian data for $N \gg 5$ in various dimensions suggested 
\begin{conjecture}[disproved]\label{conj:gorman}
If $d$ is strict negative type on a finite space and $\mathcal{I} \subseteq \textnormal{peel}(d)$, then $\textnormal{peel}(d|_\mathcal{I}) = \mathcal{I}$.
\end{conjecture}

Evan Gorman produced the counterexample 
$$d = \begin{pmatrix}
0 & 1 & 1 & 2 \\ 
1 & 0 & 2 & 1 \\ 
1 & 2 & 0 & 2 \\ 
2 & 1 & 2 & 0
\end{pmatrix}.$$ 
This counterexample is readily verified to be a metric of strict negative type by a test from \cite{hjorth1998finite}. Calculations yield that $p_*(d) = (1,1,2,2)^T/6$, so $\textnormal{peel}(d) = \{1,2,3,4\}$, while $p_*(d|_{\{1,2,3\}}) = (0,1,1)^T/2$, so $\textnormal{peel}(d|_{\{1,2,3\}}) = \{2,3\} \ne \{1,2,3\}$. 
This counterexample is also exotic: it is not a Euclidean distance matrix by the Schoenberg-Young-Householder theorem (Theorem 3.1 of \cite{alfakih2018euclidean}), and the points are not in general position (e.g., the first point is between the second and third). In light of its non-genericity, it is not surprising that this counterexample is also brittle. If we take more generally
$$d_\varepsilon = \begin{pmatrix}0 & 1+\varepsilon & 1+\varepsilon & 2+\varepsilon \\ 1+\varepsilon & 0 & 2+\varepsilon & 1+\varepsilon \\ 1+\varepsilon & 2+\varepsilon & 0 & 2+\varepsilon \\ 2+\varepsilon & 1+\varepsilon & 2+\varepsilon & 0\end{pmatrix}$$
for $\varepsilon \ge 0$, numerical computations indicate that $d_\varepsilon$ is always a metric of strict negative type: it is Euclidean for $\varepsilon > 1/\phi \approx 0.618$, where $\phi = (1+\sqrt{5})/2$ is the golden ratio. Numerics indicate that $\textnormal{peel}(d_\varepsilon) = \{1,2,3,4\}$ and $\textnormal{peel}(d_\varepsilon|_{\{1,2,3\}}) = \{1,2,3\}$ for $\varepsilon > 0$.

The exoticness and brittleness of the counterexample above suggested that a slightly weakened variant of the conjecture might hold, e.g., for Euclidean and/or general position strict negative type metrics. However, a brute-force search over subsets of $\mathbb{Z}^2$ endowed with Euclidean distance produced the following pedestrian (i.e., perturbatively stable) counterexample on five points. Let $d$ now be the Euclidean distance matrix for $\{(0,0),(0,1),(1,0),(1,3),(3,1)\} \subset \mathbb{R}^2$, i.e., 
$$d = \begin{pmatrix}
0 & 1 & 1 & \sqrt{10} & \sqrt{10} \\ 
1 & 0 & \sqrt{2} & \sqrt{5} & 3 \\ 
1 & \sqrt{2} & 0 & 3 & \sqrt{5} \\ 
\sqrt{10} & \sqrt{5} & 3 & 0 & 2\sqrt{2} \\ 
\sqrt{10} & 3 & \sqrt{5} & 2\sqrt{2} & 0 
\end{pmatrix}.$$
Symbolic calculations yield
$$p_*(d) \propto 
\frac{1}{176}\begin{pmatrix}
(-10\sqrt{2}+24)\sqrt{5}-62\sqrt{2}+96 \\
(-30\sqrt{2}+14)\sqrt{5}-10\sqrt{2}+78 \\
(-30\sqrt{2}+14)\sqrt{5}-10\sqrt{2}+78 \\
(\sqrt{2}+2)\sqrt{5}-7\sqrt{2}+30 \\
(\sqrt{2}+2)\sqrt{5}-7\sqrt{2}+30
\end{pmatrix}
\approx
\begin{pmatrix}
0.1725 \\ 
0.0017 \\ 
0.0017 \\ 
0.1576 \\ 
0.1576
\end{pmatrix},
$$
while 
$$p_*(d|_{\{1,2,3,4\}}) \propto 
\frac{1}{30}\begin{pmatrix}
-3\sqrt{10}+15 \\
0 \\
-5\sqrt{10}+20 \\
\sqrt{10}+5
\end{pmatrix}
\approx
\begin{pmatrix}
0.1838 \\ 
0 \\ 
0.1396 \\ 
0.2721
\end{pmatrix}.
$$

As a final remark, developing results that connect peel neighborhoods to constructions commonly associated with topological persistence would be interesting and probably useful. In a more computational direction, kernel solvers might inform the large-scale computation of peels and provide avenues for further scaling \cite{strumpack2020}. We are pursuing this thread in separate work.

\section*{Acknowledgments}
Thanks to Michael Robinson for helpful conversations, and to Jewell Thomas and Cynthia Ukawu for parallel developments (as it were). Thanks to Evan Gorman for helpful conversations and his disproof of Conjecture \ref{conj:gorman}. Thanks also to other interlocutors at DARPA for helpful conversations. 


Claude Opus 
helped with idea elaboration, manuscript review, finding references and some initial proofs.
I read Claude's references and verified and rewrote its proofs, and I am solely responsible for their correctness. I personally wrote everything in this paper except for some BibTeX entries, all of which I inspected, along with the actual references they point to. 

This research was developed with funding from the Defense Advanced Research Projects Agency (DARPA). The views, opinions and/or findings expressed are those of the author and should not be interpreted as representing the official views or policies of the Department of Defense or the U.S. Government. 

\bibliographystyle{siamplain}
\bibliography{peeling}

\appendix

\section{\label{sec:quantBound}A quantitative bound building on Proposition 3.11}

\begin{proposition}\label{prop:inclusionBoundStrong}
Suppose that $d$ is strict negative type on $[N]$ and $\mathcal{I} \subset \mathcal{J} = \textnormal{peel}(d)$. For an enumeration $k_1,\dots,k_K$ of the points in $\mathcal{J} - \mathcal{I}$, define $\mathcal{J}_0 := \mathcal{J}$ and $\mathcal{J}_{\ell} := \mathcal{J}_{\ell-1} - \{k_{\ell}\}$ for $\ell \in [K]$, so that $\mathcal{J}_K = \mathcal{I}$. Write
$$d|_{\mathcal{J}_\ell} = \begin{pmatrix} d|_{\mathcal{J}_{\ell+1}} & \delta^{(\ell)} \\ (\delta^{(\ell)})^T & 0 \end{pmatrix}.$$ Finally, define 
$$\Delta^{(\ell)} := d|_{\mathcal{J}_{\ell+1}}^{-1} \delta^{(\ell)}, \quad \bar \Delta^{(\ell)} := \frac{1}{|\mathcal{J}|-\ell} 1^T \Delta^{(\ell)}, \quad \varepsilon^{(\ell)} := \Delta^{(\ell)} - \bar \Delta^{(\ell)} 1, \quad w^{(\ell)} := d|_{\mathcal{J}_{\ell}}^{-1} 1.$$ 
Then 
\begin{equation}\label{eq:inclusionBoundStrong}
    \left \| \varepsilon^{(\ell)} \right \|_\infty \overset{\forall \ell}{<} \bar \Delta^{(\ell)} + \frac{\min_{i \in \mathcal{J}_{\ell+1}} w^{(\ell)}_i}{w^{(\ell)}_{|\mathcal{J}|-\ell}} \Rightarrow \textnormal{peel}(d|_\mathcal{I}) = \mathcal{I}.
\end{equation}
\end{proposition}

\begin{proof}
    As in Proposition 3.11 from the main text, we have $w^{(\ell+1)} = w^{(\ell)}|_{\mathcal{J}_{\ell+1}} + w^{(\ell)}_{|\mathcal{J}|-\ell} \Delta^{(\ell)}$. Meanwhile, $\textnormal{peel}(d|_\mathcal{I}) = \mathcal{I}$ if $w^{(\ell)} > 0$ componentwise for all $\ell$. Equivalently, $$-\varepsilon^{(\ell)} \overset{\forall \ell}{<} \bar \Delta^{(\ell)} 1 + \frac{w^{(\ell)}|_{\mathcal{J}_{\ell+1}}}{w^{(\ell)}_{|\mathcal{J}|-\ell}}\Rightarrow \textnormal{peel}(d|_\mathcal{I}) = \mathcal{I}.$$ Taking worst-case bounds of each side of the family of inequalities yields the result.
\end{proof}

\section{\label{sec:another51}Another experiment along the lines of \S 5.1}

Figure \ref{fig:neighborhood_graphs_torus} is along the lines of Figure 11 from \S 5.1 in the main text.

\begin{figure}[htbp]
  \centering
  \includegraphics[trim = 0mm 15mm 0mm 15mm, clip, width=\textwidth,keepaspectratio]{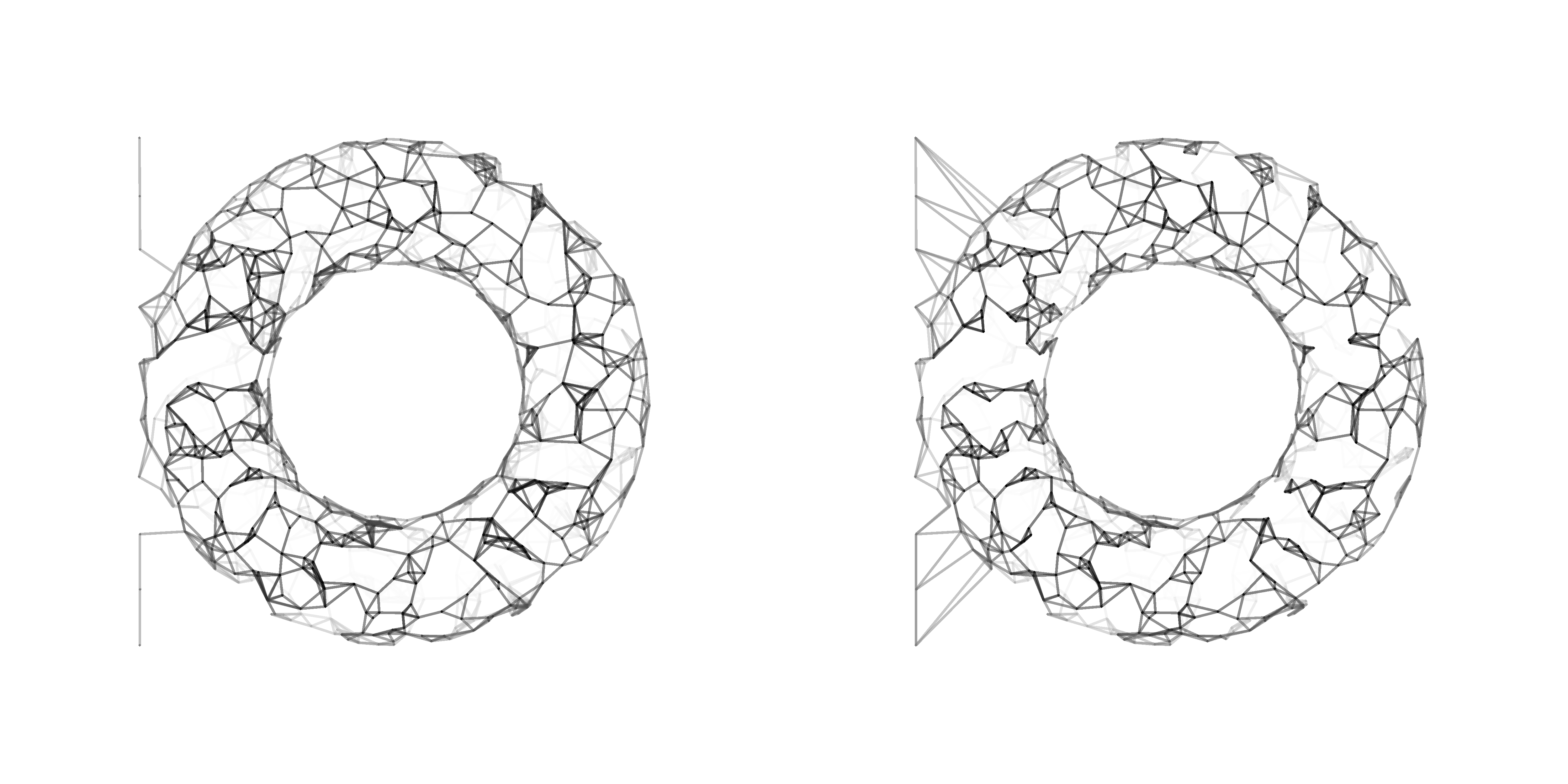}
  \includegraphics[trim = 15mm 0mm 15mm 0mm, clip, width=\textwidth,keepaspectratio]{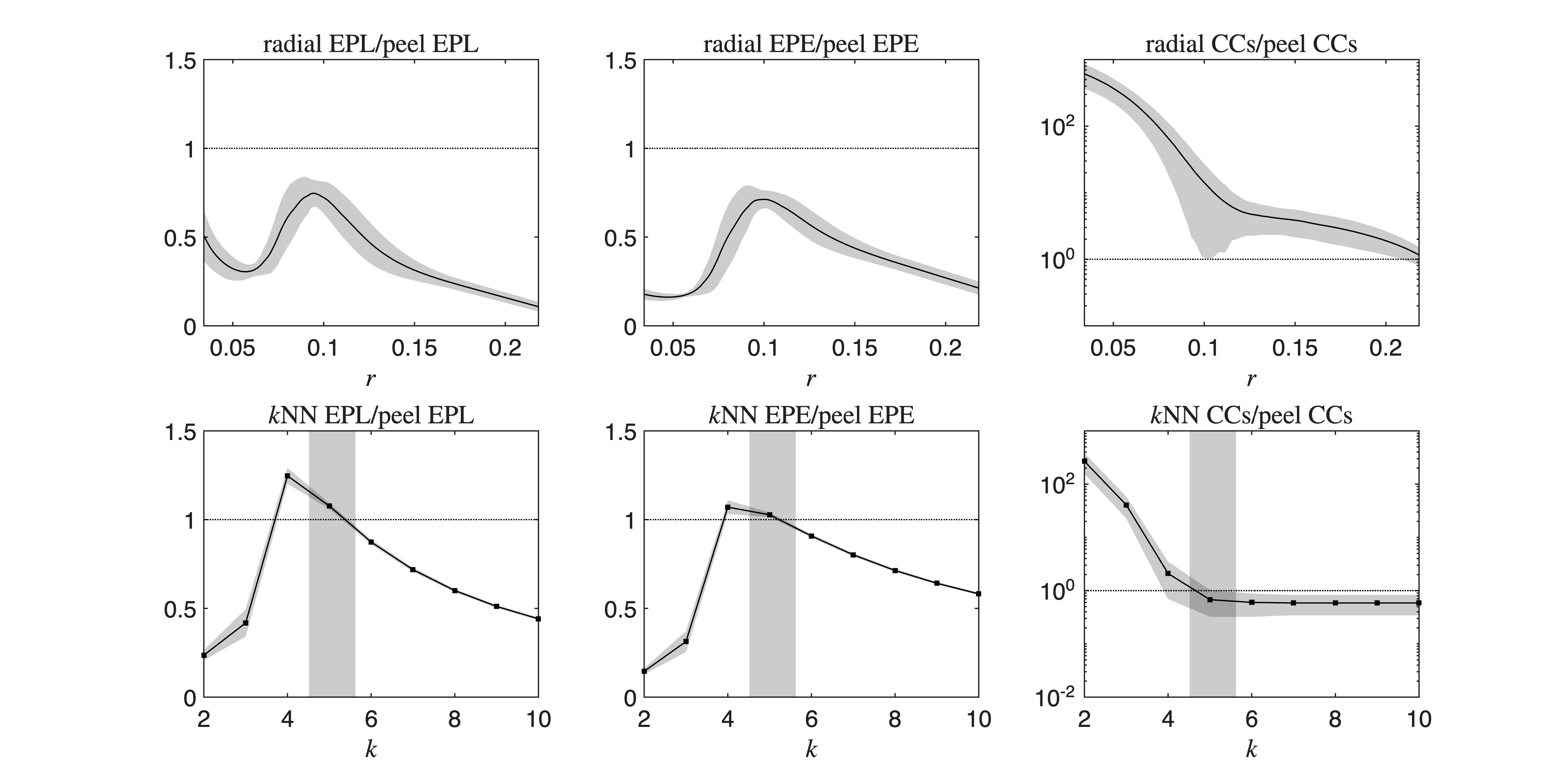}
\caption{Top left: the graph with edges given by peel neighborhoods on a uniform sample of approximately $1000$ points from a torus in $\mathbb{R}^3$, plus $11$ equispaced points slightly to the left, embedded in 10 dimensions with small Gaussian noise added. 
The exact neighborhoods and their heuristic approximations are always identical here.
Edges are shaded according to their average in the third dimension. Top right: the similar graph with edges given by $k$-nearest neighbors with the minimal $k$ such that the graph has a single connected component (here, $k = 5$). Middle: the relative efficiencies per edge and per length of radial graphs \emph{versus} peel neighborhood graphs, along with the relative numbers of connected components, all as functions of radius, with means and standard deviations over 100 trials indicated. Bottom: as in the middle panels, but for $k$NN graphs. Initial values of $k$ for which the $k$NN graph is fully connected are indicated by the vertical patch of $\pm$ a standard deviation about the mean.}
  \label{fig:neighborhood_graphs_torus}
\end{figure}

\section{Behavior of radii of thresholded peel neighborhoods from \S 5.2}\label{sec:annuli_radii}

Figures \ref{fig:peelNeighborhoodsRadiiMean} and \ref{fig:peelNeighborhoodsRadiiStd} show the behavior of radii of thresholded peel neighborhoods from the numerical experiment of \S 5.2 in the main text.

\begin{figure}[htbp]
    \centering
    \includegraphics[width=\textwidth, trim={0 0 0 0mm}, clip]{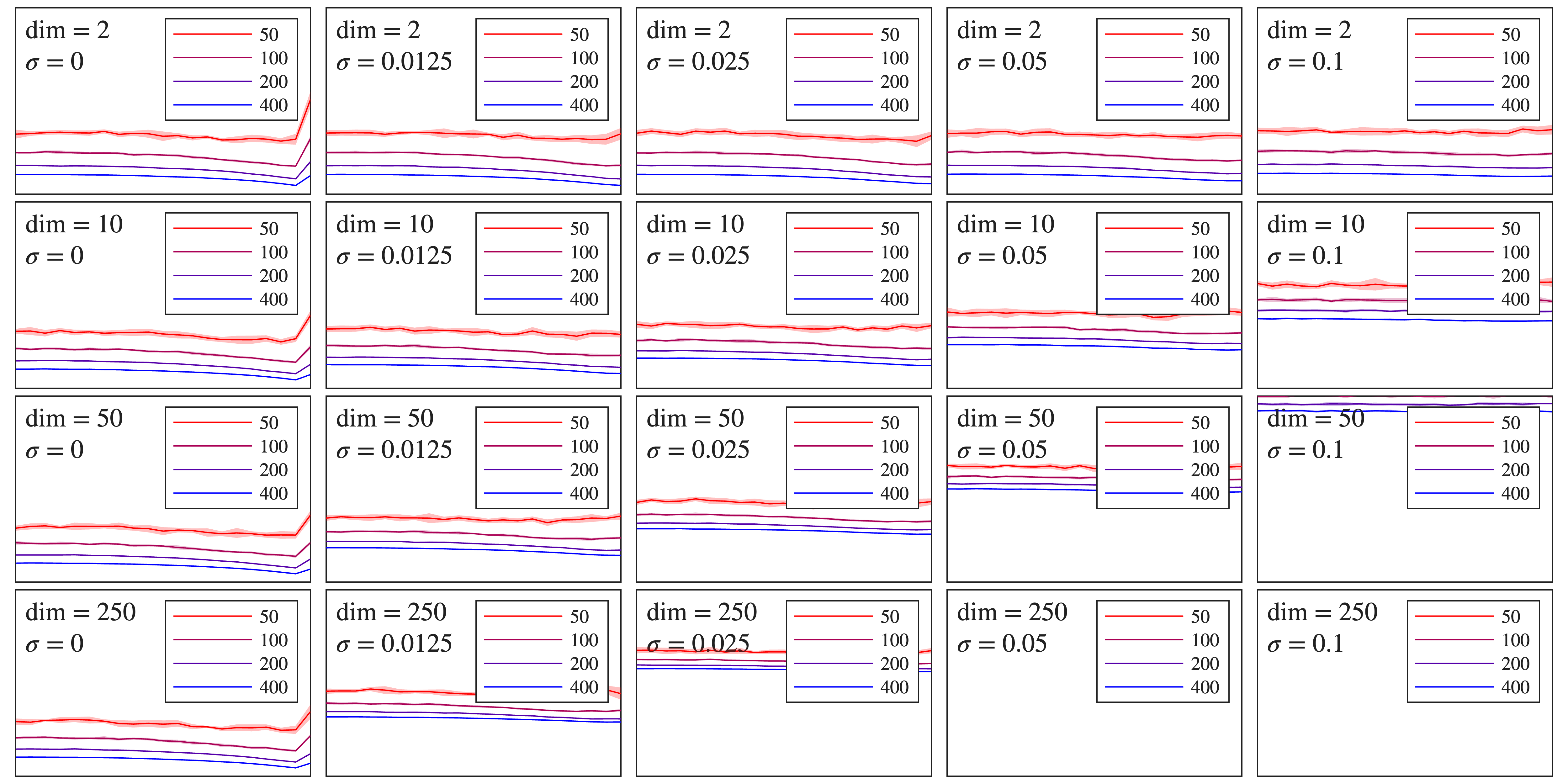}
        \caption{Average radii of thresholded peel neighborhoods (with standard deviations shown) in correspondence with Figure 13 in the main text. Each panel is over $[0,1] \times [0,1]$.
        }
    \label{fig:peelNeighborhoodsRadiiMean}
\end{figure}

\begin{figure}[htbp]
    \centering
    \includegraphics[width=\textwidth, trim={0 0 0 0mm}, clip]{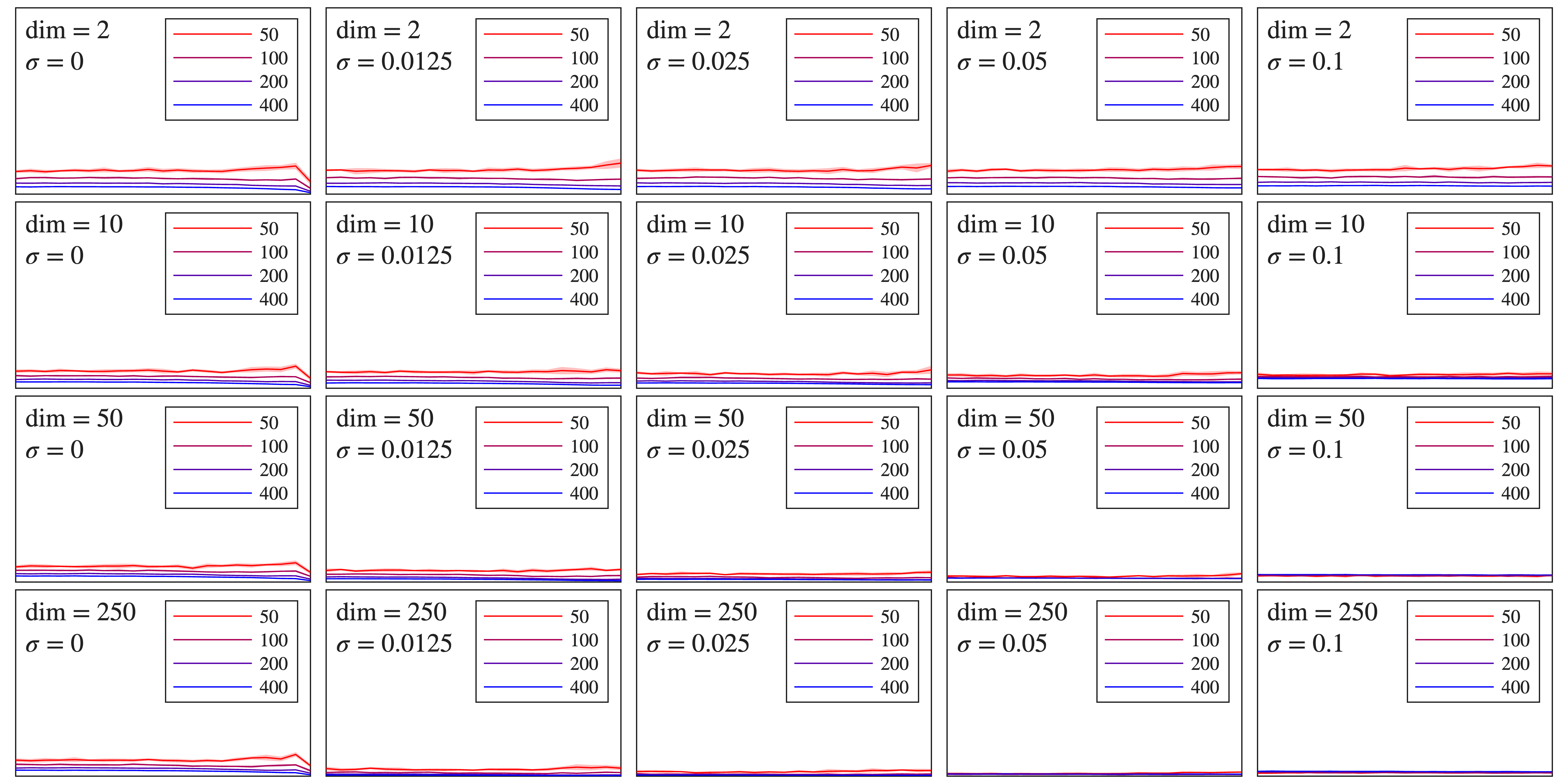}
        \caption{Standard deviations of radii of thresholded peel neighborhoods (with [their own] standard deviations shown) in correspondence with Figure 13 in the main text. Each panel is over $[0,1] \times [0,1]$. 
        }
    \label{fig:peelNeighborhoodsRadiiStd}
\end{figure}

\section{Analogues for $m=2$ of an experiment from \S 5.4}\label{sec:analogues54}

Figures \ref{fig:hair_ball_dim_3_attach_false_match_false}-\ref{fig:hair_ball_dim_3_attach_true_match_true} are direct respective analogues of Figures 20-23 in \S 5.4 in the main text, but for $m = 2$ instead of $m = 10$. 

\begin{figure}[htbp]
    \centering
    \includegraphics[width=.48\textwidth, trim={0 50 0 15mm}, clip]{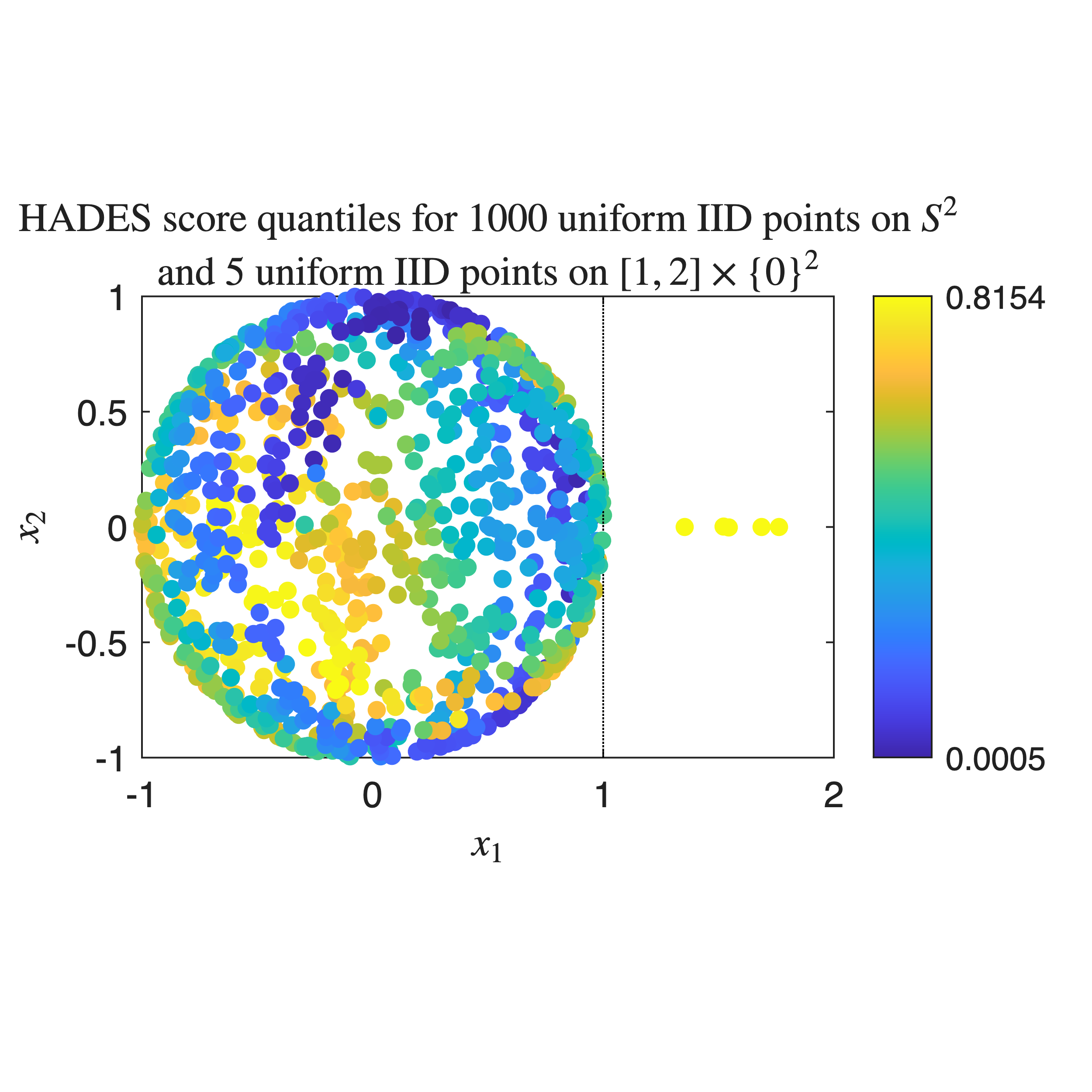}
    \includegraphics[width=.48\textwidth, trim={0 50 0 15mm}, clip]{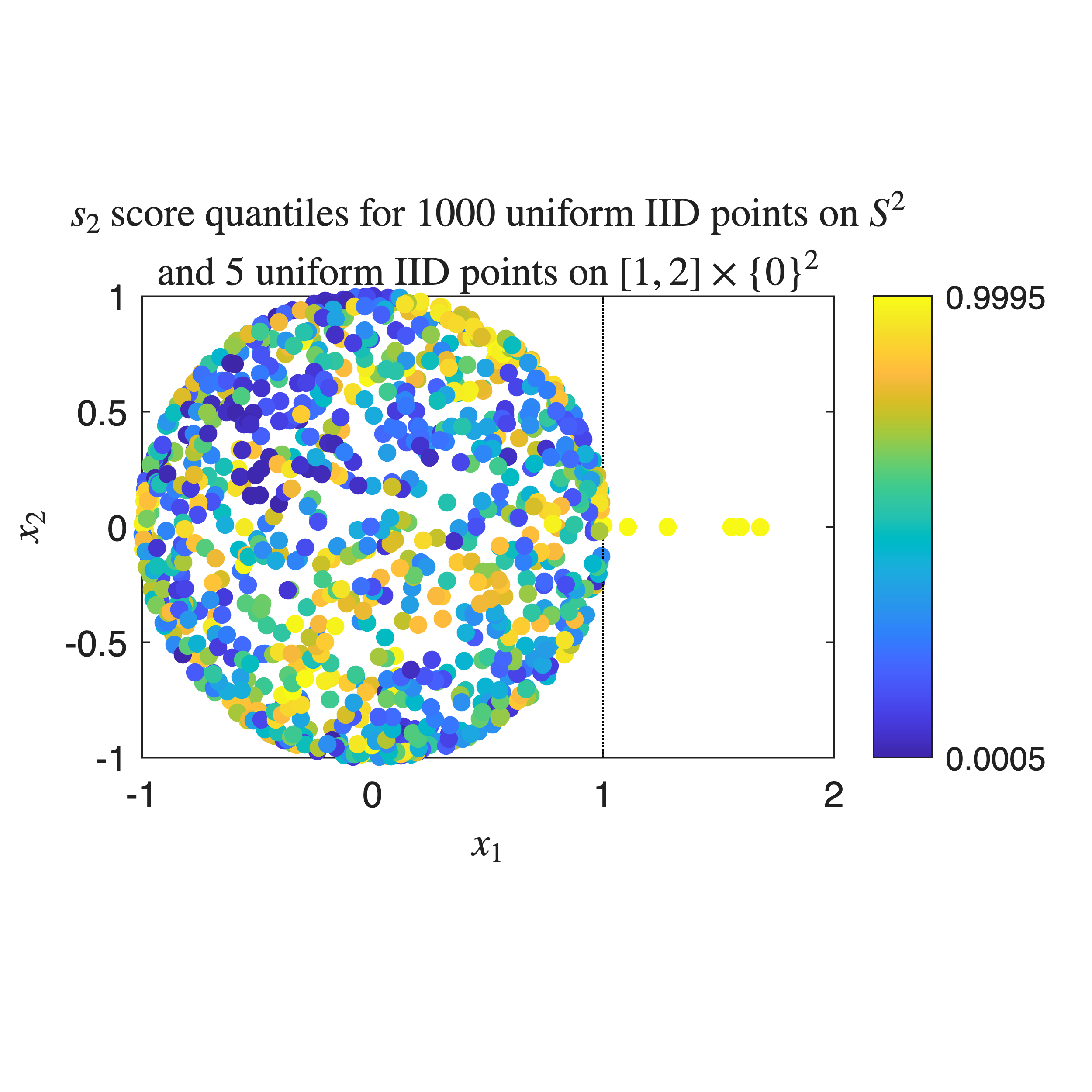}
    \includegraphics[width=.48\textwidth, trim={0 0 0 0mm}, clip]{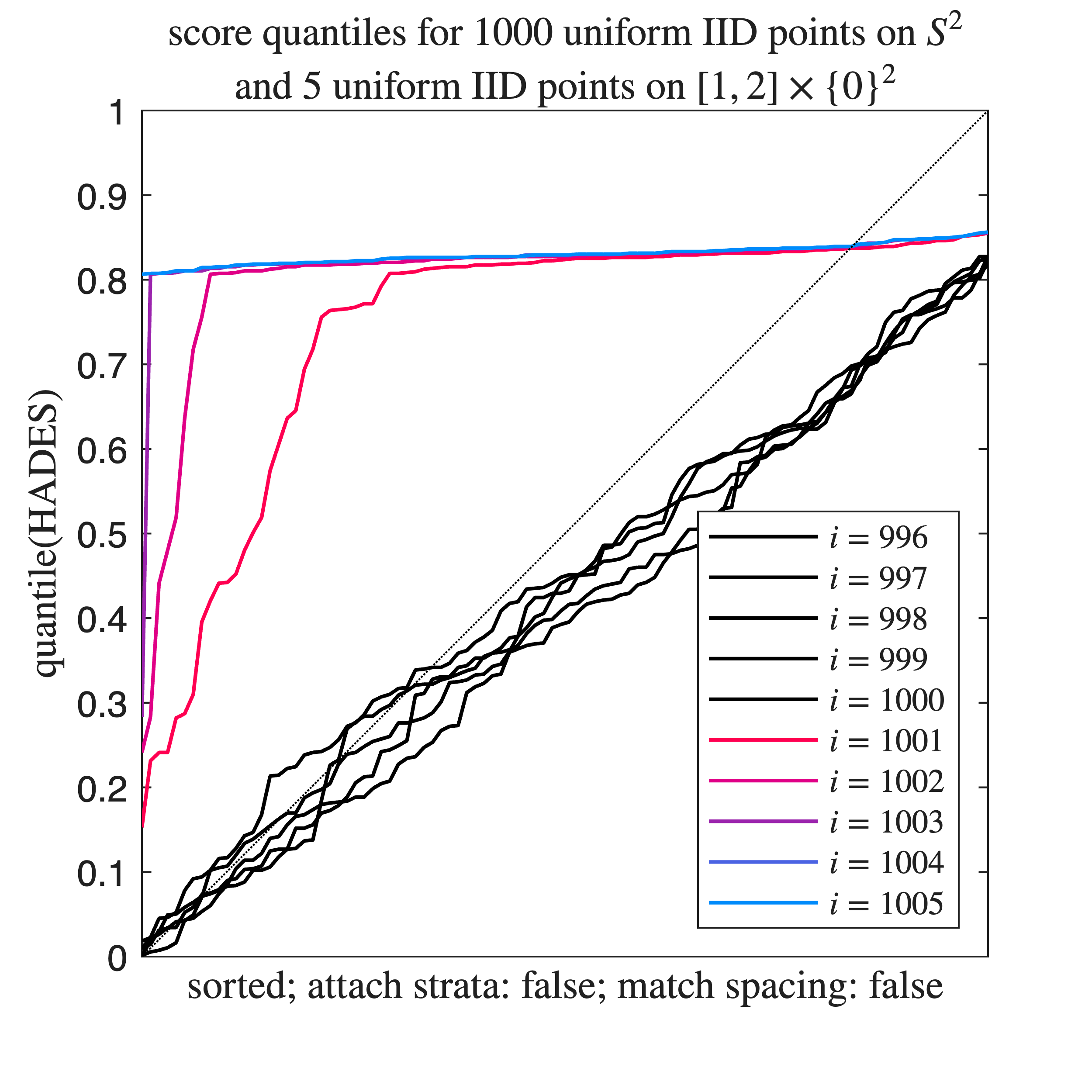}
    \includegraphics[width=.48\textwidth, trim={0 0 0 0mm}, clip]{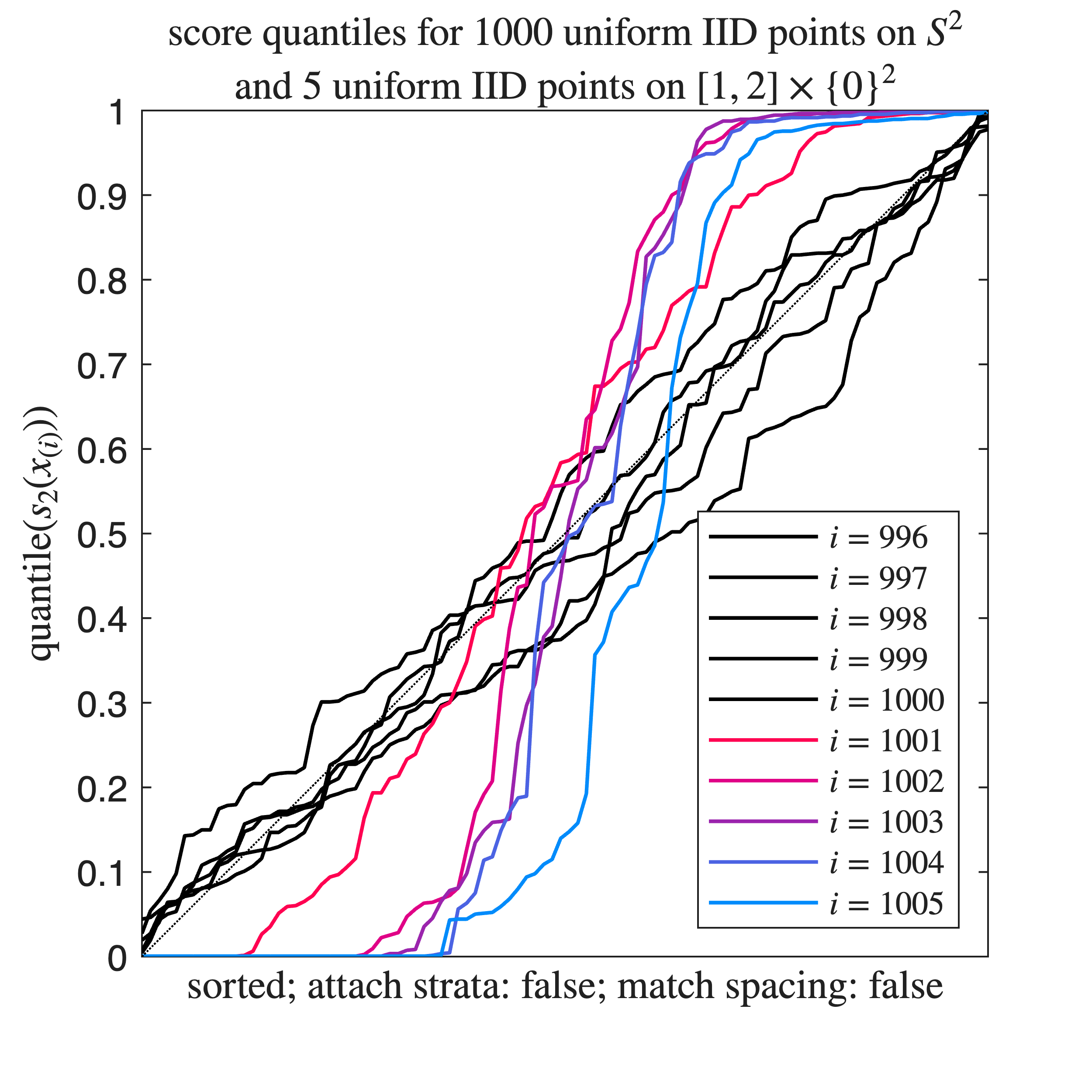}
        \caption{Top left: HADES score quantiles for a sample of size 1000 from $S^m$ for $m = 2$ and an ``unattached'' sample of size 5 from an interval of unit length, plus small but nonzero isotropic Gaussian noise after embedding in $\mathbb{R}^{100}$. Top right: as in the top left panel, but for $s_2$. Bottom left: HADES score quantiles computed over 100 sample realizations for five points on the sphere stratum (in black) and the five points on the interval stratum, indexed in order of their distance from the sphere. Bottom right: as in the bottom left panel, but for $s_2$. 
        In this case, $s_2$ data turn out to be identical for exact peels and heuristic approximations.}
    \label{fig:hair_ball_dim_3_attach_false_match_false}
\end{figure}

\begin{figure}[htbp]
    \centering
    \includegraphics[width=.48\textwidth, trim={0 50 0 15mm}, clip]{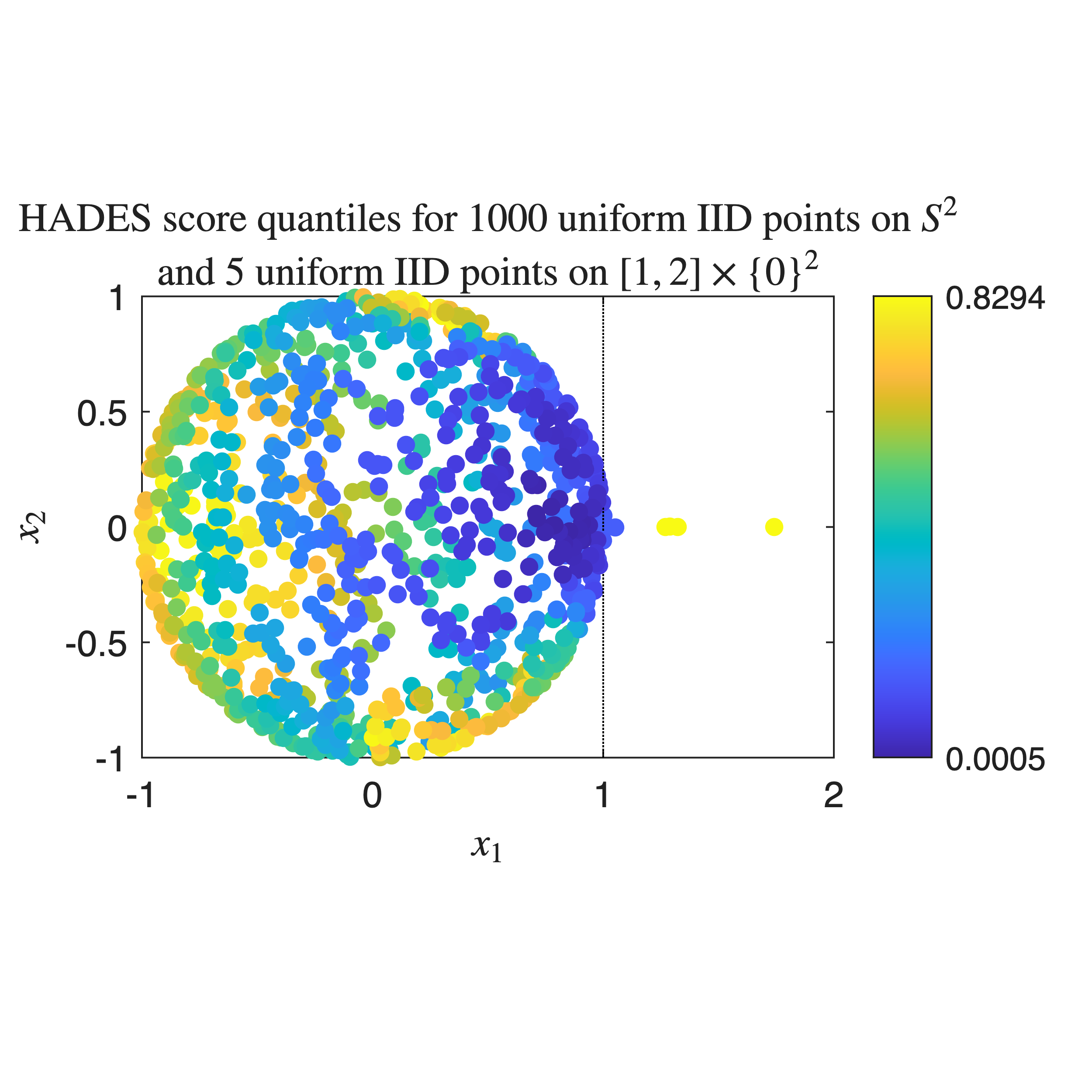}
    \includegraphics[width=.48\textwidth, trim={0 50 0 15mm}, clip]{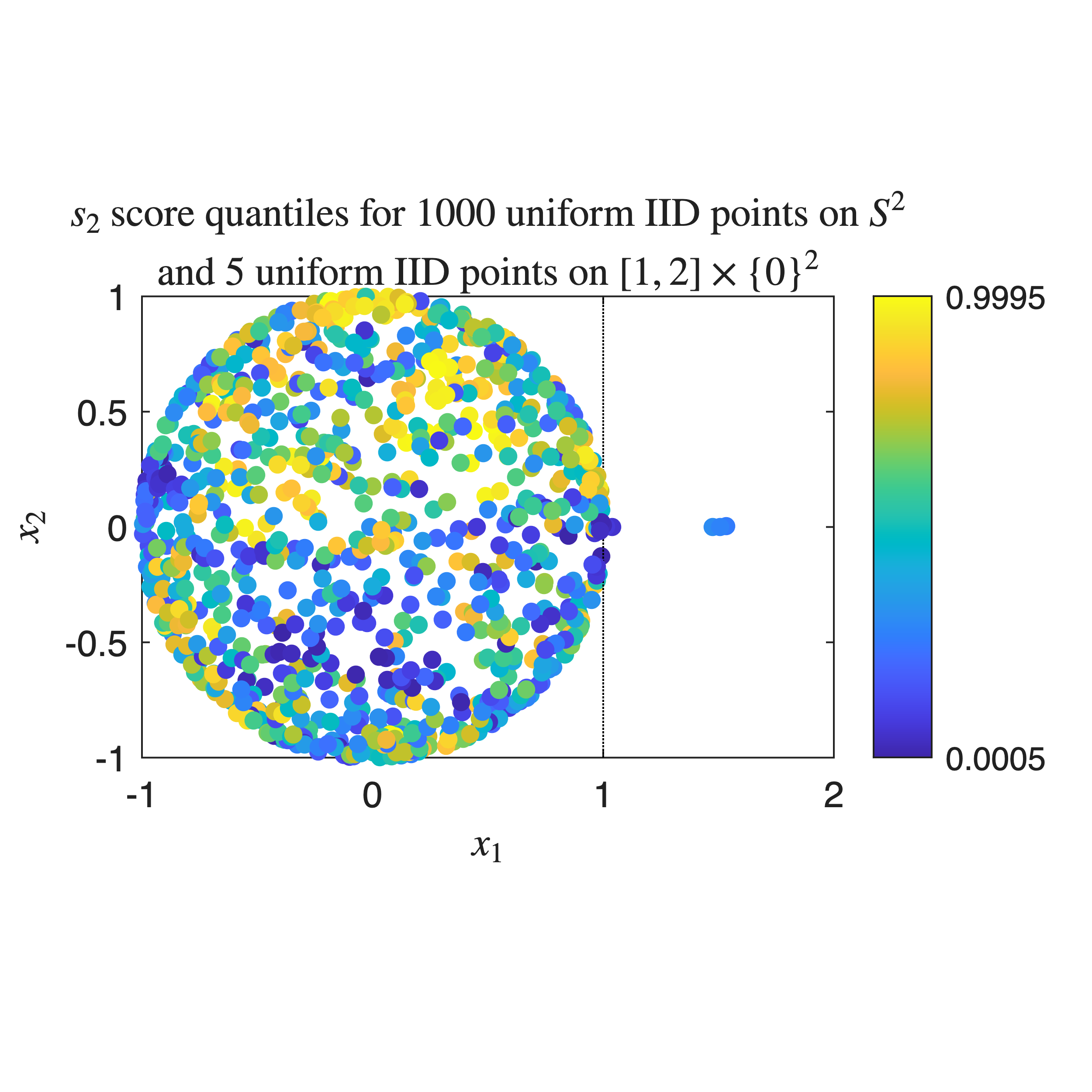}
    \includegraphics[width=.48\textwidth, trim={0 0 0 0mm}, clip]{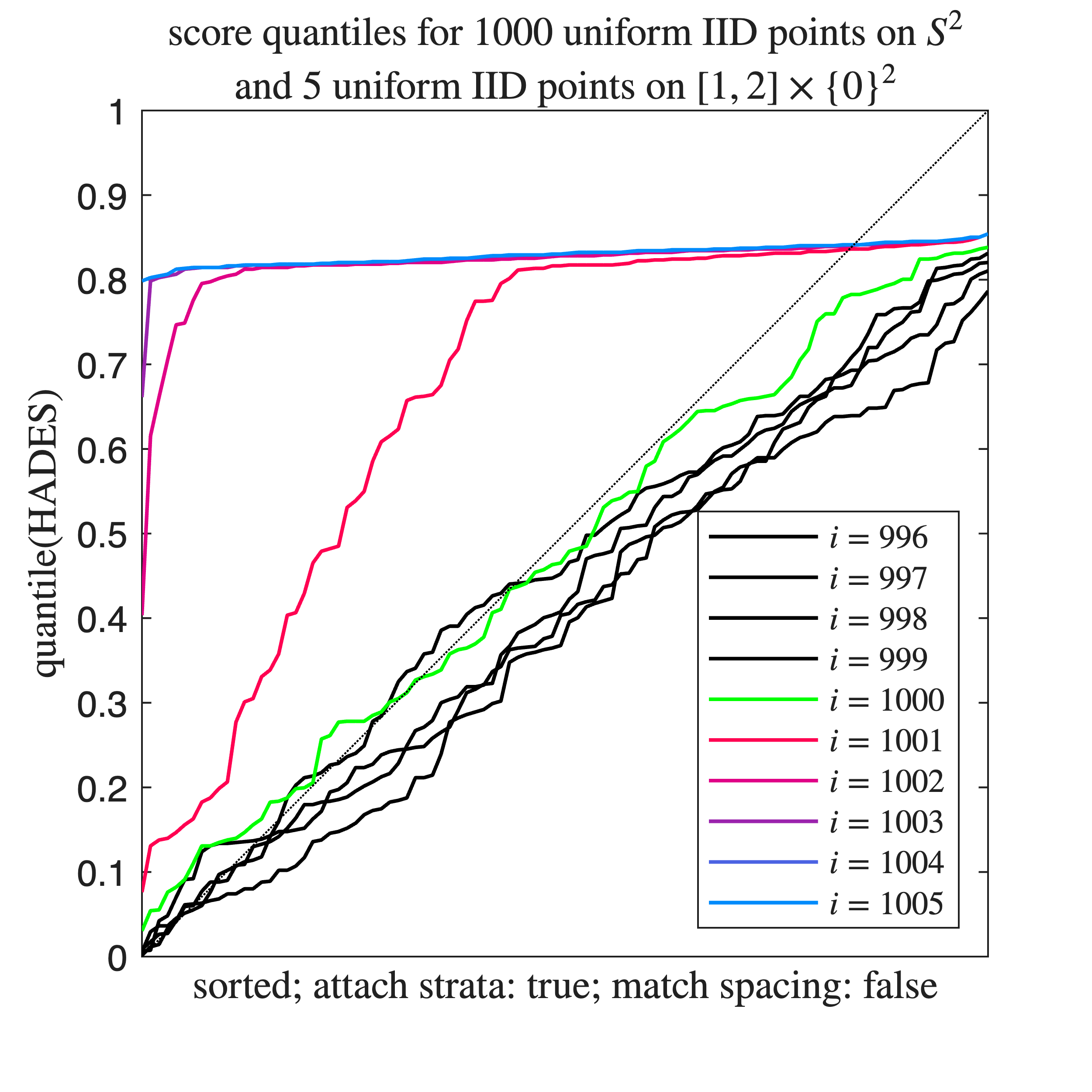}
    \includegraphics[width=.48\textwidth, trim={0 0 0 0mm}, clip]{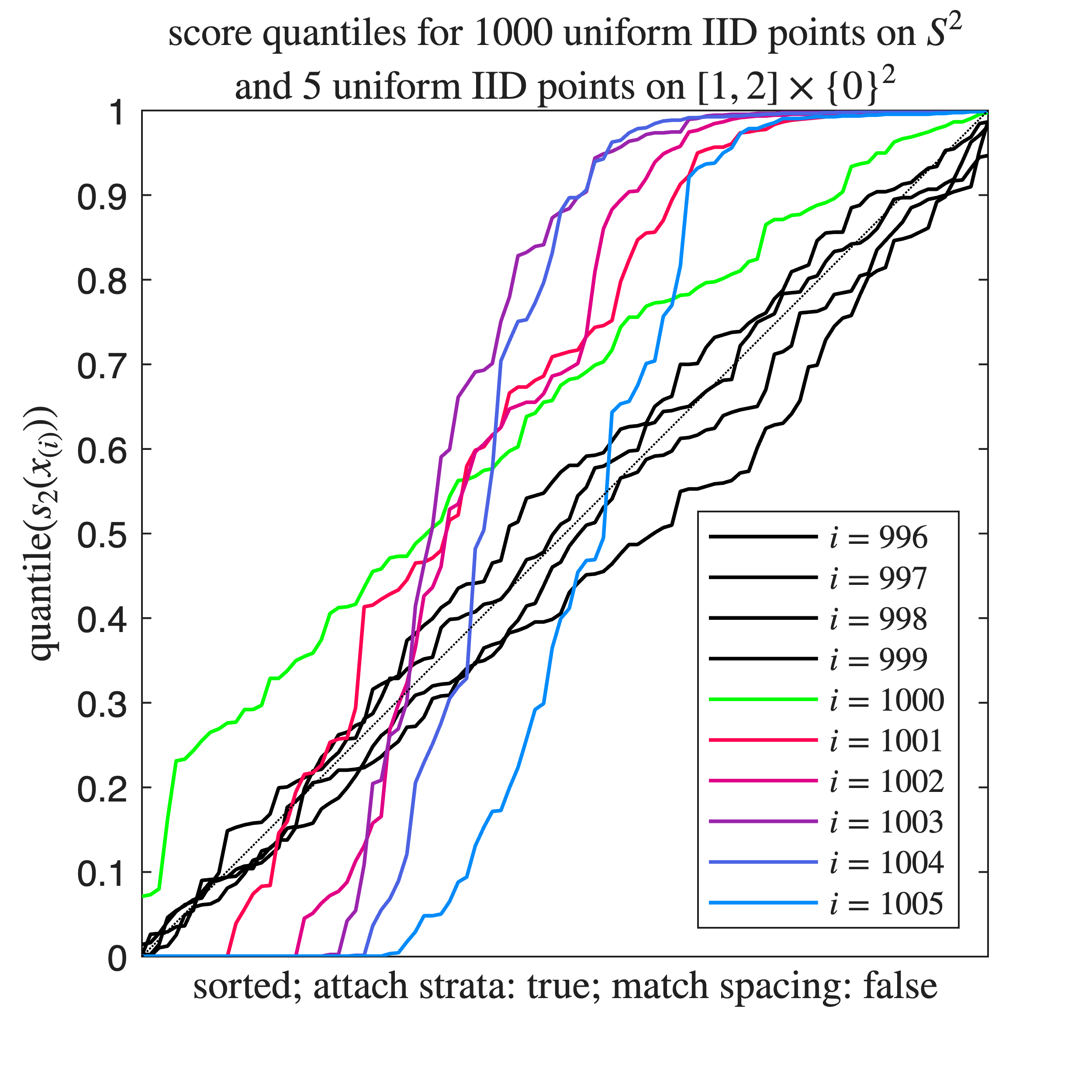}
        \caption{As in Figure \ref{fig:hair_ball_dim_3_attach_false_match_false}, but for an ``attached'' sample (with the attachment point indicated by a dashed line in the upper panels and {\color{green}in green in the lower panels}). 
        In this case, $s_2$ data turn out to be identical for exact peels and heuristic approximations.
        }
    \label{fig:hair_ball_dim_3_attach_true_match_false}
\end{figure}

\begin{figure}[htbp]
    \centering
    \includegraphics[width=.48\textwidth, trim={0 40 0 5mm}, clip]{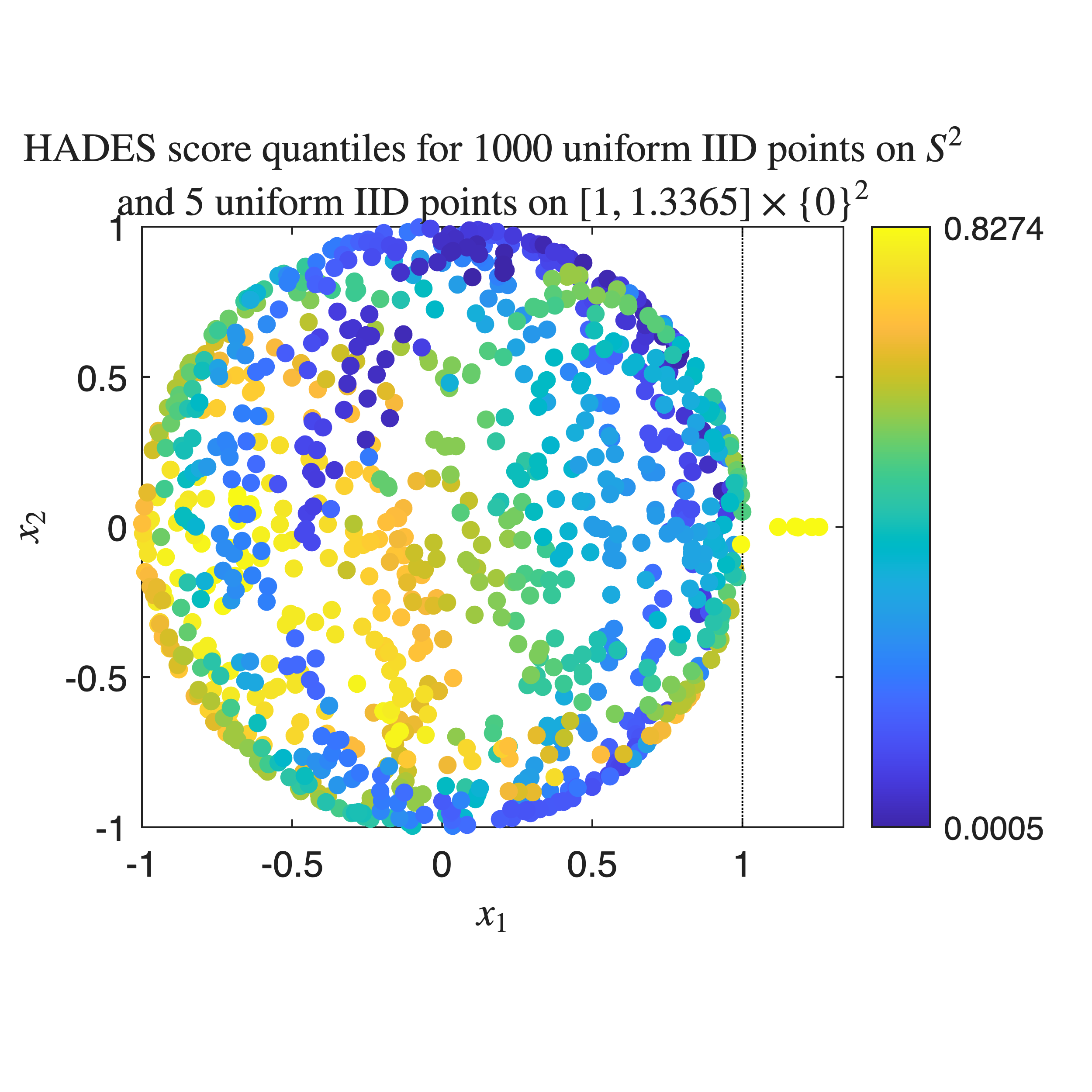}
    \includegraphics[width=.48\textwidth, trim={0 40 0 5mm}, clip]{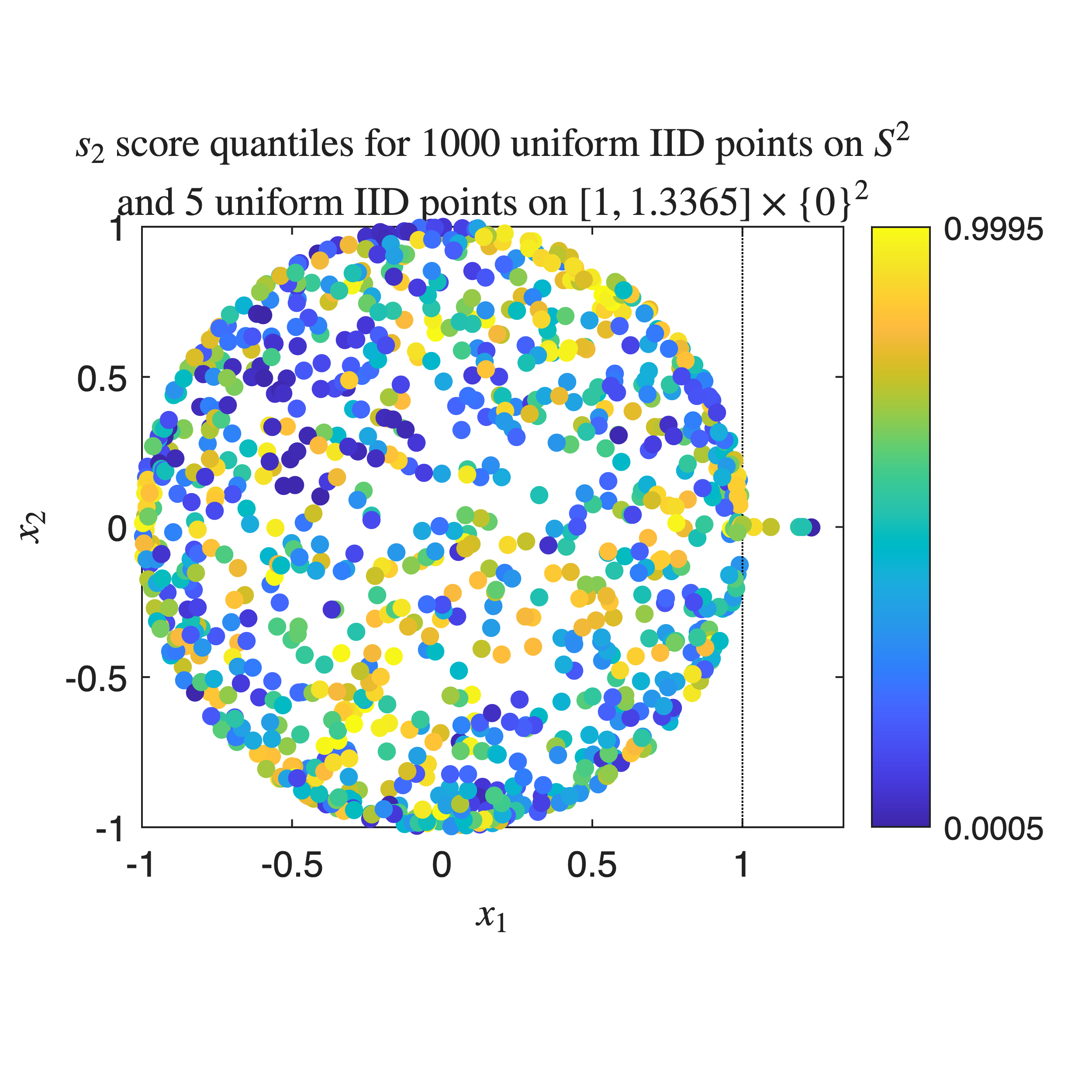}
    \includegraphics[width=.48\textwidth, trim={0 0 0 0mm}, clip]{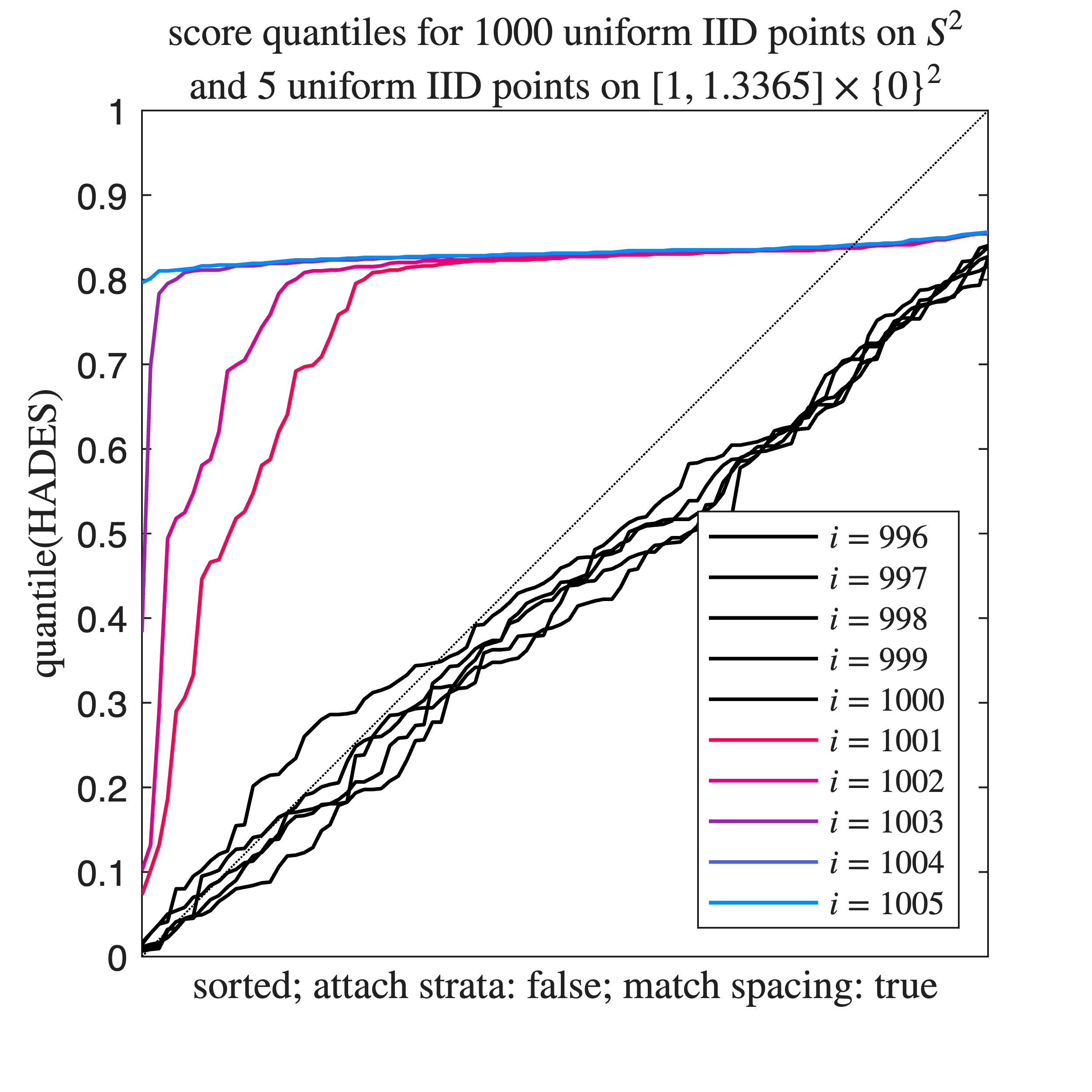}
    \includegraphics[width=.48\textwidth, trim={0 0 0 0mm}, clip]{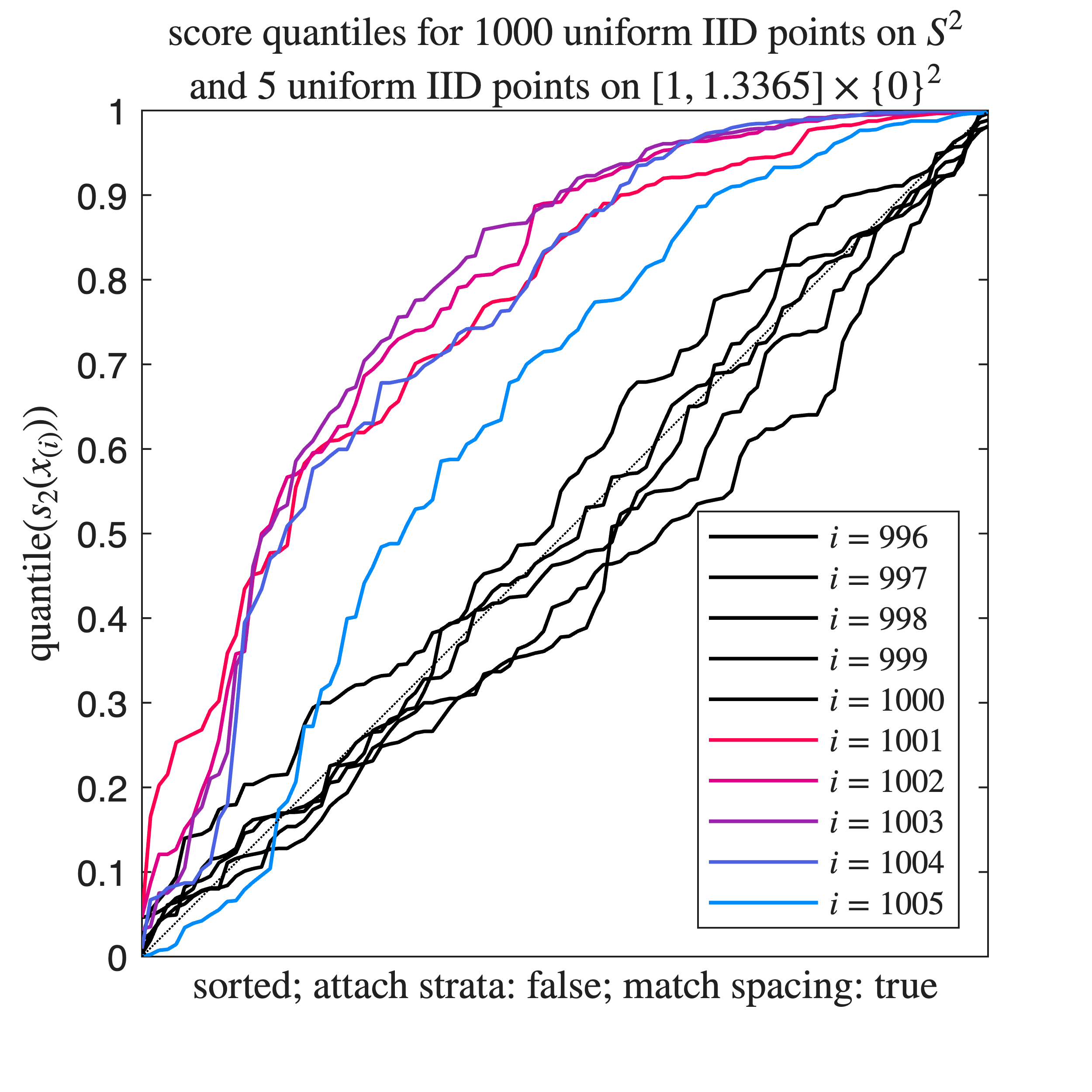}
        \caption{As in Figure \ref{fig:hair_ball_dim_3_attach_false_match_false}, but for a spacing-matched interval. 
        In this case, $s_2$ data turn out to be identical for exact peels and heuristic approximations.
        }
    \label{fig:hair_ball_dim_3_attach_false_match_true}
\end{figure}

\begin{figure}[htbp]
    \centering
    \includegraphics[width=.48\textwidth, trim={0 40 0 5mm}, clip]{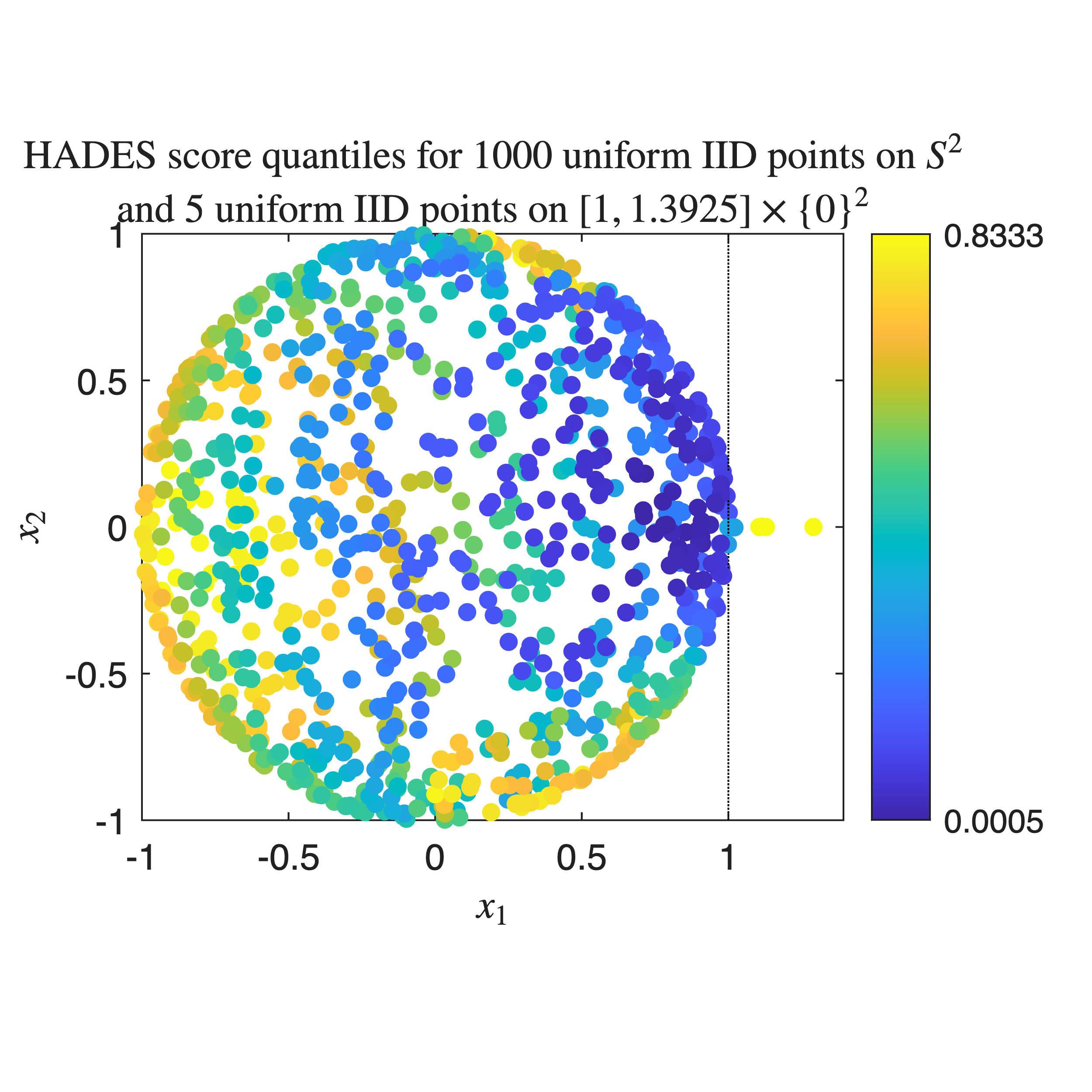}
    \includegraphics[width=.48\textwidth, trim={0 40 0 5mm}, clip]{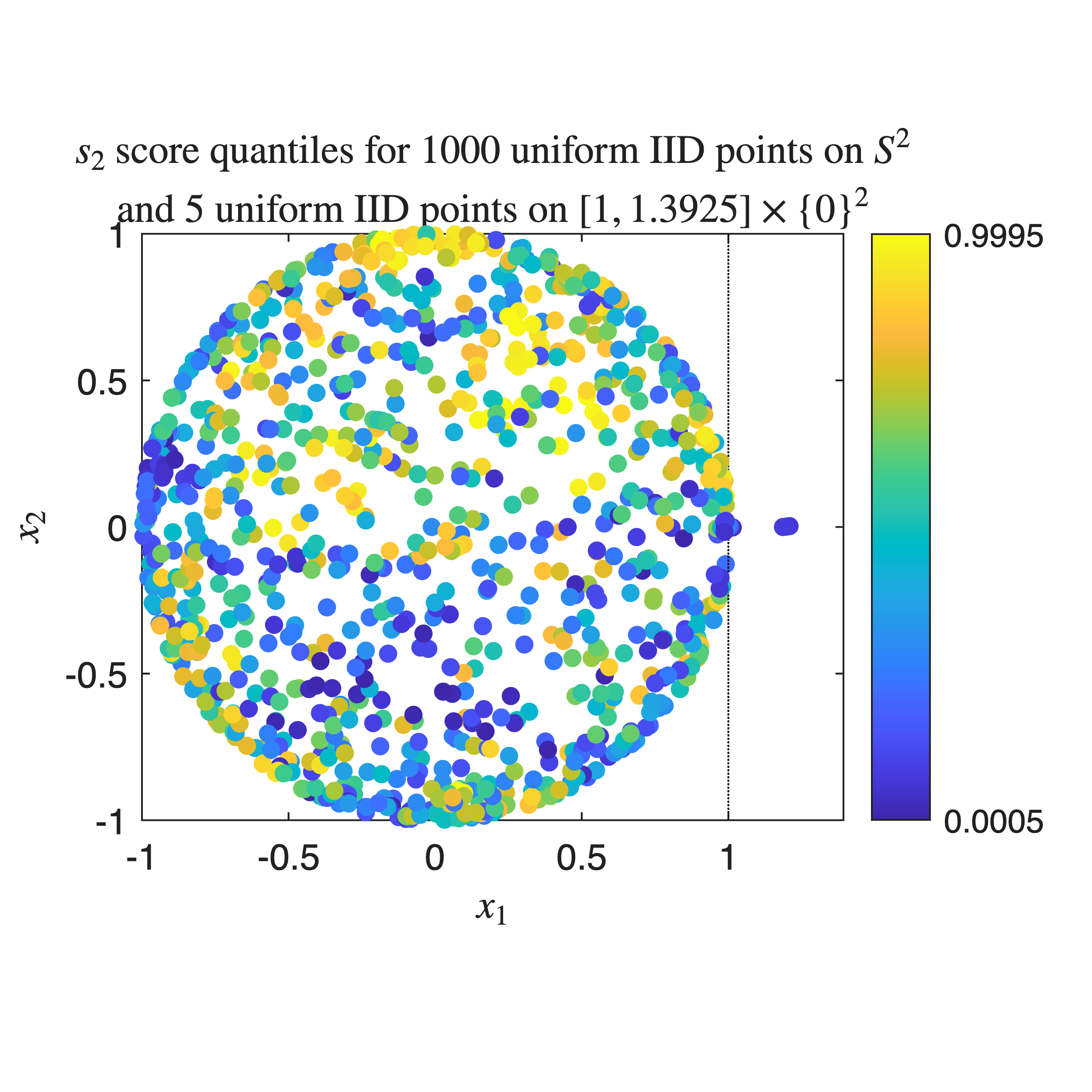}
    \includegraphics[width=.48\textwidth, trim={0 0 0 0mm}, clip]{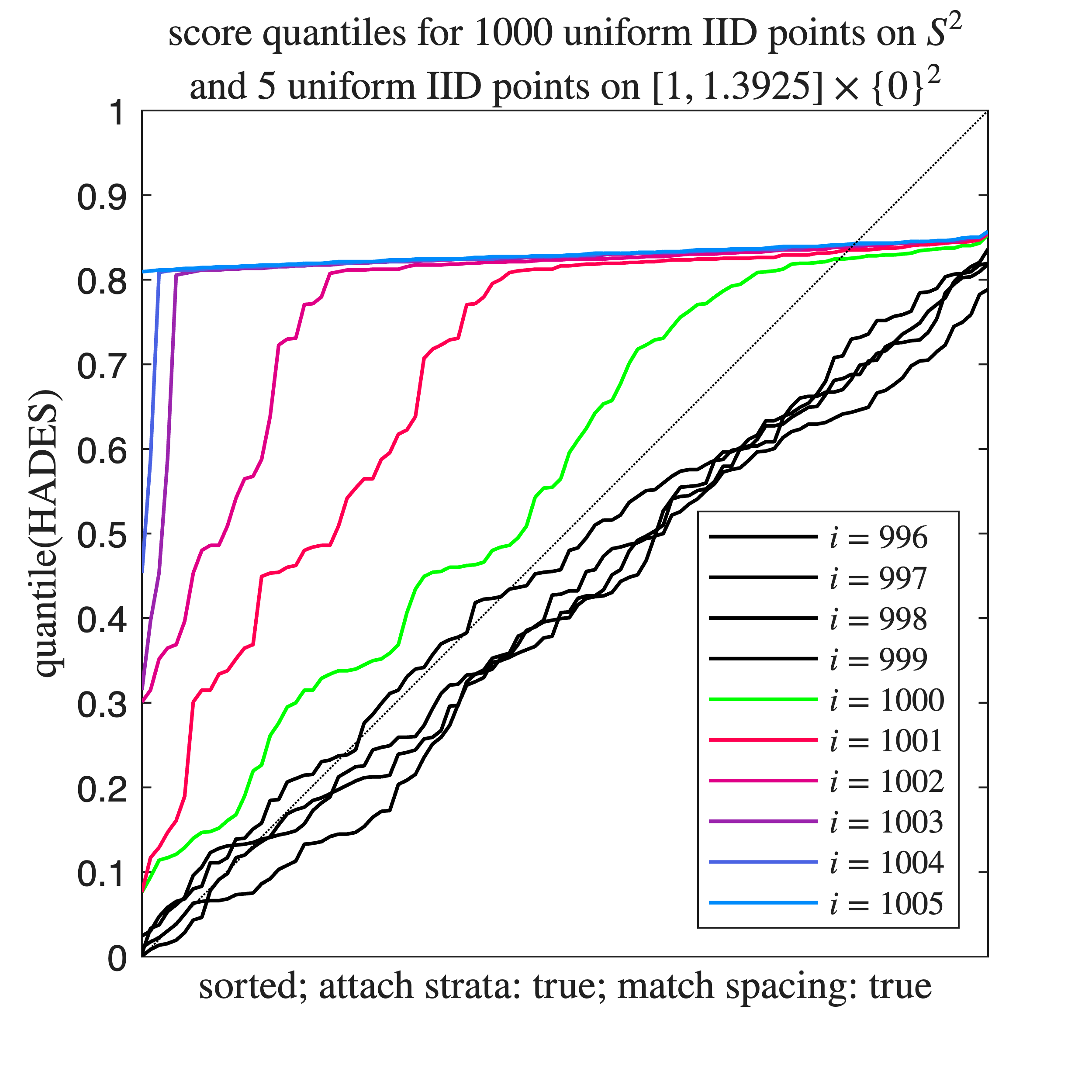}
    \includegraphics[width=.48\textwidth, trim={0 0 0 0mm}, clip]{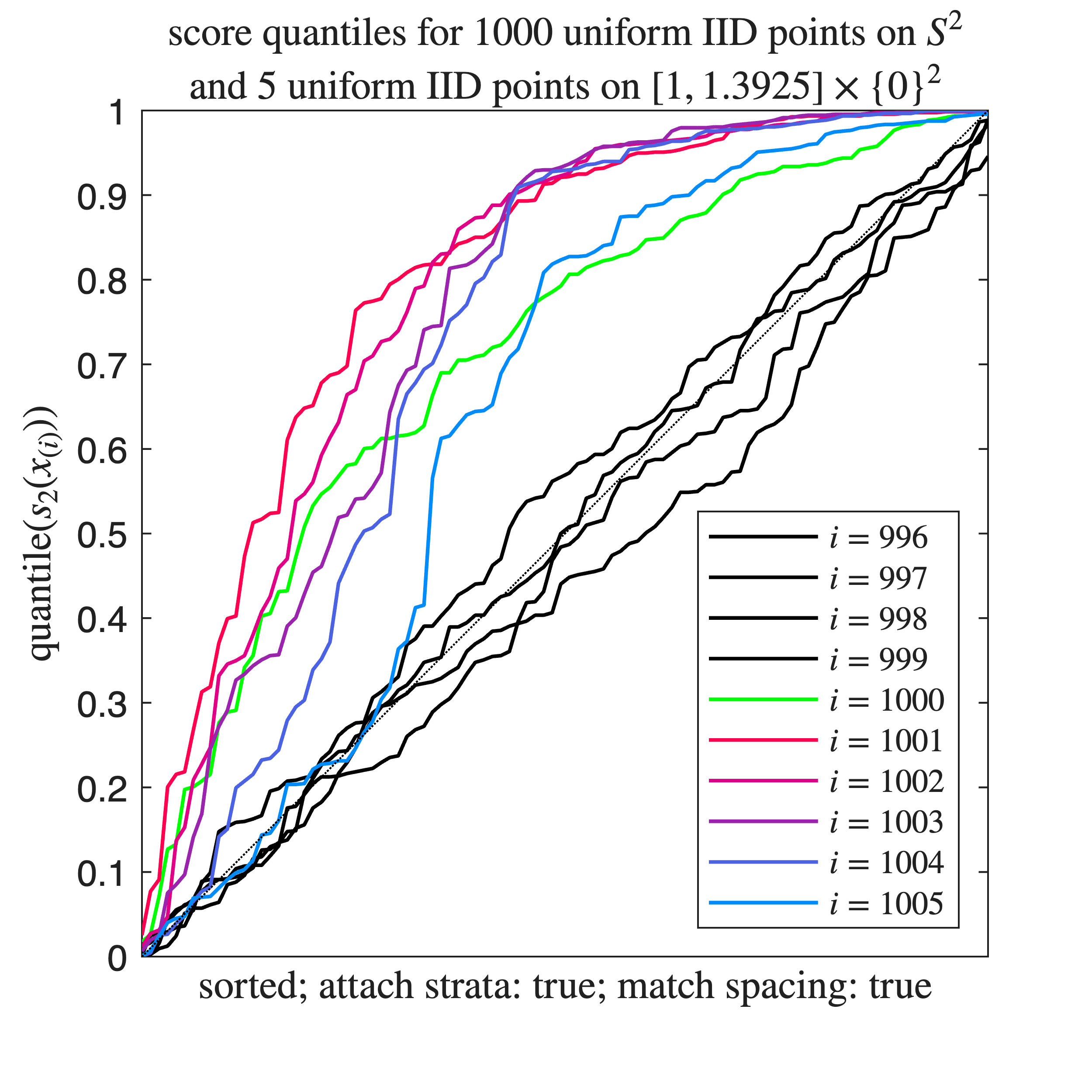}
        \caption{As in Figure \ref{fig:hair_ball_dim_3_attach_false_match_false}, but for an ``attached'' sample (with the attachment point indicated by a dashed line in the upper panels and {\color{green}in green in the lower panels}) and a spacing-matched interval. 
        In this case, $s_2$ data turn out to be identical for exact peels and heuristic approximations.
        }
    \label{fig:hair_ball_dim_3_attach_true_match_true}
\end{figure}

For lower data dimensions of the bulk strata $S^m$ such as here, both techniques experience degradations in performance. HADES loses discrimination power as its scores saturate around the 0.85 quantile, while $s_2$ loses predictive power. For attached strata $s_2$ clearly still outperforms HADES at the precise singularity point, while HADES identifies outliers on both strata. In low dimension, the number of outliers is actually higher, since more points are near-antipodes to the attachment point than in high dimension. In other words, HADES' tendency to identify outliers as singularities causes its low discrimination power in low dimension. Meanwhile, the dimension gradient is fundamentally smaller in this regime. On the other hand, the peel neighborhood of the point with largest $x_1$ will always saturate a threshold by construction, while the overall number of points whose peel neighborhood radii exceed the threshold in this experiment was around 6 percent, i.e., fewer than the roughly 15 percent points whose HADES scores saturated. From this perspective peel neighborhoods considered broadly still enable competitive singularity detection capability, and with faster runtime.

\section{Sampling uniformly from the Bolza surface of genus 2}\label{sec:bolza}

Since hyperbolic space is strict negative type by Corollary 4.2 of \cite{hjorth2002hyperbolic}, we can consider peel neighborhoods in finite subsets of hyperbolic space, say in the Poincar\'e disk model. While Theorem 5.4 of \cite{hjorth2002hyperbolic} also states that a compact Riemannian manifold of negative type must be simply connected, we can still compute peel neighborhoods of sufficiently dense samples from such a manifold in practice without trepidation.

In principle, we could sample uniformly from compact manifolds of constant negative curvature and compute distances on the result. The simplest way would be to construct a fundamental domain in the Poincar\'e disk, rejection sample (in high dimension, inefficiently) over the fundamental domain with respect to the hyperbolic measure over an enclosing disk, and pay careful attention to orbits under the corresponding group structure, taking the minimum of hyperbolic distances $2 \tanh^{-1}|(z-z')/(1-\bar zz')|$ over all the sufficiently close orbits of either argument \cite{van2021ollivier}. However, the problem of algorithmically computing distances on Riemann surfaces has (remarkably) only recently been addressed in any generality \cite{stepanyants2024computing}, and doing something similar for dimension $> 2$ is presently impractical.  

As a compromise, we restrict consideration to the \emph{Bolza surface} of genus 2 \cite{balazs1986chaos,van2021ollivier,stepanyants2024computing}. This has the advantage that the program above can be implemented fairly straightforwardly, albeit at the cost of not generalizing. \footnote{See p. 139 of \cite{balazs1986chaos} for complete details of the fundamental domain: the other necessary details are in \S C.1 of \cite{van2021ollivier}, with the minor caveat that a cyclic index $\textnormal{mod } 8$ is implicit in equation (C9) therein.} 

Rejection sampling from the surface is straightforward: the fundamental domain in $\mathbb{R}^2$ is the unit disk minus eight equispaced circles of radius $((2^{1/2}-1)/2)^{1/2}$ whose centers are at distances $((2^{1/2}+1)/2)^{1/2}$ from the origin. The bounding circumradius is $R_* := 2^{-1/4}$, and inverse transform sampling \cite{devroye1986general} gives the desired random radius as $r \sim (1-1/(1+cU))^{1/2}$, where $c = R_*^2/(1-R_*^2)$ and $U \sim U([0,1])$. \footnote{The area element on the Poincar\'e disk is $4 r (1-r^2)^{-2} \ dr \ d\theta$ (see Theorem 4.5.6 of \cite{ratcliffe2006foundations}), and $\int_0^R 4r(1-r^2)^{-2} \ dr = 2R^2/(1-R^2)$, which yields the inverse transform sampling result.} The result of rejection sampling and computing nearby orbits is shown in the left panel of Figure \ref{fig:bolzaDisk}. The corresponding peel neighborhoods are shown in the right panel.

\begin{figure}[htbp]
    \centering
    \includegraphics[width=.48\textwidth, trim={150 45 120 5mm}, clip]{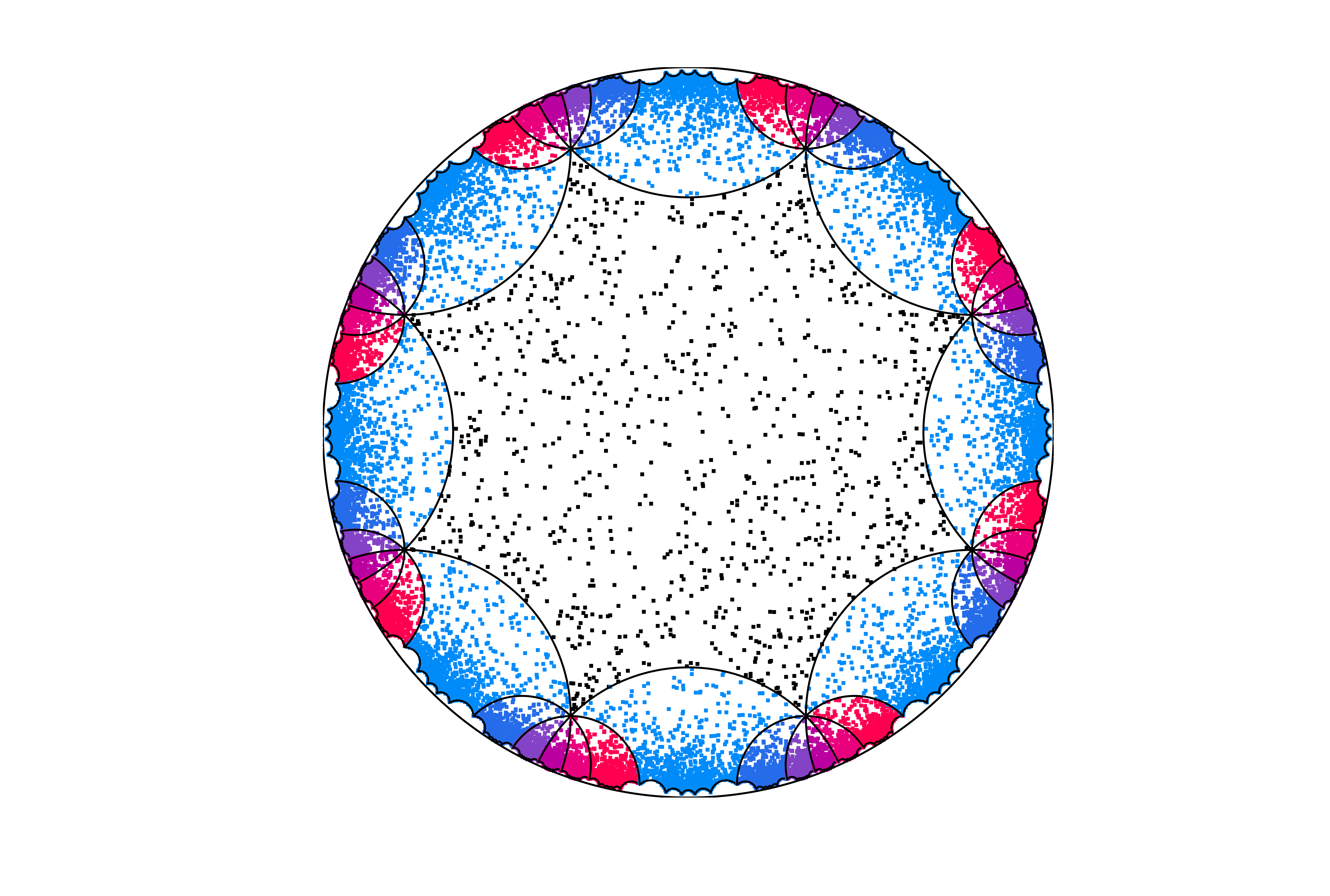}
    \includegraphics[width=.48\textwidth, trim={150 45 120 5mm}, clip]{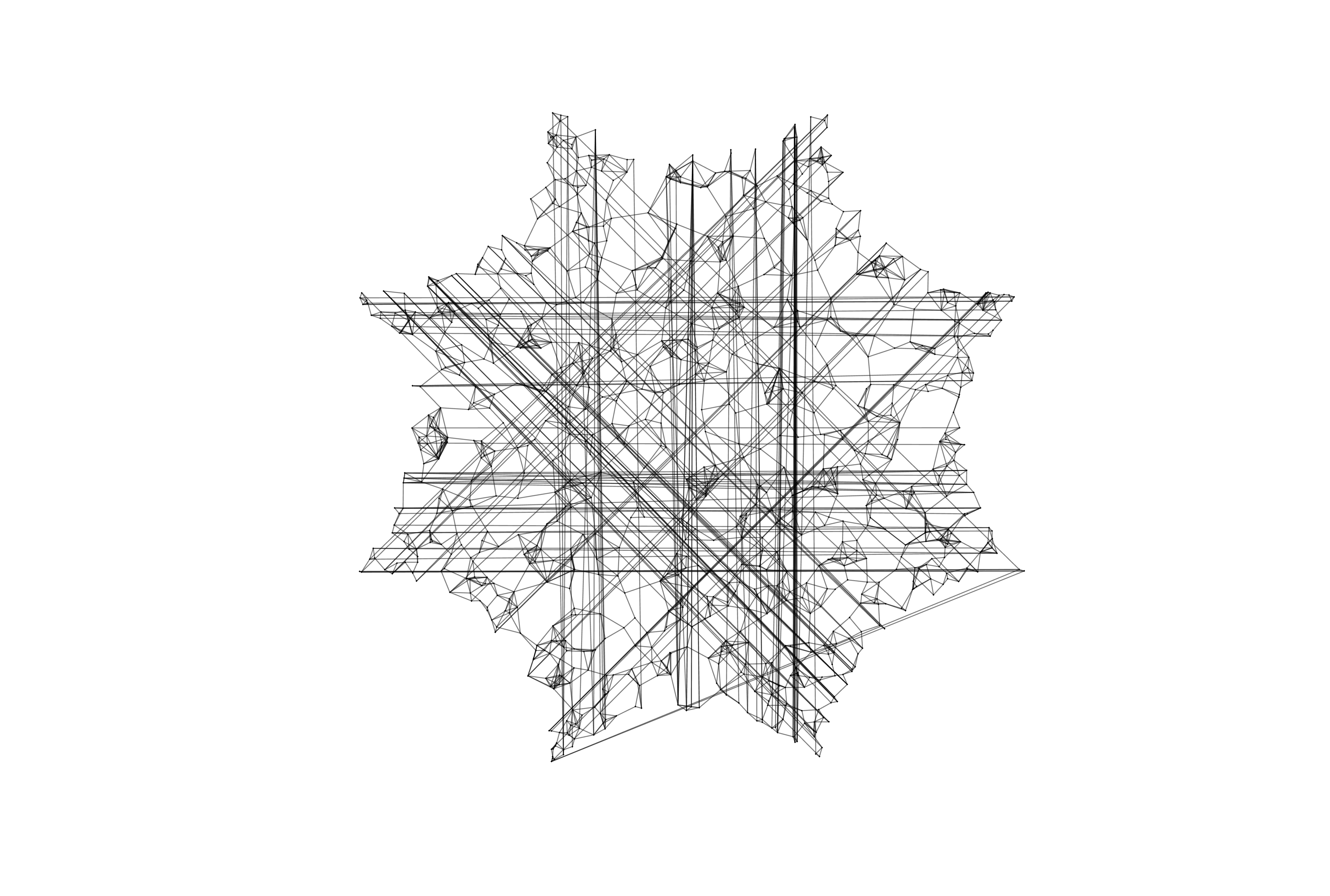}
        \caption{Left: in black, a sample of 1000 IID uniform points from the fundamental domain of the Bolza surface in the Poincar\'e disk. Colors from {\color{red}red} to {\color{blue}blue} indicate translates of the sample to the 48 edge- and vertex-adjacent domains, outlined in black. These translations suffice to compute distances on the Bolza surface via the usual metric on the Poincar\'e disk. Right: the undirected graph on the sample points obtained with peel neighborhoods along the lines of prior figures. The periodic boundary conditions are evident through long graph edges. 
        In this case, the exact peel neighborhoods and heuristic approximations are identical.
        }
    \label{fig:bolzaDisk}
\end{figure}

\section{Peel neighborhoods for MNIST demonstrate scalability}\label{sec:mnist}

Figure \ref{fig:MNIST_giant_component} shows most of the thresholded peel neighborhoods for MNIST using Euclidean distance. Thresholded peel neighborhoods can be readily computed on hundreds of thousands of points in thousands of dimensions.

\begin{figure}[htbp]
  \centering
  \includegraphics[trim = 40mm 15mm 20mm 10mm, clip, width=\textwidth,keepaspectratio]{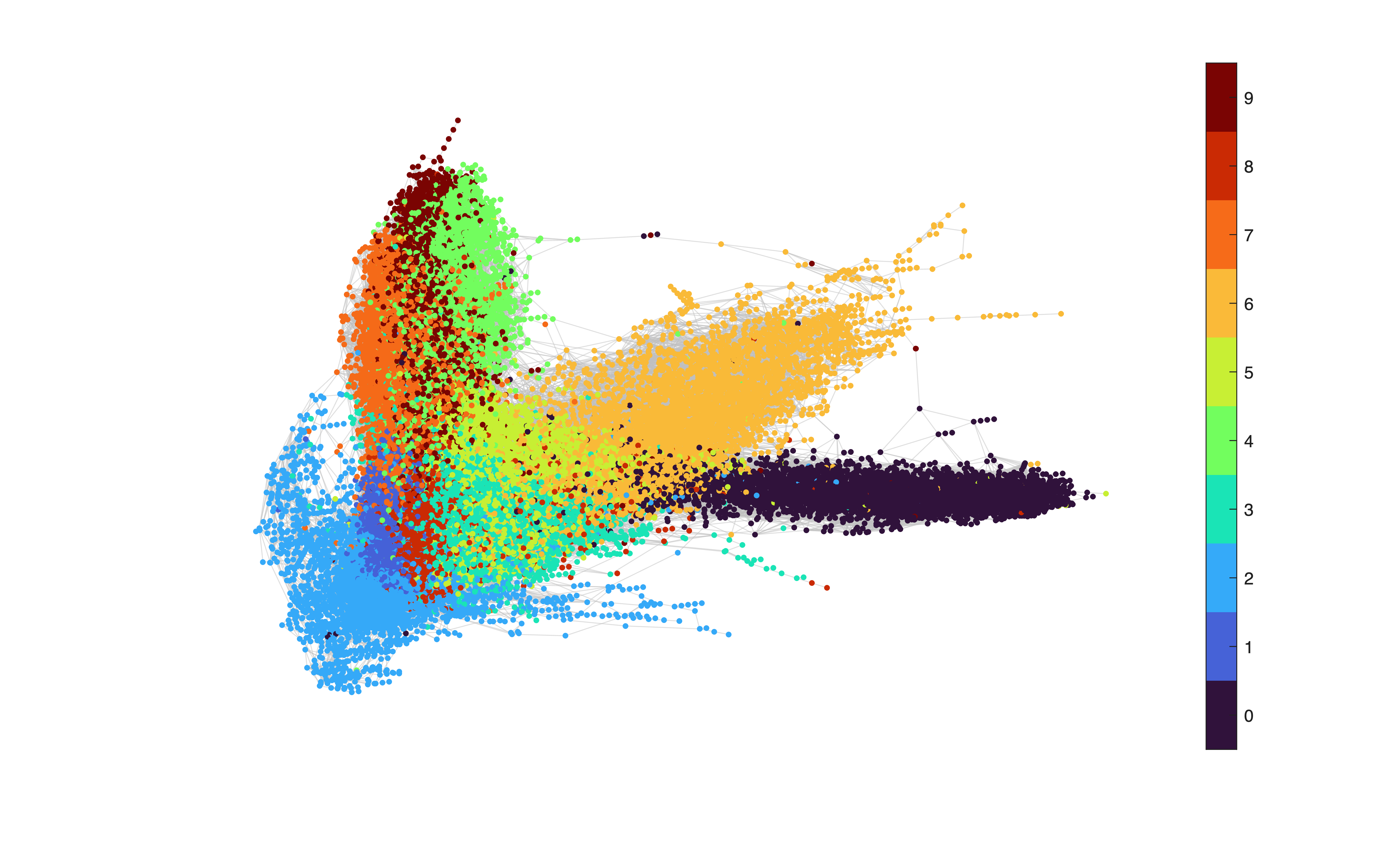}
\caption{The giant component of the peel neighborhood graph along the lines of Figure 7 in the main text but for MNIST training data, colored by digit. Using the default $(1+1/m) \times$ median distance threshold discussed in \S 4 in the main text, with $m = 28^2$, this graph has 57851 vertices (57843 vertices using heuristic approximate peel neighborhoods under the same soft threshold), while the MNIST training data has 60000 points. 25408 points in the entire graph reach the distance threshold. 55082 peel neighborhoods contain only one digit; 4364 contain two digits; 485 contain three digits; 60 contain four digits; eight contain five digits; and one peel neighborhood contains six digits. 
In this case, exact (thresholded) peel neighborhoods introduce about $1.3\%$ more edges than the heuristic approximations; the graph drawing of the latter is qualitatively similar (not shown here).
}
  \label{fig:MNIST_giant_component}
\end{figure}

\begin{figure}[htbp]
  \centering
  \includegraphics[trim = 0mm 60mm 0mm 60mm, clip, width=\textwidth,keepaspectratio]{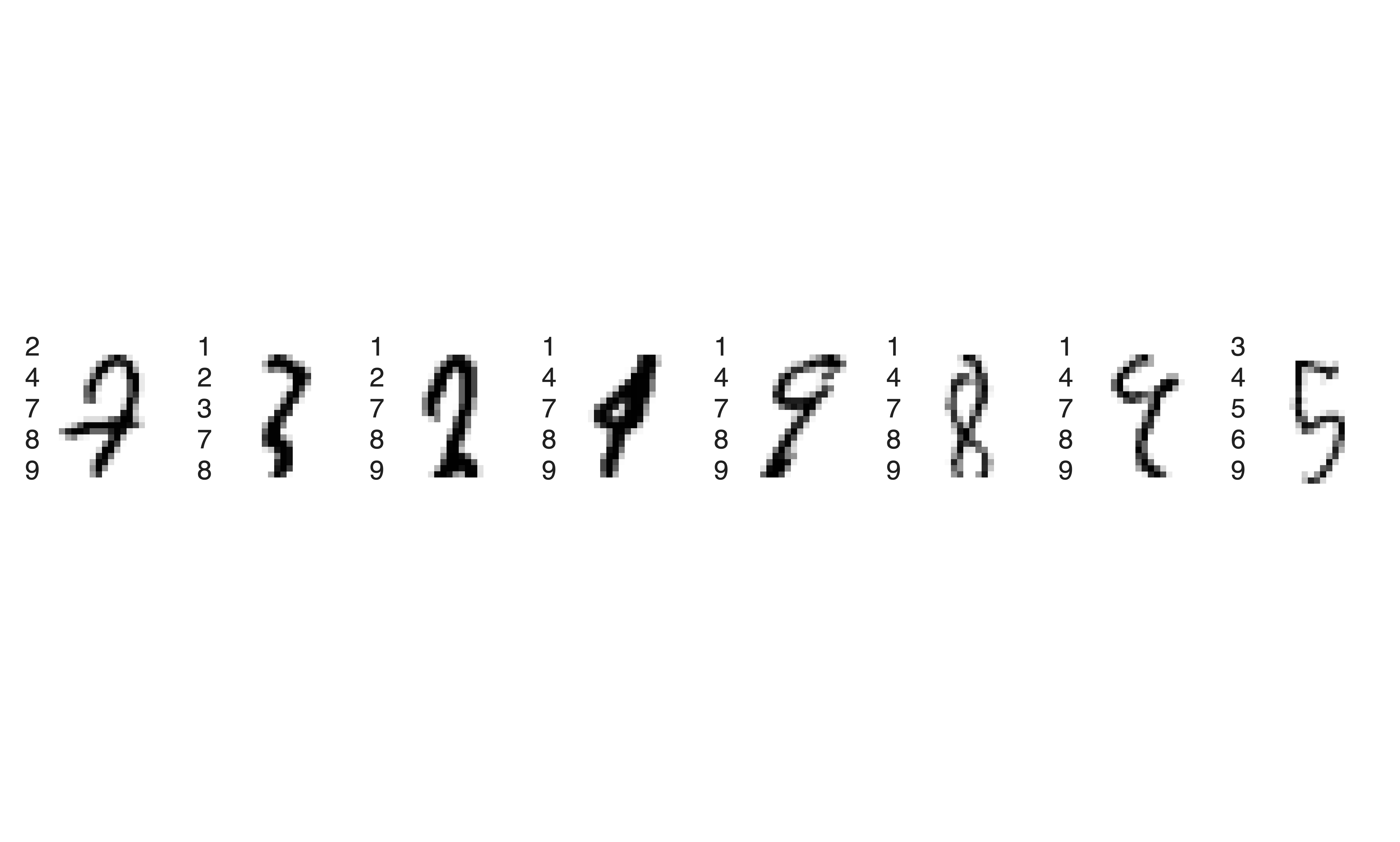}
\caption{The eight images near five classes as described in Figure \ref{fig:MNIST_giant_component}, with class labels indicated.}
  \label{fig:MNIST_near_five_classes}
\end{figure}

\end{document}